\documentclass[10pt]{amsart}

\usepackage[utf8x]{inputenc}
\usepackage[english]{babel}
\usepackage{amsmath,amsfonts,amssymb,amsthm}
\usepackage[T1]{fontenc}
\usepackage{lmodern}

\usepackage[top=3.4cm,bottom=3.4cm,left=3.2cm,right=3.2cm]{geometry}

\usepackage[dvipsnames]{xcolor}
\usepackage[hyperindex=true,frenchlinks=true,colorlinks=true,
citecolor=NavyBlue,linkcolor=Mulberry,urlcolor=LimeGreen,linktocpage]{hyperref}

\usepackage{shuffle}
\usepackage{cite}
\usepackage{subfigure}
\usepackage{tikz}
\usetikzlibrary{calc}

\author{Hayat Cheballah}
\address{GREYC CNRS UMR 6072, Boulevard Mar\'echal Juin,
        F-14032 Caen Cedex, France}
\email{hayat.cheballah@unicaen.fr}

\author{Samuele Giraudo}
\thanks{Phone number and email address of the corresponding author:
+33160957558, {\tt samuele.giraudo@univ-mlv.fr}}
\address{Laboratoire d'Informatique Gaspard-Monge, Universit\'e Paris-Est
        Marne-la-Vall\'ee, 5 Boulevard Descartes, Champs-sur-Marne,
        77454 Marne-la-Vall\'ee cedex 2, France}
\email{samuele.giraudo@univ-mlv.fr}

\author{R\'emi Maurice}
\address{Laboratoire d'Informatique Gaspard-Monge, Universit\'e Paris-Est
        Marne-la-Vall\'ee, 5 Boulevard Descartes, Champs-sur-Marne,
        77454 Marne-la-Vall\'ee cedex 2, France}
\email{remi.maurice@univ-mlv.fr}

\title[Hopf algebra on packed square matrices]
{Hopf algebra structure \\ on packed square matrices}
\date{\today}

\newtheorem{Theoreme}{Theorem}[section]
\newtheorem{Proposition}[Theoreme]{Proposition}
\newtheorem{Lemme}[Theoreme]{Lemma}
\newtheorem{Corollaire}[Theoreme]{Corollary}

\numberwithin{equation}{subsection}
\renewcommand{\leq}{\leqslant}
\renewcommand{\geq}{\geqslant}

\newcommand{\K}{\mathbb{K}}

\newcommand{\CalH}{\mathcal{H}}
\newcommand{\CalG}{\mathcal{G}}
\newcommand{\CalT}{\mathcal{T}}

\newcommand{\Compr}{\operatorname{cp}}

\newcommand{\Over}{\diagup}
\newcommand{\Under}{\diagdown}
\newcommand{\Gauche}{\prec}
\newcommand{\Droite}{\succ}
\newcommand{\DeltaG}{\Delta_\Gauche}
\newcommand{\DeltaD}{\Delta_\Droite}
\newcommand{\DeltaB}{{\bar \Delta}}
\newcommand{\LastC}{\operatorname{last_c}}
\newcommand{\LastR}{\operatorname{last_r}}
\newcommand{\ie}{\emph{i.e.}}
\newcommand{\Un}{1}
\newcommand{\Zero}{0}
\newcommand{\ZeroB}{\textcolor{Bleu}{\Zero}}
\newcommand{\ZeroR}{\textcolor{Rouge}{\Zero}}
\newcommand{\UnB}{\textcolor{Bleu}{1}}
\newcommand{\UnR}{\textcolor{Rouge}{1}}
\newcommand{\identity}{I}
\newcommand{\Plus}{{\bf \text{+}}}
\newcommand{\PlusB}{\textcolor{Bleu}{\Plus}}
\newcommand{\PlusR}{\textcolor{Rouge}{\Plus}}
\newcommand{\Moins}{\raisebox{1.5pt}{\rule{4.2pt}{.5pt}}}
\newcommand{\MoinsB}{\textcolor{Bleu}{\Moins}}

\newcommand{\RelT}{\rightharpoonup}
\newcommand{\OrdMT}{\leq_{\tt M}}
\newcommand{\OrdMQ}{\leq_{\tt MQ}}

\newcommand{\lshuffle}{*}
\newcommand{\DecompC}{\circ}
\newcommand{\Nr}{\operatorname{N_r}}
\newcommand{\Nc}{\operatorname{N_c}}
\newcommand{\Vect}{\operatorname{Vect}}
\newcommand{\La}{{\tt a}}
\newcommand{\Lb}{{\tt b}}
\newcommand{\Lc}{{\tt c}}
\newcommand{\Ld}{{\tt d}}
\newcommand{\ashuffle}{\hspace{.1em}\underline{\hspace{-.1em}\shuffle\hspace{-.1em}}\hspace{.1em}}
\newcommand{\NE}{\mathrm{ne}}
\newcommand{\SE}{\mathrm{se}}
\newcommand{\SW}{\mathrm{sw}}
\newcommand{\NW}{\mathrm{nw}}
\newcommand{\OI}{\mathrm{oi}}
\newcommand{\IO}{\mathrm{io}}
\newcommand{\Eval}{\operatorname{ev}}

\newcommand{\Equiv}[1]{{\equiv_{\operatorname{#1}}}}

\newcommand{\Adj}[1]{{\,\longleftrightarrow_{\operatorname{#1}}\,}}
\newcommand{\ideal}[1]{I_{#1}}
\newcommand{\ZZ}{\mathfrak{Z}}
\newcommand{\NN}{\mathfrak{N}}
\newcommand{\nw}{\CVert{$\bullet$}}
\newcommand{\io}{\CBleu{\tiny$\blacksquare$}}

\newcommand{\SetPart}[1]{\ensuremath{\mathcal{#1}}}
\newcommand{\EnsPermu}{\mathfrak{S}}
\newcommand{\EnsNat}{\mathbb{N}}
\newcommand{\EnsMT}{\mathcal{P}}

\newcommand{\MT}[1]{{{\bf PM}_{#1}}}
\newcommand{\MTN}[1]{{{\bf PMN}_{#1}}}
\newcommand{\MTL}[1]{{{\bf PML}_{#1}}}
\newcommand{\FQSym}{{\bf FQSym}}
\newcommand{\WQSym}{{\bf WQSym}}
\newcommand{\PQSym}{{\bf PQSym}}
\newcommand{\MQSym}{{\bf MQSym}}
\newcommand{\PBT}{{\bf PBT}}
\newcommand{\FSym}{{\bf FSym}}

\newcommand{\UBP}{{\bf UBP}}
\newcommand{\ASM}{{\bf ASM}}
\newcommand{\QASM}[1]{{{\bf ASM}/_{#1}}}

\newcommand{\FF}{{\bf F}}
\newcommand{\EE}{{\bf E}}
\newcommand{\HH}{{\bf H}}
\newcommand{\MS}{{\bf MS}}
\newcommand{\PP}{{\bf P}}

\newcommand{\VV}{{\bf V}}
\newcommand{\WW}{{\bf W}}

\makeatletter
\newenvironment{Matrice}{%
\left[ \vcenter \bgroup \vspace{1.65pt}
\Let@ \restore@math@cr \default@tag
\baselineskip6\ex@ \lineskip4\ex@ \lineskiplimit\lineskip
\halign \bgroup \thinspace \hfil$\m@th\scriptstyle##$\hfil&&
\thinspace\thinspace\hfil$\m@th\scriptstyle##$\hfil\crcr}%
{\crcr \egroup \vspace{1.65pt} \egroup \thinspace \right]}
\makeatother

\definecolor{Noir}{RGB}{0,0,0}
\definecolor{Rouge}{RGB}{205,35,38}
\definecolor{Bleu}{RGB}{2,60,195}
\definecolor{Vert}{rgb}{0, 0.5, 0}

\newcommand{\CBleu}[1]{\textcolor{Bleu}{#1}}
\newcommand{\CRouge}[1]{\textcolor{Rouge}{#1}}
\newcommand{\CVert}[1]{\textcolor{Vert}{#1}}
\newcommand{\Sloane}[1]{\href{http://oeis.org/#1}{{\bf #1}}}

\tikzstyle{Arete} = [Rouge!80,thick,draw,line width=2pt]
\tikzstyle{Injection} = [Noir!100,draw,>->]
\tikzstyle{Surjection} = [Noir!100,draw,->>]
\tikzstyle{seta}=[thick,Rouge]
\tikzstyle{gril} = [gray,very thin]
\tikzstyle{path} = [rounded corners=0.1cm,ultra thick, color=blue]
\pgfkeysdef{/orientations}{%
    \def \bi{(-.1,.1)--(0,-.1)--(.1,.1)};%
    \def \ci{(-.1,-.1)--(0,.1)--(.1,-.1)};%
    \def \di{(-.1,-.1)--(.1,0)--(-.1,.1)};%
    \def \ei{(.1,-.1)--(-.1,0)--(.1,.1)};%
}
\pgfkeysdefnargs{/conditions}{1}{%
    \pgfmathsetmacro\x{#1-.5}
    \pgfmathsetmacro\y{#1-1}
    \foreach \i in {0, ..., \y}{
        \draw[seta,shift={(\i,-.5)}] \bi; 
        \draw[seta,shift={(\i,\x)}] \ci; 
        \draw[seta,shift={(-.5,\i)}] \di; 
        \draw[seta,shift={(\x,\i)}] \ei; 
    }%
}
\pgfkeysdef{/confNE}{%
    \pgfkeys{/orientations}%
    \draw[gril] (-.5,-.5) grid (0.5,0.5);%
    \draw[seta,shift={(0,-.5)}] \ci; %
    \draw[seta,shift={(0,0.5)}] \ci; %
    \draw[seta,shift={(-.5,0)}] \di; %
    \draw[seta,shift={(0.5,0)}] \di; %
}
\pgfkeysdef{/confSW}{%
    \pgfkeys{/orientations}%
    \draw[gril] (-.5,-.5) grid (0.5,0.5);%
    \draw[seta,shift={(0,-.5)}] \bi; %
    \draw[seta,shift={(0,0.5)}] \bi; %
    \draw[seta,shift={(-.5,0)}] \ei; %
    \draw[seta,shift={(0.5,0)}] \ei; %
}
\pgfkeysdef{/confSE}{%
    \pgfkeys{/orientations}%
    \draw[gril] (-.5,-.5) grid (0.5,0.5);%
    \draw[seta,shift={(0,-.5)}] \bi; %
    \draw[seta,shift={(0,0.5)}] \bi; %
    \draw[seta,shift={(-.5,0)}] \di; %
    \draw[seta,shift={(0.5,0)}] \di; %
}
\pgfkeysdef{/confNW}{%
    \pgfkeys{/orientations}%
    \draw[gril] (-.5,-.5) grid (0.5,0.5);%
    \draw[seta,shift={(0,-.5)}] \ci; %
    \draw[seta,shift={(0,0.5)}] \ci; %
    \draw[seta,shift={(-.5,0)}] \ei; %
    \draw[seta,shift={(0.5,0)}] \ei; %
}
\pgfkeysdef{/confOI}{%
    \pgfkeys{/orientations}%
    \draw[gril] (-.5,-.5) grid (0.5,0.5);%
    \draw[seta,shift={(0,-.5)}] \bi; %
    \draw[seta,shift={(0,0.5)}] \ci; %
    \draw[seta,shift={(-.5,0)}] \di; %
    \draw[seta,shift={(0.5,0)}] \ei; %
}
\pgfkeysdef{/confIO}{%
    \pgfkeys{/orientations}%
    \draw[gril] (-.5,-.5) grid (0.5,0.5);%
    \draw[seta,shift={(0,-.5)}] \ci; %
    \draw[seta,shift={(0,0.5)}] \bi; %
    \draw[seta,shift={(-.5,0)}] \ei; %
    \draw[seta,shift={(0.5,0)}] \di; %
}

\begin{document}

\maketitle

\begin{abstract}
    We construct a new bigraded Hopf algebra whose bases are indexed by
    square matrices with entries in the alphabet~$\{\Zero, 1, \dots, k\}$,
    $k \geq 1$, without null rows or columns. This Hopf algebra generalizes
    the one of permutations of Malvenuto and Reutenauer, the one of
    $k$-colored permutations of Novelli and Thibon, and the one of uniform
    block permutations of Aguiar and Orellana. We study the algebraic structure
    of our Hopf algebra and show, by exhibiting multiplicative bases, that
    it is free. We moreover show that it is self-dual and admits a bidendriform
    bialgebra structure. Besides, as a Hopf subalgebra, we obtain a new
    one indexed by alternating sign matrices. We study some of its properties
    and algebraic quotients defined through alternating sign matrices
    statistics.
\end{abstract}

\tableofcontents

\section*{Introduction}

The combinatorial class of permutations is naturally endowed with two
operations. One of them, called {\em shifted shuffle product}, takes
two permutations as input and put these together by blending their letters.
The other one, called {\em deconcatenation coproduct}, takes one permutation
as input and takes it apart by cutting it into prefixes and suffixes. These
two operations satisfy certain compatibility relations, resulting in that
the vector space spanned by the set of permutations forms a Hopf
algebra~\cite{MR95}, namely the {\em Malvenuto-Reutenauer Hopf algebra},
also known as~$\FQSym$~\cite{DHT02}.
\smallskip

This Hopf algebra plays a central role in algebraic combinatorics for at
least two reasons. On the one hand,~$\FQSym$ contains, as Hopf subalgebras,
several structures based on well-known combinatorial objects as {\em e.g.,}
standard Young tableaux~\cite{DHT02}, binary trees~\cite{HNT05}, and integer
compositions~\cite{GKLLRT94}. The construction of these substructures
revisits many algorithms coming from computer science and combinatorics.
Indeed, the insertion of a letter into a Young tableau (following
Robinson-Schensted~\cite{Sch61}) or in a binary search tree~\cite{Knu98}
are algorithms which prove to be as enlightening as surprising in this
algebraic context~\cite{DHT02,HNT02,HNT05}. On the other hand, the polynomial
realization of~$\FQSym$ allows to associate a polynomial with any
permutation~\cite{DHT02} providing a generalization of symmetric functions,
the {\em free quasi-symmetric functions}. This generalization offers
alternative ways to prove several properties of (quasi)symmetric functions.
\smallskip

It is thus natural to enrich this theory by proposing generalizations
of~$\FQSym$. In the last years, several generalizations were proposed and
each of these depends on the way we regard permutations. By regarding a
permutation as a word and allowing repetitions of letters, Hivert introduced
in~\cite{Hiv99} (see~\cite{NT06} for a detailed study) a Hopf algebra~$\WQSym$
on packed words. Additionally, by allowing some jumps for the values of
the letters of permutations, Novelli and Thibon defined in~\cite{NT07}
another Hopf algebra~$\PQSym$ which involves parking functions. These
authors also showed in~\cite{NT10} that the $k$-colored permutations
admit a Hopf algebra structure~$\FQSym^{(k)}$. Furthermore, by regarding a
permutation~$\sigma$ as a bijection associating the singleton~$\{\sigma(i)\}$
with any singleton~$\{i\}$, Aguiar and Orellana constructed~\cite{AO08}
a Hopf algebra structure~$\UBP$ on uniform block permutations, \ie,
bijections between set partitions of~$[n]$, where each part has the same
cardinality as its image. Finally, by regarding a permutation within its
permutation matrix, Duchamp, Hivert and Thibon introduced in~\cite{DHT02}
a Hopf algebra~$\MQSym$ which involves some kind of integer matrices.
\smallskip

In this paper we propose a new generalization of~$\FQSym$ by regarding
permutations as permutation matrices. For this purpose, we consider the
set of {\em $1$-packed matrices} that are square matrices with entries in
the alphabet~$\{\Zero, 1\}$ which have at least one~$1$ by row and by column.
By equipping these matrices with a product and a coproduct, we obtain a
bigraded Hopf algebra, denoted by~$\MT{1}$. By only considering the
gradation offered by the size (resp. the number of nonzero entries) of
matrices, we obtain a simply graded Hopf algebra denoted by $\MTN{1}$
(resp. $\MTL{1}$). Note that since permutation matrices form a Hopf
subalgebra of~$\MTN{1}$ (and $\MTL{1}$) isomorphic to~$\FQSym$, $\MTN{1}$
(and $\MTL{1}$) provides a generalization of~$\FQSym$. Now, by allowing
the entries different from~$\Zero$ of a packed matrix to belong to the
alphabet~$\{1, \dots, k\}$ where~$k$ is a positive integer, we obtain the
notion of a {\em $k$-packed matrix}. The definition of~$\MT{1}$ (and $\MTN{1}$
and $\MTL{1}$) obviously extends to these matrices and leads to the Hopf
algebra~$\MT{k}$ (and $\MTN{k}$ and $\MTL{k}$) involving $k$-packed
matrices. Besides, since any $k$-packed matrix is also a $k + 1$-packed
matrix, $(\MT{k})_{k \geq 1}$ is an increasing infinite sequence of Hopf
algebras for inclusion.
\smallskip

Our results are presented as follows. We give in
Section~\ref{sec:Matrices_tassees} some elementary definitions about
$k$-packed matrices, enumerate them according to their size, and then
define the Hopf algebra of $k$-packed matrices by describing its product
and its coproduct. Section~\ref{sec:Proprietes_alg} is devoted to the study
of the algebraic properties of~$\MT{k}$. In order to show that~$\MT{k}$
is free as an algebra, we define, by introducing a partial order relation
on the $k$-packed matrices, two multiplicative bases: the bases of the
{\em elementary} and {\em homogeneous} elements. We then describe the dual
Hopf algebra~$\MT{k}^\star$ of~$\MT{k}$ in explaining the product and the
coproduct and show that~$\MT{k}$ is self-dual. In Section~\ref{sec:Liens_AHC},
we show how several well-known Hopf algebras are linked with~$\MT{k}$.
In particular, we show that the Hopf algebra of the $k$-colored
permutations~$\FQSym^{(k)}$ embeds into~$\MTN{k}$ (and $\MTL{k}$) and that
the dual~$\UBP^\star$ of the Hopf algebra of uniform block permutations
embeds into~$\MTN{1}$. We also exhibit an injective algebra morphism from
$\MTL{1}^\star$ to $\MQSym$. We conclude this section by providing a method
to construct Hopf subalgebras of $\MT{k}$, analogous to the construction
of Hopf subalgebras of $\FQSym$ by {\em good congruences}~\cite{HN07,Gir11}.
The analogs of the sylvester~\cite{HNT02,HNT05}, plactic~\cite{LS81,Lot02},
hypoplactic~\cite{KT97,KT99}, Bell~\cite{Rey07}, and Baxter~\cite{Gir12}
congruences are still good congruences in our context and give rise to
Hopf subalgebras of $\MT{k}$. We end this article by Section~\ref{sec:ASM}
where we show that~$\MTN{1}$ contains a Hopf subalgebra  whose bases are
indexed by alternating sign matrices, denoted by~$\ASM$. We consider then
some well-known statistics on the six-vertex model with domain wall
boundary conditions~\cite{KO82}, that are combinatorial objects in bijection
with alternating sign matrices~\cite{Kup96,bressoud99}. We study these
statistics from the algebraic point of view offered by the Hopf algebra~$\ASM$.
This section is concluded with a complete study of quotients of~$\ASM$ by
equivalence relations defined through these statistics.
\medskip

{\noindent \bf Acknowledgements.}
This work is based on computer exploration and the authors used, for this
purpose, the open-source mathematical software Sage~\cite{Sage} and one
of its extensions, Sage-Combinat~\cite{SageC}. The authors would like
to thank the anonymous referees which, by their suggestions, greatly
improved Sections \ref{subsec:equivalences}, \ref{subsec:ASM_stats},
and \ref{subsec:ASM_stats_alg}.
\medskip

\section{Packed matrices} \label{sec:Matrices_tassees}

\subsection{Definitions}
Let~$k \geq 1$ be an integer. We denote by~$\mathcal{M}_{k, n, \ell}$
the set
\index{set!$\mathcal{M}_{k, n, \ell}$}%
of~$n \times n$ matrices with exactly $\ell$ nonzero entries in the
alphabet~$A_k := \{\Zero, 1, \dots, k\}$
\index{set!$A_k$}%
and by~$\Nr(M)$
\index{operator!$\Nr$}%
(resp. $\Nc(M)$)
\index{operator!$\Nc$}%
the set of the indices of the zero rows (resp. columns)
of~$M \in \mathcal{M}_{k, n, \ell}$. For example, consider the matrix
\begin{equation}
    M := \begin{Matrice}
        0 & 1 & 0 & 0 & 1 & 0 \\
        0 & 0 & 0 & 1 & 0 & 1 \\
        0 & 1 & 0 & 0 & 0 & 0 \\
        0 & 0 & 0 & 1 & 1 & 0 \\
        0 & 0 & 0 & 0 & 0 & 0 \\
        0 & 0 & 0 & 0 & 0 & 1
    \end{Matrice}.
\end{equation}
We have
\begin{equation}
    \Nr(M) = \{5\} \quad \text{and} \quad
    \Nc(M) = \{1, 3\}.
\end{equation}
\medskip

A {\em $k$-packed matrix}
\index{packed matrix}%
$M$ of size $n$ is a matrix in~$\bigcup_{\ell \geq 0} \mathcal{M}_{k, n, \ell}$
in which each row and each column contains at least one entry different
from~$\Zero$, that is to say if the subsets $\Nr(M)$ and $\Nc(M)$ are empty.
\medskip

We shall denote in the sequel by~$\EnsMT_{k, n, \ell}$
\index{set!$\EnsMT_{k, n, \ell}$}%
the set of~$k$-packed matrices of size~$n$ with exactly~$\ell$ nonzero
entries, by~$\EnsMT_{k, n, -}$
\index{set!$\EnsMT_{k, n, -}$}%
the set of all $k$-packed matrices of size~$n$, by~$\EnsMT_{k, -, \ell}$
\index{set!$\EnsMT_{k, -, \ell}$}%
the set of all $k$-packed matrices with exactly~$\ell$ nonzero entries,
and by~$\EnsMT_k$
\index{set!$\EnsMT_k$}%
the set of all~$k$-packed matrices. The $k$-packed matrix of size~$0$ is
denoted by~$\emptyset$. For instance, the seven $1$-packed matrices of
size~$2$ are
\begin{equation}
    \begin{Matrice} 1 & 0 \\ 0 & 1 \end{Matrice}, \qquad
    \begin{Matrice} 0 & 1 \\ 1 & 0 \end{Matrice}, \qquad
    \begin{Matrice} 1 & 1 \\ 1 & 0 \end{Matrice}, \qquad
    \begin{Matrice} 1 & 1 \\ 0 & 1 \end{Matrice}, \qquad
    \begin{Matrice} 1 & 0 \\ 1 & 1 \end{Matrice}, \qquad
    \begin{Matrice} 0 & 1 \\ 1 & 1 \end{Matrice}, \qquad
    \begin{Matrice} 1 & 1 \\ 1 & 1 \end{Matrice}.
\end{equation}
Besides, the ten $1$-packed matrices of~$\EnsMT_{1, -, 3}$ are
\begin{equation}
    \begin{Matrice} 1 & 1 \\ 1 & 0 \end{Matrice}, \:
    \begin{Matrice} 1 & 1 \\ 0 & 1 \end{Matrice}, \:
    \begin{Matrice} 1 & 0 \\ 1 & 1 \end{Matrice}, \:
    \begin{Matrice} 0 & 1 \\ 1 & 1 \end{Matrice}, \:
    \begin{Matrice} 1 & 0 & 0 \\ 0 & 1 & 0 \\ 0 & 0 & 1 \end{Matrice}, \:
    \begin{Matrice} 1 & 0 & 0 \\ 0 & 0 & 1 \\ 0 & 1 & 0 \end{Matrice}, \:
    \begin{Matrice} 0 & 1 & 0 \\ 1 & 0 & 0 \\ 0 & 0 & 1 \end{Matrice}, \:
    \begin{Matrice} 0 & 0 & 1 \\ 1 & 0 & 0 \\ 0 & 1 & 0 \end{Matrice}, \:
    \begin{Matrice} 0 & 1 & 0 \\ 0 & 0 & 1 \\ 1 & 0 & 0 \end{Matrice}, \:
    \begin{Matrice} 0 & 0 & 1 \\ 0 & 1 & 0 \\ 1 & 0 & 0 \end{Matrice}.
\end{equation}
\medskip

Let us now define some operations on packed matrices. We shall denote
by~$Z_n^m$
\index{operator!$Z_n^m$}%
the~$n \times m$ null matrix. Given~$M_1$ and~$M_2$ two $k$-packed
matrices of respective sizes~$n_1$ and~$n_2$, set
\begin{equation}
    \CBleu{M_1} \Over \CRouge{M_2} :=
    \left[\begin{array}{c|c}
     \CBleu{M_1} & Z_{n_1}^{n_2} \\ \hline
     Z_{n_2}^{n_1} & \CRouge{M_2}
    \end{array}\right]
    \qquad \mbox{and} \qquad
    \CBleu{M_1} \Under \CRouge{M_2} :=
    \left[\begin{array}{c|c}
     Z_{n_1}^{n_2} &  \CBleu{M_1} \\ \hline
     \CRouge{M_2} & Z_{n_2}^{n_1}
    \end{array}\right].
\end{equation}
\index{operator!$\Over$}%
\index{operator!$\Under$}%
Note that these two matrices are $k$-packed matrices of size~$n_1 + n_2$.
We shall respectively call~$\Over$ and~$\Under$ the {\em over}
\index{over}%
and {\em under}
\index{under}%
operators. These two operators are obviously associative.
\medskip

Given a matrix~$M$ whose entries are elements of the alphabet~$A_k$, the
{\em compression}
\index{compression}%
of~$M$ is the matrix~$\Compr(M)$
\index{operator!$\Compr$}%
obtained by deleting in~$M$ all null rows and columns. Let~$M$ be a
$k$-packed matrix. The tuple~$(M_1, \dots, M_r)$ is a
{\em column decomposition}
\index{column decomposition}%
of~$M$, and we write~$M = M_1 \bullet \dots \bullet M_r$,
\index{operator!$\bullet$}
if for all $i \in [r]$ the~$\Compr(M_i)$ are square matrices (and not
necessarily column matrices) and
\begin{equation}
    M = \left[\begin{array}{c|c|c} M_1 & \dots & M_r \end{array}\right].
\end{equation}
Similarly, the tuple~$(M_1, \dots, M_r)$ is a {\em row decomposition}
\index{row decomposition}%
of~$M$, and we write~$M = M_1 \DecompC \cdots \DecompC M_r$,
\index{operator!$\circ$}%
if for all $i \in [r]$ the~$\Compr(M_i)$ are square matrices (and not
necessarily row matrices) and
\begin{equation}
    M =
     \left[\begin{array}{c}
    M_1 \\ \hline
    \dots \\ \hline
    M_r \end{array}\right].
\end{equation}
\medskip

For instance, here are a $1$-packed matrix of size~$5$, one of its column
decompositions and one of its row decompositions:
\begin{equation}
    \begin{Matrice}
        \Zero & \Un & \Un & \Zero & \Zero \\
        \Zero & \Zero & \Un & \Zero & \Zero \\
        \Zero & \Zero & \Zero & \Un & \Un \\
        \Un & \Zero & \Zero & \Zero & \Zero \\
        \Zero & \Zero & \Zero & \Un & \Un
    \end{Matrice}
    =
    \begin{Matrice}
        \Zero & \Un & \Un \\
        \Zero & \Zero & \Un \\
        \Zero & \Zero & \Zero \\
        \Un & \Zero & \Zero \\
        \Zero & \Zero & \Zero
    \end{Matrice}
    \bullet
    \begin{Matrice}
        \Zero & \Zero \\
        \Zero & \Zero \\
        \Un & \Un \\
        \Zero & \Zero \\
        \Un & \Un
    \end{Matrice}
    =
    \begin{Matrice}
        \Zero & \Un & \Un & \Zero & \Zero \\
        \Zero & \Zero & \Un & \Zero & \Zero
    \end{Matrice} \\ \DecompC
    \begin{Matrice}
        \Zero & \Zero & \Zero & \Un & \Un \\
        \Un & \Zero & \Zero & \Zero & \Zero \\
        \Zero & \Zero & \Zero & \Un & \Un
    \end{Matrice}.
\end{equation}
\medskip

These two decompositions have the following property.
\begin{Lemme} \label{lem:Decomposition}
    Let~$M$ be a packed square matrix and~$(M_1,M_2)$ be a column
    (resp. row) decomposition of~$M$. Then, there is no integer~$i$
    such that the $i$th rows (resp. columns) of~$M_1$ and~$M_2$ contain
    both a nonzero entry.
\end{Lemme}
\begin{proof}
    We prove here the lemma only when~$(M_1, M_2)$ is a column
    decomposition of~$M$. The case of a row decomposition can be proven
    in an analogous way.
    \smallskip

    Let us denote by~$n$ the size of~$M$ and assume that~$M_1$ (resp. $M_2$)
    has~$n_1$ (resp. $n_2$) columns. The lemma follows from the fact that
    since~$(M_1, M_2)$ is a column decomposition of~$M$, there are~$n_1$
    nonzero rows in~$M_1$, $n_2$ nonzero rows in~$M_2$, and $n = n_1 + n_2$.
\end{proof}
\medskip

Lemma~\ref{lem:Decomposition} provides a sufficient condition to
ensure that a given pair~$(M_1, M_2)$ of matrices cannot be a column
(resp. row) decomposition of a matrix~$M$. Nevertheless, it is not a
necessary condition. Indeed, let
\begin{equation}
    M :=
    \begin{Matrice}
        \Un & \Un & \Zero \\
        \Zero & \Zero & \Un \\
        \Zero & \Zero & \Un
    \end{Matrice}
    \qquad \mbox{and} \qquad
    (M_1, M_2) :=
    \left(
    \begin{Matrice}
        \Un     & \Un   \\
        \Zero   & \Zero \\
        \Zero   & \Zero
    \end{Matrice},
    \begin{Matrice}
        \Zero \\
        \Un \\
        \Un
    \end{Matrice}\right).
\end{equation}
Then, even if there is no nonzero entry on the same row in~$M_1$
and~$M_2$, $(M_1, M_2)$ is not a column decomposition of~$M$.
\medskip

\subsection{Enumeration} \label{subsec:Enum_MT}

Using the sieve principle, we obtain the following enumerative result.
\begin{Proposition} \label{prop:Enum_MT}
    For any~$k \geq 1$, $n \geq 0$, and~$\ell \geq 0$, the
    number~$\# \EnsMT_{k, n, \ell}$ of $k$-packed matrices of size~$n$
    with exactly~$\ell$ nonzero entries is
    \begin{equation} \label{eq::Enum_MT}
        \# \EnsMT_{k, n, \ell} =
        \sum_{0 \leq i, j \leq n} (-1)^{i + j}
        \binom{n}{i}\binom{n}{j}
        \binom{ij}{\ell}k^\ell.
    \end{equation}
\end{Proposition}
\begin{proof}
    For any subsets~$R$ and~$C$ of~$[n]$ let us define the set
    \begin{equation}
        \mathcal{S}(R, C) :=
        \left\{M \in \mathcal{M}_{k, n, \ell} :
        \Nr(M) = R \mbox{ and } \Nc(M) = C\right\}.
    \end{equation}
    Since~$\# \EnsMT_{k, n, \ell} = \# \mathcal{S}(\emptyset, \emptyset)$,
    we shall compute~$\# \mathcal{S}(\emptyset, \emptyset)$ to
    prove~\eqref{eq::Enum_MT}.
    \smallskip

    For that, let us consider the order relation~$\leq$ defined on the
    set of pairs~$(R, C)$ of subsets of~$[n]$ by
    \begin{equation}
        (R_1, C_1) \leq (R_2, C_2)
        \quad \mbox{if and only if} \quad
        R_1 \subseteq R_2
        \mbox{ and } C_1 \subseteq C_2.
    \end{equation}
    We have, by setting~$r := \# R$ and~$c := \# C$,
    \begin{equation} \label{eq:Nombre_Mat_Sup}
        \sum_{(R, C) \leq (R', C')} \#\mathcal{S}(R', C') =
        \binom{(n - r)\,(n - c)}{\ell} k^\ell
    \end{equation}
    since~\eqref{eq:Nombre_Mat_Sup} is the number of
    matrices~$M \in\mathcal{M}_{k, n, \ell}$ such that~$R \subseteq \Nr(M)$
    and~$C \subseteq \Nc(M)$. Then, by Möbius inversion on the Boolean
    lattice, we obtain
    \begin{equation}
        \#\mathcal{S}(\emptyset, \emptyset) =
        \sum_{(\emptyset,\emptyset)\leq (R, C)}
        (-1)^{r + c}
        \binom{(n - r)\,(n - c)}{\ell} k^\ell,
    \end{equation}
    and~\eqref{eq::Enum_MT} follows.
\end{proof}
\medskip

Table~\ref{tab:Nb_Mat_nl} shows the first few values of~$\# \EnsMT_{k, n, \ell}$.
The enumeration in the case~$k = 1$ is Sequence~\Sloane{A055599}
of~\cite{Slo}.
\begin{table}[h]
    \centering
    \subtable[Number of $1$-packed matrices.]{
    \begin{tabular}{l|llllllllll}
        & 0 & 1 & 2 & 3 & 4 & 5 & 6 & 7 & 8 & 9 \\ \hline
        0 & 1 \\
        1 &   & 1 \\
        2 &   &  & 2 & 4 & 1 \\
        3 &   &  &   & 6 & 45 & 90 & 78 & 36 & 9 & 1
    \end{tabular}}
    \bigskip

    \subtable[Number of $2$-packed matrices.]{
    \begin{tabular}{l|llllllllll}
        & 0 & 1 & 2 & 3 & 4 & 5 & 6 & 7 & 8 & 9 \\ \hline
        0 & 1 \\
        1 &   & 2  \\
        2 &   &   & 8 & 32 & 16 \\
        3 &   &   &   & 48 & 720 & 2880 & 4992 & 4608 & 2304 & 512 \\
    \end{tabular}}
    \bigskip
    \caption{The number of $k$-packed matrices of size~$n$ (vertical
    values) with exactly~$\ell$ nonzero entries (horizontal values).}
    \label{tab:Nb_Mat_nl}
\end{table}
\medskip

Notice that for any~$n \geq 0$, since
\begin{equation}
    \EnsMT_{k, n, -} =
    \biguplus_{n \leq \ell \leq n^2} \EnsMT_{k, n, \ell},
\end{equation}
the set~$\EnsMT_{k, n, -}$ is finite. Hence, by using
Proposition~\ref{prop:Enum_MT}, we obtain
\begin{equation}
    \# \EnsMT_{k, n, -} = \sum_{0 \leq i, j \leq n} (-1)^{i + j}
    \binom{n}{i}\binom{n}{j} (k + 1)^{ij}.
\end{equation}
Sequences~$\left(\# \EnsMT_{1, n, -}\right)_{n \geq 0}$
and~$\left(\# \EnsMT_{2, n, -}\right)_{n \geq 0}$ respectively start with
\begin{equation} \label{equ:Dim_MTN1}
    1, \; 1, \; 7, \; 265, \; 41503, \; 24997921, \; 57366997447,
    \qquad \mbox{\cite[\Sloane{A048291}]{Slo}}
\end{equation}
and
\begin{equation} \label{equ:Dim_MTN2}
    1, \; 2, \; 56, \; 16064, \; 39156608, \; 813732073472, \;
    147662286695991296.
\end{equation}
\medskip

Similarly, since for any~$\ell \geq 0$,
\begin{equation}
    \EnsMT_{k, -, \ell} =
    \biguplus_{\left\lceil\sqrt{\ell}\right\rceil \leq n \leq \ell}
    \EnsMT_{k, n, \ell},
\end{equation}
the set~$\EnsMT_{k, -, \ell}$ is finite. Hence, by using
Proposition~\ref{prop:Enum_MT}, we obtain
\begin{equation}
    \# \EnsMT_{k, -, \ell} = \sum_{0 \leq i, j \leq n \leq \ell}
    (-1)^{i + j}
    \binom{n}{i}\binom{n}{j}
    \binom{ij}{\ell}k^\ell.
\end{equation}
Sequences~$\left(\# \EnsMT_{1, -, \ell}\right)_{\ell \geq 0}$
and~$\left(\# \EnsMT_{2, -, \ell}\right)_{\ell \geq 0}$  respectively
start with
\begin{equation} \label{equ:Dim_MTL1}
    1, \; 1, \; 2, \; 10, \; 70, \; 642, \; 7246, \; 97052, \; 1503700,
    \qquad \mbox{\cite[\Sloane{A104602}]{Slo}}
\end{equation}
and
\begin{equation} \label{equ:Dim_MTL2}
    1, \; 2, \; 8, \; 80, \; 1120, \; 20544, \; 463744, \;
    12422656, \; 384947200.
\end{equation}
\medskip

\subsection{Hopf algebra structure}
In the sequel, all the algebraic structures have a field~$\K$ of
characteristic zero as ground field.
\medskip

Let for any~$k \geq 1$
\begin{equation}
    \MT{k} := \bigoplus_{n \geq 0} \; \bigoplus_{\ell \geq 0} \;
    \Vect\left(\EnsMT_{k, n, \ell}\right)
\end{equation}
\index{Hopf algebra!$\MT{k}$}%
be the bigraded vector space spanned by the set of all $k$-packed
matrices. The elements~$\FF_M$, where the~$M$ are $k$-packed matrices,
form a basis of~$\MT{k}$. We shall call this basis the
{\em fundamental basis}
\index{fundamental basis}%
of~$\MT{k}$.
\medskip

Given~$M_1$ and~$M_2$ two $k$-packed matrices of respective sizes~$n_1$
and~$n_2$, set
\begin{equation}
    M_1 \circ n_2 :=
    \left[\begin{array}{c}
        \textcolor{Bleu}{M_1} \\ \hline
        Z_{n_2}^{n_1}
    \end{array}\right]
    \qquad \mbox{and} \qquad
    n_1 \circ M_2 :=
    \left[\begin{array}{c}
        Z_{n_1}^{n_2} \\ \hline
        \textcolor{Rouge}{M_2}
    \end{array}\right].
\end{equation}
The {\em column shifted shuffle}
\index{column shifted shuffle}
$M_1 \cshuffle M_2$
\index{operator!$\cshuffle$}%
of~$M_1$ and~$M_2$ is the set of all matrices obtained by shuffling the
columns of~$M_1 \circ n_2$ with the columns of~$n_1 \circ M_2$.
\index{operator!$\circ$}%
\medskip

Let us endow~$\MT{k}$ with a product~$\cdot$ linearly defined, for any
$k$-packed matrices~$M_1$ and~$M_2$, by
\begin{equation} \label{equ:Def_Produit}
    \FF_{M_1} \cdot \FF_{M_2} :=
    \sum_{M \: \in \: M_1 \cshuffle M_2} \FF_M.
\end{equation}
For instance, in~$\MT{1}$ one has
\begin{equation}\begin{split}
    \FF_{\begin{Matrice}
        \ZeroB & \UnB \\
        \UnB & \UnB
    \end{Matrice}}
    \cdot
    \FF_{\begin{Matrice}
        \UnR & \ZeroR \\
        \ZeroR & \UnR
    \end{Matrice}}
    & =
    \FF_{\begin{Matrice}
        \ZeroB & \UnB & \Zero & \Zero \\
        \UnB & \UnB & \Zero & \Zero \\
        \Zero & \Zero & \UnR & \ZeroR \\
        \Zero & \Zero & \ZeroR & \UnR
    \end{Matrice}}
    +
    \FF_{\begin{Matrice}
        \ZeroB & \Zero & \UnB & \Zero \\
        \UnB & \Zero & \UnB & \Zero \\
        \Zero & \UnR & \Zero & \ZeroR \\
        \Zero & \ZeroR & \Zero & \UnR
    \end{Matrice}}
    +
    \FF_{\begin{Matrice}
        \ZeroB & \Zero & \Zero & \UnB \\
        \UnB & \Zero & \Zero & \UnB \\
        \Zero & \UnR & \ZeroR & \Zero \\
        \Zero & \ZeroR & \UnR & \Zero
    \end{Matrice}} \\[1em]
    & +
    \FF_{\begin{Matrice}
        \Zero & \ZeroB & \UnB & \Zero \\
        \Zero & \UnB & \UnB & \Zero \\
        \UnR & \Zero & \Zero & \ZeroR \\
        \ZeroR & \Zero & \Zero & \UnR
    \end{Matrice}}
    +
    \FF_{\begin{Matrice}
        \Zero & \ZeroB & \Zero & \UnB \\
        \Zero & \UnB & \Zero & \UnB \\
        \UnR & \Zero & \ZeroR & \Zero \\
        \ZeroR & \Zero & \UnR & \Zero
    \end{Matrice}}
    +
    \FF_{\begin{Matrice}
        \Zero & \Zero & \ZeroB & \UnB \\
        \Zero & \Zero & \UnB & \UnB \\
        \UnR & \ZeroR & \Zero & \Zero \\
        \ZeroR & \UnR & \Zero & \Zero
    \end{Matrice}}.
\end{split}\end{equation}
\medskip

Moreover, we endow~$\MT{k}$ with a coproduct~$\Delta$ linearly defined,
for any \mbox{$k$-packed} matrix~$M$, by
\begin{equation} \label{equ:Def_Coproduit}
    \Delta\left(\FF_M\right) :=
    \sum_{M = M_1 \bullet M_2} \FF_{\Compr(M_1)} \otimes \FF_{\Compr(M_2)}.
\end{equation}
For instance, in~$\MT{1}$ one has
\begin{equation}
    \Delta
    \FF_{\begin{Matrice}
        \Un & \Un & \Zero & \Zero \\
        \Zero & \Zero & \Zero & \Un \\
        \Un & \Zero & \Un & \Zero \\
        \Zero & \Un & \Zero & \Zero
    \end{Matrice}}
    =
    \FF_{\begin{Matrice}
        \Un & \Un & \Zero & \Zero \\
        \Zero & \Zero & \Zero & \Un \\
        \Un & \Zero & \Un & \Zero \\
        \Zero & \Un & \Zero & \Zero
    \end{Matrice}}
    \otimes
    \FF_{\emptyset}
    +
    \FF_{\begin{Matrice}
        \Un & \Un & \Zero \\
        \Un & \Zero & \Un \\
        \Zero & \Un & \Zero
    \end{Matrice}}
    \otimes
    \FF_{\begin{Matrice}
        \Un
    \end{Matrice}}
    +
    \FF_{\emptyset}
    \otimes
    \FF_{\begin{Matrice}
        \Un & \Un & \Zero & \Zero \\
        \Zero & \Zero & \Zero & \Un \\
        \Un & \Zero & \Un & \Zero \\
        \Zero & \Un & \Zero & \Zero
    \end{Matrice}}.
\end{equation}
\medskip

Note that by definition, the product and the coproduct of~$\MT{k}$ are
multiplicity free.
\medskip

\begin{Theoreme} \label{thm:MT_AHC}
    The vector space~$\MT{k}$ endowed with the product~$\cdot$ and the
    coproduct~$\Delta$ is a bigraded and connected bialgebra where
    homogeneous components are finite-dimensional.
\end{Theoreme}
\begin{proof}
    First, it is plain that the product of~$\MT{k}$ respects the bigradation.
    Moreover, Lemma~\ref{lem:Decomposition} implies that it is also the
    case for its coproduct. Since $\emptyset$ is the only packed matrix
    of size~$0$ without nonzero entries, $\MT{k}$ is connected. Besides,
    since for all~$n, \ell \geq 0$, the sets~$\EnsMT_{k, n, \ell}$ are
    finite, homogeneous components of~$\MT{k}$ are finite-dimensional.
    \smallskip

    The associativity of~$\cdot$ arises from the associativity of the
    shifted shuffle operation on words on a totally ordered alphabet.
    Indeed, a packed matrix~$M$ can be seen as a word~$u$ where the $i$th
    letter of~$u$ is the $i$th column of~$M$. Moreover, the coassociativity
    of~$\Delta$ comes from the fact that~$(M_1\bullet M_2)\bullet M_3$ is
    a column decomposition of a packed matrix~$M$ if and only
    if~$M_1\bullet (M_2\bullet M_3)$ also is.
    \smallskip

    It remains to show that~$\Delta$ is an algebra morphism. Let~$M_1$
    and~$M_2$ be two packed matrices. The obvious fact that~$(L, R)$ is
    a column decomposition of a matrix~$M$ appearing in the shifted
    shuffle of~$M_1$ and~$M_2$ if and only if~$L$ (resp. $R$) appears in
    the shifted shuffle of~$L_1$ and~$L_2$ (resp. $R_1$ and~$R_2$)
    where~$(L_1, R_1)$ is a column decomposition of~$M_1$ and~$(L_2, R_2)$
    is a column decomposition of~$M_2$, ensures that~$\Delta$ is an
    algebra morphism.
\end{proof}
\medskip

Since~$\MT{k}$ is, by Theorem~\ref{thm:MT_AHC}, a bigraded and connected
bialgebra, it admits an antipode and hence, is a Hopf algebra. The
antipode~$S$ of~$\MT{k}$ satisfies, for any $k$-packed matrix~$M$,
\begin{equation}
    S\left(\FF_M\right) =
    \sum_{\substack{\ell \geq 1 \\ M = M_1 \bullet \dots \bullet M_\ell \\
    M_i \ne \emptyset, \; i \in [\ell]}}
    (-1)^\ell \; \FF_{\Compr(M_1)}
    \cdot \ldots \cdot
    \FF_{\Compr(M_\ell)}.
\end{equation}
For instance, in~$\MT{1}$ one has
\begin{equation}
    \begin{split}
    S \FF_{
    \begin{Matrice}
        \Zero & \Un & \Un \\
        \Un & \Zero & \Zero \\
        \Zero & \Un & \Zero
    \end{Matrice}} & =
    - \FF_{
    \begin{Matrice}
        \Zero & \Un & \Un \\
        \Un & \Zero & \Zero \\
        \Zero & \Un & \Zero
    \end{Matrice}}
    +
    \FF_{
    \begin{Matrice}
        \Un
    \end{Matrice}} \cdot
    \FF_{
    \begin{Matrice}
        \Un & \Un \\
        \Un & \Zero
    \end{Matrice}} \\[.5em] & =
    \FF_{
    \begin{Matrice}
        \Un & \Zero & \Zero \\
        \Zero & \Un & \Un \\
        \Zero & \Un & \Zero
    \end{Matrice}}
    +
    \FF_{
    \begin{Matrice}
        \Zero & \Un & \Zero \\
        \Un & \Zero & \Un \\
        \Un & \Zero & \Zero
    \end{Matrice}}
    +
    \FF_{
    \begin{Matrice}
        \Zero & \Zero & \Un \\
        \Un & \Un & \Zero \\
        \Un & \Zero & \Zero
    \end{Matrice}}
    -
    \FF_{
    \begin{Matrice}
        \Zero & \Un & \Un \\
        \Un & \Zero & \Zero \\
        \Zero & \Un & \Zero
    \end{Matrice}}.
    \end{split}
\end{equation}
\medskip

Note besides that~$S$ is not an involution. Indeed,
\begin{equation}
    \begin{split}
    S^2 \FF_{
    \begin{Matrice}
        \Zero & \Un & \Un \\
        \Un & \Zero & \Zero \\
        \Zero & \Un & \Zero
    \end{Matrice}} & =
    \FF_{
    \begin{Matrice}
        \Un & \Un & \Zero \\
        \Un & \Zero & \Zero \\
        \Zero & \Zero & \Un
    \end{Matrice}}
    +
    \FF_{
    \begin{Matrice}
        \Un & \Zero & \Un \\
        \Un & \Zero & \Zero \\
        \Zero & \Un & \Zero
    \end{Matrice}}
    +
    \FF_{
    \begin{Matrice}
        \Zero & \Un & \Un \\
        \Zero & \Un & \Zero \\
        \Un & \Zero & \Zero
    \end{Matrice}}
    +
    \FF_{
    \begin{Matrice}
        \Zero & \Un & \Un \\
        \Un & \Zero & \Zero \\
        \Zero & \Un & \Zero
    \end{Matrice}}
    \\ & -
    \FF_{
    \begin{Matrice}
        \Un & \Zero & \Zero \\
        \Zero & \Un & \Un \\
        \Zero & \Un & \Zero
    \end{Matrice}}
    -
    \FF_{
    \begin{Matrice}
        \Zero & \Un & \Zero \\
        \Un & \Zero & \Un \\
        \Un & \Zero & \Zero
    \end{Matrice}}
    -
    \FF_{
    \begin{Matrice}
        \Zero & \Zero & \Un \\
        \Un & \Un & \Zero \\
        \Un & \Zero & \Zero
    \end{Matrice}}.
    \end{split}
\end{equation}
\medskip

Notice that since any $k$-packed matrix is also a $k + 1$-packed matrix,
the vector space~$\MT{k}$ is included in~$\MT{k + 1}$. Hence, and by
Theorem~\ref{thm:MT_AHC},
\begin{equation}
    \MT{1} \; \hookrightarrow \; \MT{2} \; \hookrightarrow \; \cdots
\end{equation}
is an increasing infinite sequence of Hopf algebras for inclusion. The
first few dimensions of~$\MT{1}$ and~$\MT{2}$ are given by
Table~\ref{tab:Nb_Mat_nl}.
\medskip

Let us now set
\begin{equation}
    \MTN{k} := \bigoplus_{n \geq 0}
    \Vect\left(\EnsMT_{k, n, -}\right)
    \qquad \mbox{and} \qquad
    \MTL{k} := \bigoplus_{\ell \geq 0}
    \Vect\left(\EnsMT_{k, -, \ell}\right)
\end{equation}
\index{Hopf algebra!$\MTN{k}$}%
\index{Hopf algebra!$\MTL{k}$}%
the vector spaces of $k$-packed matrices respectively graded by the size
and by the number of nonzero entries of matrices. By Theorem~\ref{thm:MT_AHC},
and since each homogeneous component of these vector spaces is
finite-dimensional (see Section~\ref{subsec:Enum_MT}), $\MTN{k}$
and~$\MTL{k}$ are Hopf algebras. Besides,
\begin{equation}
    \MTN{1} \; \hookrightarrow \; \MTN{2} \; \hookrightarrow \; \cdots
    \qquad \mbox{and} \qquad
    \MTL{1} \; \hookrightarrow \; \MTL{2} \; \hookrightarrow \; \cdots
\end{equation}
are increasing infinite sequences of Hopf algebras for inclusion. The
first few dimensions of~$\MTN{1}$ and~$\MTN{2}$ are given by~\eqref{equ:Dim_MTN1}
and~\eqref{equ:Dim_MTN2}, and the first few dimensions of~$\MTL{1}$
and~$\MTL{2}$ are given by~\eqref{equ:Dim_MTL1} and~\eqref{equ:Dim_MTL2}.
In the sequel, we shall denote by~$\CalH_{k,n}(t)$
\index{series!$\CalH_{k,n}(t)$}%
(resp. $\CalH_{k,\ell}(t)$
\index{series!$\CalH_{k,\ell}(t)$}
the Hilbert series of~$\MTN{k}$ (resp. $\MTL{k}$).
\medskip

\section{Algebraic properties} \label{sec:Proprietes_alg}

\subsection{Multiplicative bases and freeness} \label{subsec:Mult_Liberte}

\subsubsection{Poset structure}
We endow the set~$\EnsMT_k$ with a binary relation~$\to$
\index{order relation!$\to$}%
defined in the following way. If~$M_1$ and~$M_2$ are two $k$-packed
matrices of size~$n$, we have~$M_1 \to M_2$ if there is an index
$i \in [n - 1]$ such that, denoting by~$s$ the number of~$\Zero$ ending
the $i$th column of~$M_1$, and by~$p$ the number of~$\Zero$ starting the
$(i + 1)$st column of~$M_1$, one has~$s + p \geq n$ and~$M_2$ is obtained
from~$M_1$ by exchanging its $i$th and $(i + 1)$st columns (see
Figure~\ref{fig:CouvertureOrdreMT}).
\begin{figure}[ht]
    \begin{tikzpicture}[scale=.45]
        \filldraw[draw=black,fill=BrickRed!40] (0,0) rectangle (1,-3);
        \draw[draw=black,fill=white] (0,-3) rectangle
            node{\begin{math}0\end{math}}(1,-7);
        \draw[draw=black,fill=white] (1,0) rectangle
            node{\begin{math}0\end{math}} (2,-5);
        \filldraw[draw=black,fill=RoyalBlue!40] (1,-5) rectangle (2,-7);
        \node at(.5,.4){\begin{math}i\end{math}};
        \node at(1.5,.4){\begin{math}i\!+\!1\end{math}};
        \draw[draw=black](-.25,-3)edge[<->]
            node[anchor=center,left]{\begin{math}s\end{math}}(-.25,-7);
        \draw[draw=black](2.25,0)edge[<->]
            node[anchor=center,right]{\begin{math}p\end{math}}(2.25,-5);
        \draw[draw=black](-1.05,0)edge[<->]
            node[anchor=center,left]{\begin{math}n\end{math}}(-1.05,-7);
        \draw[black!80,draw,->,line width=1.5pt](3,-3.5)--(5,-3.5);
        \filldraw[draw=black,fill=BrickRed!40] (7,0) rectangle (8,-3);
        \draw[draw=black,fill=white] (7,-3) rectangle
            node{\begin{math}0\end{math}}(8,-7);
        \draw[draw=black,fill=white] (6,0) rectangle
            node{\begin{math}0\end{math}} (7,-5);
        \filldraw[draw=black,fill=RoyalBlue!40] (6,-5) rectangle (7,-7);
    \end{tikzpicture}
    \caption{The condition for swapping the $i$th and $(i + 1)$st
    columns of a packed matrix according to the relation~$\to$. The darker
    regions contain any entries and the white ones, only zeros.}
    \label{fig:CouvertureOrdreMT}
\end{figure}
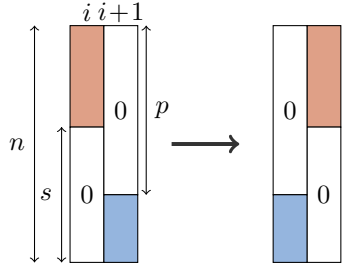
\medskip

We now endow~$\EnsMT_k$ with the partial order relation~$\OrdMT$
\index{order relation!$\OrdMT$}%
defined as the reflexive and transitive closure of~$\to$.
Figure~\ref{fig:OrdreMT} shows an interval of this partial order.
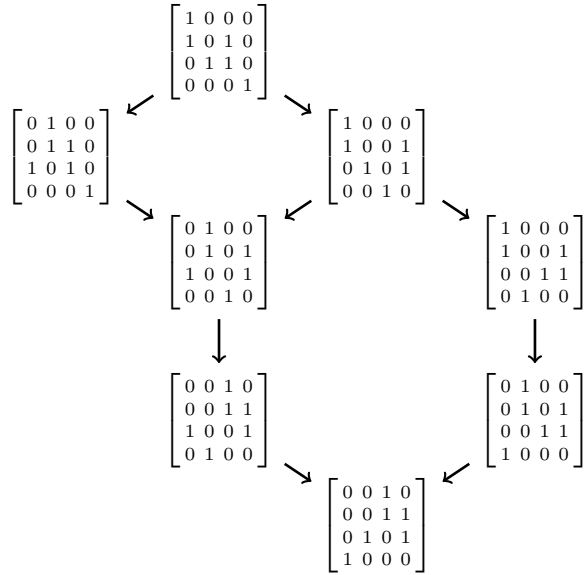
\begin{figure}[ht]
    \scalebox{1}{\begin{tikzpicture}[scale=.7]
        \node(1)at(0,0){$
        \begin{Matrice}
            \Un & \Zero & \Zero & \Zero \\
            \Un & \Zero & \Un & \Zero \\
            \Zero & \Un & \Un & \Zero \\
            \Zero & \Zero & \Zero & \Un
        \end{Matrice}$};
        \node(2)at(-3,-2){$
        \begin{Matrice}
            \Zero & \Un & \Zero & \Zero \\
            \Zero & \Un & \Un & \Zero \\
            \Un & \Zero & \Un & \Zero \\
            \Zero & \Zero & \Zero & \Un
        \end{Matrice}$};
        \node(3)at(3,-2){$
        \begin{Matrice}
            \Un & \Zero & \Zero & \Zero \\
            \Un & \Zero & \Zero & \Un \\
            \Zero & \Un & \Zero & \Un \\
            \Zero & \Zero & \Un & \Zero
        \end{Matrice}$};
        \node(4)at(0,-4){$
        \begin{Matrice}
            \Zero & \Un & \Zero & \Zero \\
            \Zero & \Un & \Zero & \Un \\
            \Un & \Zero & \Zero & \Un \\
            \Zero & \Zero & \Un & \Zero
        \end{Matrice}$};
        \node(5)at(6,-4){$
        \begin{Matrice}
            \Un & \Zero & \Zero & \Zero \\
            \Un & \Zero & \Zero & \Un \\
            \Zero & \Zero & \Un & \Un \\
            \Zero & \Un & \Zero & \Zero
        \end{Matrice}$};
        \node(6)at(0,-7){$
        \begin{Matrice}
            \Zero & \Zero & \Un & \Zero \\
            \Zero & \Zero & \Un & \Un \\
            \Un & \Zero & \Zero & \Un \\
            \Zero & \Un & \Zero & \Zero
        \end{Matrice}$};
        \node(7)at(6,-7){$
        \begin{Matrice}
            \Zero & \Un & \Zero & \Zero \\
            \Zero & \Un & \Zero & \Un \\
            \Zero & \Zero & \Un & \Un \\
            \Un & \Zero & \Zero & \Zero
        \end{Matrice}$};
        \node(8)at(3,-9){$
        \begin{Matrice}
            \Zero & \Zero & \Un & \Zero \\
            \Zero & \Zero & \Un & \Un \\
            \Zero & \Un & \Zero & \Un \\
            \Un & \Zero & \Zero & \Zero
        \end{Matrice}$};
        \draw[line width=1pt,->] (1)--(2);
        \draw[line width=1pt,->] (1)--(3);
        \draw[line width=1pt,->] (2)--(4);
        \draw[line width=1pt,->] (3)--(4);
        \draw[line width=1pt,->] (3)--(5);
        \draw[line width=1pt,->] (4)--(6);
        \draw[line width=1pt,->] (5)--(7);
        \draw[line width=1pt,->] (6)--(8);
        \draw[line width=1pt,->] (7)--(8);
    \end{tikzpicture}}
    \caption{The Hasse diagram of an interval for the order~$\OrdMT$
    on packed matrices.}
    \label{fig:OrdreMT}
\end{figure}
\medskip

Notice that by regarding a permutation~$\sigma$ of~$\EnsPermu_n$ as its
{\em permutation matrix}
\index{permutation matrix}%
(\ie, the $1$-packed matrix~$M$ of size~$n$ satisfying~$M_{ij} = 1$ if
and only if~$\sigma_j = i$), the poset~$(\EnsMT_{k, n, -}, \OrdMT)$
restricted to permutation matrices is the right weak order on
permutations~\cite{GR63}.
\medskip

\begin{Lemme} \label{lem:Melange_Ordre_MT}
    Let~$M$, $A$ and~$B$ be three packed matrices. Then,
    \begin{enumerate}
        \item \label{item:Melange_Ordre_MT_1}
        $A \Over B \OrdMT M$ if and only if there are two packed
        matrices~$A'$ and~$B'$ such that $A \OrdMT A'$, $B \OrdMT B'$,
        and~$M \in A' \cshuffle B'$;
        \smallskip

        \item \label{item:Melange_Ordre_MT_2}
        $M \OrdMT A \Under B$ if and only if there are two packed
        matrices~$A'$ and~$B'$ such that $A' \OrdMT A$, $B' \OrdMT B$,
        and~$M \in A' \cshuffle B'$.
    \end{enumerate}
\end{Lemme}
\begin{proof}
    Assume that~$A \Over B \OrdMT M$. By definition of the order~$\OrdMT$,
    $M$ can be obtained from~$A \Over B$ by swapping columns coming
    from~$A$ to obtain a matrix~$A'$ satisfying~$A \OrdMT A'$, by swapping
    columns coming from~$B$ to obtain a matrix~$B'$ satisfying~$B \OrdMT B'$,
    and then, by swapping columns coming from~$A'$ and from~$B'$ together.
    Thereby, $M \in A' \cshuffle B'$.
    \smallskip

    Conversely assume that~$A \OrdMT A'$, $B \OrdMT B'$,
    and~$M \in A' \cshuffle B'$. Then, by definition of the shifted
    shuffle product and the over operator, $A' \Over B' \OrdMT M$. This
    implies~$A \Over B \OrdMT M$.
    \smallskip

    By very similar arguments, \eqref{item:Melange_Ordre_MT_2} is
    established.
\end{proof}
\medskip

\subsubsection{Multiplicative bases}
By mimicking definitions of the bases of symmetric functions, for any
$k$-packed matrix~$M$, the {\em elementary elements}~$\EE_M$
\index{elementary element}%
and the {\em homogeneous elements}~$\HH_M$
\index{homogeneous element}%
are respectively defined by
\begin{equation}
    \EE_M := \sum_{M \OrdMT M'} \FF_{M'}
    \qquad \mbox{and} \qquad
    \HH_M := \sum_{M' \OrdMT M} \FF_{M'}.
\end{equation}
By triangularity, these two families are bases of~$\MT{k}$. For instance,
in~$\MT{1}$ one has
\begin{equation}
    \EE_{
    \begin{Matrice}
        \Un & \Zero & \Zero & \Zero \\
        \Un & \Zero & \Zero & \Un \\
        \Zero & \Zero & \Un & \Un \\
        \Zero & \Un & \Zero & \Zero
    \end{Matrice}}
    =
    \FF_{
    \begin{Matrice}
        \Un & \Zero & \Zero & \Zero \\
        \Un & \Zero & \Zero & \Un \\
        \Zero & \Zero & \Un & \Un \\
        \Zero & \Un & \Zero & \Zero
    \end{Matrice}}
    +
    \FF_{
    \begin{Matrice}
        \Zero & \Un & \Zero & \Zero \\
        \Zero & \Un & \Zero & \Un \\
        \Zero & \Zero & \Un & \Un \\
        \Un & \Zero & \Zero & \Zero
    \end{Matrice}}
    +
    \FF_{
    \begin{Matrice}
        \Zero & \Zero & \Un & \Zero \\
        \Zero & \Zero & \Un & \Un \\
        \Zero & \Un & \Zero & \Un \\
        \Un & \Zero & \Zero & \Zero
    \end{Matrice}},
\end{equation}
and
\begin{equation}
    \HH_{
    \begin{Matrice}
        \Zero & \Un & \Zero & \Zero \\
        \Zero & \Un & \Zero & \Un \\
        \Un & \Zero & \Zero & \Un \\
        \Zero & \Zero & \Un & \Zero
    \end{Matrice}}
    =
    \FF_{
    \begin{Matrice}
        \Zero & \Un & \Zero & \Zero \\
        \Zero & \Un & \Zero & \Un \\
        \Un & \Zero & \Zero & \Un \\
        \Zero & \Zero & \Un & \Zero
    \end{Matrice}
    }
    +
    \FF_{
    \begin{Matrice}
        \Zero & \Un & \Zero & \Zero \\
        \Zero & \Un & \Un & \Zero \\
        \Un & \Zero & \Un & \Zero \\
        \Zero & \Zero & \Zero & \Un
    \end{Matrice}}
    +
    \FF_{
    \begin{Matrice}
        \Un & \Zero & \Zero & \Zero \\
        \Un & \Zero & \Zero & \Un \\
        \Zero & \Un & \Zero & \Un \\
        \Zero & \Zero & \Un & \Zero
    \end{Matrice}}
    +
    \FF_{
    \begin{Matrice}
        \Un & \Zero & \Zero & \Zero \\
        \Un & \Zero & \Un & \Zero \\
        \Zero & \Un & \Un & \Zero \\
        \Zero & \Zero & \Zero & \Un
    \end{Matrice}}.
\end{equation}
\medskip

\begin{Proposition} \label{prop:Produit_intervalle}
    The elements appearing in a product of~$\MT{k}$ expressed in the
    fundamental basis form an interval for the $\OrdMT$-partial order.
    More precisely, for any $k$-packed matrices~$M_1$ and~$M_2$,
    \begin{equation} \label{eq:Produit_intervalle}
        \FF_{M_1} \cdot \FF_{M_2} =
        \sum_{M_1 \Over M_2 \OrdMT M \OrdMT M_1 \Under M_2} \FF_M.
    \end{equation}
\end{Proposition}
\begin{proof}
    It is plain that the left and right-hand side of~\eqref{eq:Produit_intervalle}
    are multiplicity-free. Then, it is enough to show that the
    sets~$M_1 \cshuffle M_2$ and~$[M_1 \Over M_2, M_1 \Under M_2]$ are
    equal. This is a consequence of Lemma~\ref{lem:Melange_Ordre_MT}.
\end{proof}
\medskip

\begin{Proposition} \label{prop:Bases_mult}
    The product of~$\MT{k}$ satisfies, for any $k$-packed matrices~$M_1$
    and~$M_2$,
    \begin{equation}
        \EE_{M_1} \cdot \EE_{M_2} = \EE_{M_1 \Over M_2}
        \qquad \mbox{and} \qquad
        \HH_{M_1} \cdot \HH_{M_2} = \HH_{M_1 \Under M_2}.
    \end{equation}
\end{Proposition}
\begin{proof}
    We shall prove the product rule for the elementary basis by
    expanding~$\EE_{M_1} \cdot \EE_{M_2}$ and $\EE_{M_1 \Over M_2}$ over
    the fundamental basis. First, since any element~$\FF_N$, where~$N$ is
    a packed matrix, appearing in $\EE_{M_1} \cdot \EE_{M_2}$ is obtained
    by shifting and shuffling two matrices~$N_1$ and~$N_2$ such
    that~$M_1 \OrdMT N_1$ and~$M_2 \OrdMT N_2$, $\EE_{M_1} \cdot \EE_{M_2}$
    is multiplicity-free over the fundamental basis. Moreover, by definition
    of the elementary basis, $\EE_{M_1 \Over M_2}$ is multiplicity-free
    over the fundamental basis.
    \smallskip

    Therefore, it is enough to prove that the sets
    \begin{equation}
        \left\{N \in N_1 \cshuffle N_2 :
        M_1 \OrdMT N_1 \mbox{ and } M_2 \OrdMT N_2\right\}
    \end{equation}
    and
    \begin{equation}
        \left\{N \in \EnsMT_k : M_1 \Over M_2 \OrdMT N\right\}
    \end{equation}
    are equal. This is exactly~\eqref{item:Melange_Ordre_MT_1}
    of Lemma~\ref{lem:Melange_Ordre_MT}.
    \smallskip

    The proof for the homogeneous basis is analogous.
\end{proof}
\medskip

\subsubsection{Freeness} \label{subsubsec:freeness}
Given a $k$-packed matrix~$M \ne \emptyset$, we say that~$M$ is
{\em connected}
\index{connected}%
(resp. {\em anti-connected})
\index{anti-connected}%
if, for all $k$-packed matrices~$M_1$ and~$M_2$, $M = M_1 \Over M_2$
(resp. $M = M_1 \Under M_2$) implies~$M_1 = M$ or~$M_2 = M$.

\begin{Theoreme} \label{thm:MT_Libre}
    The Hopf algebra~$\MT{k}$ is freely generated as an algebra by the
    elements~$\EE_M$ (resp. $\HH_M$) where the~$M$ are connected
    (resp. anti-connected) $k$-packed matrices.
\end{Theoreme}
\begin{proof}
    Since any packed matrix~$M$ can be written as
    \begin{equation} \label{equ:Decomp_Over}
        M = M_1 \Over \dots \Over M_r,
    \end{equation}
    where the~$M_i$ are connected packed matrices, by
    Proposition~\ref{prop:Bases_mult}, we have
    \begin{equation}
        \EE_M = \EE_{M_1} \cdot \ldots \cdot \EE_{M_r},
    \end{equation}
    showing that the~$\EE_M$, where~$M$ is a connected packed matrix,
    generate~$\MT{k}$ as an algebra. Besides, the obvious unicity of the
    factorization~\eqref{equ:Decomp_Over} shows that this family is free.
    \smallskip

    The proof for the homogeneous basis is analogous.
\end{proof}
\medskip

Theorem~\ref{thm:MT_Libre} also implies that~$\MTN{k}$ and~$\MTL{k}$
are freely generated by the~$\EE_M$ (resp. $\HH_M$) where the~$M$ are
connected (resp. anti-connected) $k$-packed matrices. Hence, the
generating series~$\CalG_{k,n}(t)$
\index{series!$\CalG_{k,n}(t)$}%
and~$\CalG_{k,\ell}(t)$
\index{series!$\CalG_{k,\ell}(t)$}%
of algebraic generators of~$\MTN{k}$ and~$\MTL{k}$ satisfy respectively
\begin{equation}
    \CalG_{k,n}(t) = 1 - \frac{1}{\CalH_{k,n}(t)}
    \qquad \mbox{and} \qquad
    \CalG_{k,\ell}(t) = 1 - \frac{1}{\CalH_{k,\ell}(t)}.
\end{equation}
The first few numbers of algebraic generators of~$\MTN{1}$ and~$\MTN{2}$
are respectively
\begin{equation} \label{equ:Dim_Gen_Alg_PMN1}
    0, \; 1, \; 6, \; 252, \; 40944, \; 24912120, \; 57316485000
\end{equation}
and
\begin{equation} \label{equ:Dim_Gen_Alg_PMN2}
    0, \; 2, \; 52, \; 15848, \; 39089872, \; 813573857696, \;
    147659027604370240.
\end{equation}
The first few numbers of algebraic generators of~$\MTL{1}$ and~$\MTL{2}$
are respectively
\begin{equation} \label{equ:Dim_Gen_Alg_PML1}
   0, \; 1, \; 1, \; 7, \; 51, \; 497, \; 5865, \; 81305, \; 1293333
\end{equation}
and
\begin{equation} \label{equ:Dim_Gen_Alg_PML2}
    0, \; 2, \; 4, \; 56, \; 816, \; 15904, \; 375360, \; 10407040, \;
    331093248.
\end{equation}
\medskip

\subsection{Self-duality} \label{subsec:Autodualite}

\subsubsection{Dual Hopf algebra}
Let us denote by~$\MT{k}^\star$
\index{Hopf algebra!$\MT{k}^\star$}%
the bigraded dual vector space of~$\MT{k}$, by~$\FF^\star_M$, where
the~$M$ are $k$-packed matrices, the adjoint basis of the fundamental
basis of~$\MT{k}$, and by~$\langle - , - \rangle$ the associated duality
bracket.
\medskip

Let~$M_1$ and~$M_2$ be two $k$-packed matrices of respective sizes~$n_1$
and~$n_2$. By duality, the product in~$\MT{k}^\star$ satisfies
\begin{equation}
    \FF^\star_{M_1} \cdot \FF^\star_{M_2} =
    \sum_{M \in \EnsMT_k}
    \left\langle \Delta\left(\FF_M\right),
    \FF^\star_{M_1} \otimes \FF^\star_{M_2} \right\rangle
    \: \FF^\star_M.
\end{equation}
Let us set
\begin{equation}
    M_1 \bullet n_2 :=
    \left[\begin{array}{c|c}
        \textcolor{Bleu}{M_1} & Z_{n_1}^{n_2}
    \end{array}\right]
    \qquad \mbox{and} \qquad
    n_1 \bullet M_2 :=
    \left[\begin{array}{c|c}
        Z_{n_2}^{n_1} & \textcolor{Rouge}{M_2}
    \end{array}\right].
\end{equation}
The {\em row shifted shuffle}
\index{row shifted shuffle}%
$M_1 \lshuffle M_2$
\index{operator!$\lshuffle$}%
of~$M_1$ and~$M_2$ is the set of all matrices obtained by shuffling the
rows of~$M_1 \bullet n_2$ with the rows of~$n_1 \bullet M_2$.
\index{operator!$\bullet$}%
By a routine computation, we obtain the following expression for the
product of~$\MT{k}^\star$:
\begin{equation}
    \FF^\star_{M_1} \cdot \FF^\star_{M_2} =
    \sum_{M \in M_1 {\mbox{\tiny \begin{math}\lshuffle\end{math}}} M_2}
    \FF^\star_M.
\end{equation}

For instance, in~$\MT{1}^\star$ one has
\begin{equation}\begin{split}
    \FF^\star_{\begin{Matrice}
        \ZeroB & \UnB \\
        \UnB & \UnB
    \end{Matrice}}
    \cdot
    \FF^\star_{\begin{Matrice}
        \UnR & \ZeroR \\
        \ZeroR & \UnR
    \end{Matrice}}
    & =
    \FF^\star_{\begin{Matrice}
        \ZeroB & \UnB & \Zero & \Zero \\
        \UnB & \UnB & \Zero & \Zero \\
        \Zero & \Zero & \UnR & \ZeroR \\
        \Zero & \Zero & \ZeroR & \UnR
    \end{Matrice}}
    +
    \FF^\star_{\begin{Matrice}
        \ZeroB & \UnB & \Zero & \Zero \\
        \Zero & \Zero & \UnR & \ZeroR \\
        \UnB & \UnB & \Zero & \Zero \\
        \Zero & \Zero & \ZeroR & \UnR
    \end{Matrice}}
    +
    \FF^\star_{\begin{Matrice}
        \Zero & \Zero & \UnR & \ZeroR \\
        \ZeroB & \UnB & \Zero & \Zero \\
        \UnB & \UnB & \Zero & \Zero \\
        \Zero & \Zero & \ZeroR & \UnR
    \end{Matrice}} \\[1em]
    & +
    \FF^\star_{\begin{Matrice}
        \ZeroB & \UnB & \Zero & \Zero \\
        \Zero & \Zero & \UnR & \ZeroR \\
        \Zero & \Zero & \ZeroR & \UnR \\
        \UnB & \UnB & \Zero & \Zero
    \end{Matrice}}
    +
    \FF^\star_{\begin{Matrice}
        \Zero & \Zero & \UnR & \ZeroR \\
        \ZeroB & \UnB & \Zero & \Zero \\
        \Zero & \Zero & \ZeroR & \UnR \\
        \UnB & \UnB & \Zero & \Zero
    \end{Matrice}}
    +
    \FF^\star_{\begin{Matrice}
        \Zero & \Zero & \UnR & \ZeroR \\
        \Zero & \Zero & \ZeroR & \UnR \\
        \ZeroB & \UnB & \Zero & \Zero \\
        \UnB & \UnB & \Zero & \Zero
    \end{Matrice}}.
\end{split}\end{equation}
\medskip

Let~$M$ be a $k$-packed matrix. By duality, the coproduct
in~$\MT{k}^\star$ satisfies
\begin{equation}
    \Delta\left(\FF^\star_M\right) =
    \sum_{M_1, M_2 \in \EnsMT_k}
    \left\langle \FF_{M_1} \cdot \FF_{M_2},
    \FF^\star_M \right\rangle
    \: \FF^\star_{M_1} \otimes \FF^\star_{M_2}.
\end{equation}
By a routine computation, we obtain the following expression for the
coproduct of~$\MT{k}^\star$:
\begin{equation}
    \Delta\left(\FF^\star_M\right) =
    \sum_{M = M_1 \DecompC M_2} \FF^\star_{\Compr(M_1)}
        \otimes \FF^\star_{\Compr(M_2)}.
\end{equation}
For instance, in~$\MT{1}^\star$ one has
\begin{equation}
    \Delta
    \FF^\star_{\begin{Matrice}
        \Zero & \Zero & \Un & \Zero \\
        \Zero & \Zero & \Zero & \Un \\
        \Un & \Zero & \Zero & \Zero \\
        \Un & \Un & \Zero & \Zero
    \end{Matrice}}
    =
    \FF^\star_{\begin{Matrice}
        \Zero & \Zero & \Un & \Zero \\
        \Zero & \Zero & \Zero & \Un \\
        \Un & \Zero & \Zero & \Zero \\
        \Un & \Un & \Zero & \Zero
    \end{Matrice}}
    \otimes
    \FF^\star_{\emptyset}
    +
    \FF^\star_{\begin{Matrice}
        \Un & \Zero \\
        \Un & \Un
    \end{Matrice}}
    \otimes
    \FF^\star_{\begin{Matrice}
        \Un & \Zero \\
        \Zero & \Un
    \end{Matrice}}
    +
    \FF^\star_{\begin{Matrice}
        \Zero & \Zero & \Un \\
        \Un & \Zero & \Zero \\
        \Un & \Un & \Zero \\
    \end{Matrice}}
    \otimes
    \FF^\star_{\begin{Matrice}
        \Un
    \end{Matrice}}
    +
    \FF^\star_{\emptyset}
    \otimes
    \FF^\star_{\begin{Matrice}
        \Zero & \Zero & \Un & \Zero \\
        \Zero & \Zero & \Zero & \Un \\
        \Un & \Zero & \Zero & \Zero \\
        \Un & \Un & \Zero & \Zero
    \end{Matrice}}.
\end{equation}
\medskip

Let us denote by $M^T$ the transpose of~$M$.
\index{operator!$M^T$}%
\begin{Proposition} \label{prop:Autodualite_MT}
    The map~$\phi : \MT{k} \to \MT{k}^\star$ linearly defined for any
    $k$-packed matrix~$M$ by
    \begin{equation}
        \phi\left(\FF_M\right) := \FF^\star_{M^T}
    \end{equation}
    is a Hopf isomorphism.
\end{Proposition}
\begin{proof}
    The product and the coproduct of~$\MT{k}$ in the fundamental basis
    handle the columns of the matrices while the product and the coproduct
    of~$\MT{k}^\star$ in the adjoint basis of the fundamental basis handle
    the rows. Since the transpose of a matrix swaps its rows and its
    columns, $\phi$ is a Hopf isomorphism.
\end{proof}
\medskip

Since the transpose of any packed matrix of~$\EnsMT_{k, n, \ell}$ also
belongs to~$\EnsMT_{k, n, \ell}$, Proposition~\ref{prop:Autodualite_MT}
also implies that~$\MTN{k}$ and~$\MTL{k}$ are self-dual for the
isomorphism~$\phi$.
\medskip

\subsubsection{Primitive elements}
For any~$k$-packed matrix~$M$, define
\begin{equation}
    \WW^M := \FF^\star_{M_1} \cdot \ldots \cdot \FF^\star_{M_r}
\end{equation}
where the~$M_i$ are connected packed matrices (see
Section~\ref{subsubsec:freeness}) and~$M = M_1 \Over \dots \Over M_r$.
Then, we have
\begin{equation}
    \WW^M = \FF^\star_M + \sum_{M' \in R} \FF^\star_{M'}
\end{equation}
where any matrix~$M'$ of~$R$ satisfies~$M^T \OrdMT M'^T$ since the product
in~$\MT{k}^\star$ consists in shifting and shuffling rows of matrices.
Thus, by triangularity, the~$\WW^M$ form a basis of~$\MT{k}^\star$.
Moreover, for any $k$-packed matrices~$M_1$ and~$M_2$, the product
of~$\MT{k}^\star$ is expressed as
\begin{equation}
    \WW^{M_1} \cdot \WW^{M_2} = \WW^{M_1 \Over M_2}.
\end{equation}
\medskip

Let us denote by~$\VV_M$, where the~$M$ are $k$-packed matrices, the
adjoint elements of the~$\WW^M$.
\begin{Proposition} \label{prop:Elements_Primitifs}
    The elements~$\VV_M$, where~$M$ are connected $k$-packed matrices,
    form a basis of the vector space of primitive elements of~$\MT{k}$.
\end{Proposition}
\begin{proof}
    Since~$\WW^M$ is indecomposable, by duality,
    $\VV_M$ is primitive. Moreover, let~$X$ be a primitive element
    of~$\MT{k}$. Then, $X$ is expressed as
    \begin{equation}
        X = \sum_{M \in \EnsMT_k} c_M \VV_M.
    \end{equation}
    Let~$M$ be a nonconnected $k$-packed matrix and~$M = M_1 \Over M_2$
    be a nontrivial factorization. Then, by duality, the coefficient of
    $\VV_{M_1} \otimes \VV_{M_2}$ in~$\Delta(X)$ is~$c_M$. Since~$X$
    is primitive, $c_M = 0$, showing that~$X$ is a sum of~$\VV_M$
    where~$M$ are connected $k$-packed matrices.
\end{proof}
\medskip

By Proposition~\ref{prop:Elements_Primitifs}, the~$\VV_M$, where~$M$ are
connected $k$-packed matrices, generate the Lie algebra of primitive
elements of~$\MT{k}$. The first few dimensions of the Lie algebras of
primitive elements of~$\MTN{1}$, $\MTN{2}$, $\MTL{1}$, $\MTL{2}$
are respectively given by~\eqref{equ:Dim_Gen_Alg_PMN1},
\eqref{equ:Dim_Gen_Alg_PMN2}, \eqref{equ:Dim_Gen_Alg_PML1},
and~\eqref{equ:Dim_Gen_Alg_PML2}.
\medskip

\subsection{Bidendriform bialgebra structure}

\subsubsection{Dendriform algebra structure}
An algebra~$(\SetPart{A}, \cdot)$ admits a
\emph{dendriform algebra structure}~\cite{Lod01}
\index{dendriform algebra}%
if its product can be split into two operations
\begin{equation}
    \cdot = \Gauche + \Droite,
\end{equation}
where~$\Gauche, \Droite : \SetPart{A} \otimes \SetPart{A} \to \SetPart{A}$
are non-degenerated linear maps such that, by denoting by~$\SetPart{A}^+$
the augmentation ideal of~$\SetPart{A}$, for all $x, y, z \in \SetPart{A}^+$,
the following relations hold
\begin{subequations}
\begin{align}
    (x \Gauche y) \Gauche z & = x \Gauche (y \cdot z),
        \label{eq:Dendr1} \\
    (x \Droite y) \Gauche z & = x \Droite (y \Gauche z),
        \label{eq:Dendr2} \\
    (x \cdot y) \Droite z & = x \Droite (y \Droite z).
        \label{eq:Dendr3}
\end{align}
\end{subequations}
\medskip

For any nonempty matrix~$M$, we shall denote by~$\LastC(M)$
\index{operator!$\LastC$}%
its last column. Let us endow~$\MT{k}^+$ with two products~$\Gauche$
and~$\Droite$ linearly defined, for any nonempty $k$-packed matrices~$M_1$
and~$M_2$ of respective sizes~$n_1$ and~$n_2$, by
\begin{equation} \label{equ:produit_gauche}
\index{operator!\begin{math}\Gauche\end{math}}%
    \FF_{M_1} \Gauche \FF_{M_2} :=
    \sum_{\substack{M \in M_1 \cshuffle M_2 \\
    \LastC(M) = \LastC\left(M_1 \circ n_2\right)}}
    \FF_M
\end{equation}
and
\begin{equation} \label{equ:produit_droit}
\index{operator!\begin{math}\Droite\end{math}}%
    \FF_{M_1} \Droite \FF_{M_2} :=
    \sum_{\substack{M \: \in \: M_1 \cshuffle M_2 \\
    \LastC(M) = \LastC\left(n_1 \circ {M_2}\right)}}
    \FF_M.
\end{equation}

In other words, the matrices appearing in a $\Gauche$-product (resp.
$\Droite$-product) in the fundamental basis involving~$M_1$ and~$M_2$
are the matrices~$M$ obtained by shifting and shuffling the columns
of~$M_1$ and~$M_2$ such that the last column of~$M$ comes from~$M_1$
(resp. $M_2$). For example,
\begin{align}
    \FF_{\begin{Matrice}
        \ZeroB & \UnB \\
        \UnB & \UnB
    \end{Matrice}}
    \Gauche
    \FF_{\begin{Matrice}
        \UnR & \ZeroR \\
        \ZeroR & \UnR
    \end{Matrice}}
    & =
        \FF_{\begin{Matrice}
        \ZeroB & \Zero & \Zero & \UnB \\
        \UnB & \Zero & \Zero & \UnB \\
        \Zero & \UnR & \ZeroR & \Zero \\
        \Zero & \ZeroR & \UnR & \Zero
    \end{Matrice}}
    +
    \FF_{\begin{Matrice}
        \Zero & \ZeroB & \Zero & \UnB \\
        \Zero & \UnB & \Zero & \UnB \\
        \UnR & \Zero & \ZeroR & \Zero \\
        \ZeroR & \Zero & \UnR & \Zero
    \end{Matrice}}
    +
    \FF_{\begin{Matrice}
        \Zero & \Zero & \ZeroB & \UnB \\
        \Zero & \Zero & \UnB & \UnB \\
        \UnR & \ZeroR & \Zero & \Zero \\
        \ZeroR & \UnR & \Zero & \Zero
    \end{Matrice}}, \displaybreak[0] \\[1em]
    \FF_{\begin{Matrice}
        \ZeroB & \UnB \\
        \UnB & \UnB
    \end{Matrice}}
    \Droite
    \FF_{\begin{Matrice}
        \UnR & \ZeroR \\
        \ZeroR & \UnR
    \end{Matrice}}
    & =
        \FF_{\begin{Matrice}
        \ZeroB & \UnB & \Zero & \Zero \\
        \UnB & \UnB & \Zero & \Zero \\
        \Zero & \Zero & \UnR & \ZeroR \\
        \Zero & \Zero & \ZeroR & \UnR
    \end{Matrice}}
    +
    \FF_{\begin{Matrice}
        \ZeroB & \Zero & \UnB & \Zero \\
        \UnB & \Zero & \UnB & \Zero \\
        \Zero & \UnR & \Zero & \ZeroR \\
        \Zero & \ZeroR & \Zero & \UnR
    \end{Matrice}}
    +
    \FF_{\begin{Matrice}
        \Zero & \ZeroB & \UnB & \Zero \\
        \Zero & \UnB & \UnB & \Zero \\
        \UnR & \Zero & \Zero & \ZeroR \\
        \ZeroR & \Zero & \Zero & \UnR
    \end{Matrice}}.
\end{align}
\medskip

Since the last column of any matrix appearing in the shifted shuffle of
two matrices comes from one of the two operands, for any nonempty packed
matrices~$M_1$ and~$M_2$, one obviously has
\begin{equation} \label{eq:DecompositionDendrMT}
    \FF_{M_1} \cdot \FF_{M_2} =
    \FF_{M_1} \Gauche \FF_{M_2} + \FF_{M_1} \Droite \FF_{M_2}.
\end{equation}

\begin{Proposition} \label{prop:PMDendriforme}
    The Hopf algebra~$\MT{k}$ admits a dendriform algebra structure for
    the products~$\Gauche$ and~$\Droite$.
\end{Proposition}
\begin{proof}
    We have to prove that~\eqref{eq:Dendr1}, \eqref{eq:Dendr2},
    and~\eqref{eq:Dendr3} hold. Let~$M_1$, $M_2$, and~$M_3$ be three
    packed matrices of respective sizes $n_1$, $n_2$ and $n_3$.
    \smallskip

    By definition of $\Gauche$ and $\Droite$, and since $\cshuffle$ is
    associative, the set $S$ of matrices indexing the support of
    $(\FF_{M_1} \Droite \FF_{M_2}) \Gauche \FF_{M_3}$ satisfies
    \begin{equation}\begin{split}
        S & = \{M \in (M_1 \cshuffle M_2) \cshuffle M_3 :
            \LastC(M) = \LastC(n_1 \circ M_2 \circ n_3)\} \\
            & = \{M \in M_1 \cshuffle (M_2 \cshuffle M_3) :
            \LastC(M) = \LastC(n_1 \circ M_2 \circ n_3)\}.
    \end{split}\end{equation}
    Hence, $S$ also is the set of matrices indexing the support of
    $\FF_{M_1} \Droite (\FF_{M_2} \Gauche \FF_{M_3})$. Since the shifted
    shuffle of packed matrices is multiplicity-free,
    \eqref{eq:Dendr2} holds.
    \smallskip

    By definition of $\Gauche$ and $\Droite$, and since $\cshuffle$ is
    associative, the set $T$ of matrices indexing the support of
    $(\FF_{M_1} \Gauche \FF_{M_2}) \Gauche \FF_{M_3}$ satisfies
    \begin{equation}\begin{split}
        T & = \{ M \in (M_1 \cshuffle M_2) \cshuffle M_3 :
        \LastC(M) = \LastC(M_1 \circ (n_2 + n_3)) \} \\
            & = \{ M \in M_1 \cshuffle (M_2 \cshuffle M_3) :
        \LastC(M) = \LastC(M_1 \circ (n_2 + n_3)) \}.
    \end{split}\end{equation}
    Hence, by~\eqref{eq:DecompositionDendrMT}, $T$ also is the set of
    matrices indexing the support of
    $\FF_{M_1} \Gauche (\FF_{M_2} \cdot \FF_{M_3})$. Since the shifted
    shuffle of packed matrices is multiplicity-free, \eqref{eq:Dendr1}
    holds. By a very similar argument, \eqref{eq:Dendr3} also holds.
\end{proof}
\medskip

\subsubsection{Codendriform coalgebra structure}
By dualizing the notion of dendriform algebra structure, one obtains
the notion of {\em codendriform coalgebra structure}~\cite{Foi07}.
\index{codendriform coalgebra}%
A coalgebra~$(\SetPart{C}, \Delta)$ admits a codendriform coalgebra
structure if its coproduct can be split into two operations
\begin{equation}
    \Delta = 1 \otimes \identity + \DeltaG + \DeltaD
        + \identity \otimes 1,
\end{equation}
where~$\DeltaG, \DeltaD : \SetPart{C} \to \SetPart{C} \otimes \SetPart{C}$
are non-degenerated linear maps such that following relations hold
\begin{subequations}
\begin{align}
    (\DeltaG \otimes \identity) \circ \DeltaG = &
    (\identity \otimes \DeltaB) \circ \DeltaG,
        \label{eq:CoAlgDendr1} \\
    (\DeltaD \otimes \identity) \circ \DeltaG = & (\identity \otimes \DeltaG)
    \circ \DeltaD, \label{eq:CoAlgDendr2} \\
    (\DeltaB \otimes \identity) \circ \DeltaD = &
    (\identity \otimes \DeltaD) \circ \DeltaD, \label{eq:CoAlgDendr3}
\end{align}
\end{subequations}
where~$\DeltaB := \DeltaG + \DeltaD$.
\medskip

For any nonempty matrix~$M$, we shall denote by~$\LastR(M)$
\index{operator!$\LastR$}%
its last row. Let us endow~$\MT{k}$ with two coproducts~$\DeltaG$
and~$\DeltaD$ linearly defined, for any nonempty $k$-packed matrix~$M$,
by
\begin{equation}
    \index{operator!\begin{math}\Delta_\Gauche\end{math}}%
    \DeltaG\left(\FF_M\right) :=
    \sum_{\substack{M = L \bullet R \\ \LastR(L \bullet r) = \LastR(M)}}
    \FF_{\Compr(L)} \otimes \FF_{\Compr(R)}
\end{equation}
and
\begin{equation}
    \index{operator!\begin{math}\Delta_\Droite\end{math}}%
    \DeltaD\left(\FF_M\right) :=
    \sum_{\substack{M = L \bullet R \\ \LastR(\ell \bullet R) = \LastR(M)}}
    \FF_{\Compr(L)} \otimes \FF_{\Compr(R)},
\end{equation}
where~$r$ (resp. $\ell$) is the number of columns of~$R$ (resp. $L$). In
other words, the pairs of matrices appearing in a $\DeltaG$-coproduct
(resp. $\DeltaD$-coproduct) in the fundamental basis are the pairs~$(L, R)$
of packed matrices such that the last row of~$L$ (resp. $R$) comes from
the last row of~$M$. For example,
\begin{align}
    \DeltaG
    \FF_{\begin{Matrice}
        \Un & \Zero & \Zero & \Zero & \Zero & \Zero \\
        \Zero & \Un & \Un & \Zero & \Zero & \Zero \\
        \Zero & \Zero & \Zero & \Zero & \Zero & \Un \\
        \Zero & \Zero & \Zero & \Un & \Zero & \Zero \\
        \Zero & \Zero & \Zero & \Un & \Un & \Zero \\
        \Zero & \Zero & \UnB & \Zero & \Zero & \Zero
    \end{Matrice}}
    = &
    \FF_{\begin{Matrice}
        \Un & \Zero & \Zero \\
        \Zero & \Un & \Un \\
        \Zero & \Zero & \UnB
    \end{Matrice}}
    \otimes
    \FF_{\begin{Matrice}
        \Zero & \Zero & \Un \\
        \Un & \Zero & \Zero \\
        \Un & \Un & \Zero
    \end{Matrice}}
    +
    \FF_{\begin{Matrice}
        \Un & \Zero & \Zero & \Zero & \Zero \\
        \Zero & \Un & \Un & \Zero & \Zero \\
        \Zero & \Zero & \Zero & \Un & \Zero \\
        \Zero & \Zero & \Zero & \Un & \Un \\
        \Zero & \Zero & \UnB & \Zero & \Zero
    \end{Matrice}}
    \otimes
    \FF_{\begin{Matrice}
        \Un
    \end{Matrice}}, \displaybreak[0] \\[1em]
    \DeltaD
    \FF_{\begin{Matrice}
        \Un & \Zero & \Zero & \Zero & \Zero & \Zero \\
        \Zero & \Un & \Un & \Zero & \Zero & \Zero \\
        \Zero & \Zero & \Zero & \Zero & \Zero & \Un \\
        \Zero & \Zero & \Zero & \Un & \Zero & \Zero \\
        \Zero & \Zero & \Zero & \Un & \Un & \Zero \\
        \Zero & \Zero & \UnR & \Zero & \Zero & \Zero
    \end{Matrice}}
    = &
    \FF_{\begin{Matrice}
        \Un
    \end{Matrice}}
    \otimes
    \FF_{\begin{Matrice}
        \Un & \Un & \Zero & \Zero & \Zero \\
        \Zero & \Zero & \Zero & \Zero & \Un \\
        \Zero & \Zero & \Un & \Zero & \Zero \\
        \Zero & \Zero & \Un & \Un & \Zero \\
        \Zero & \UnR & \Zero & \Zero & \Zero
    \end{Matrice}}.
\end{align}
\medskip

Since by Lemma \ref{lem:Decomposition}, one cannot vertically split
a packed matrix by separating two nonzero entries on a same row, for
any nonempty packed matrix~$M$, one has
\begin{equation} \label{eq:DecompositionCoDendrMT}
    \Delta\left(\FF_M\right) =
    1 \otimes \FF_M + \DeltaG\left(\FF_M\right) +
    \DeltaD\left(\FF_M\right) + \FF_M \otimes 1.
\end{equation}

\begin{Proposition} \label{prop:PMDendriformCoalgebra}
    The Hopf algebra~$\MT{k}$ admits a codendriform coalgebra structure
    for the coproducts~$\DeltaG$ and~$\DeltaD$.
\end{Proposition}
Since the proof of this statement is similar to that of
Proposition~\ref{prop:PMDendriforme} it has been omitted.
\medskip

\subsubsection{Bidendriform bialgebra structure}
A bialgebra~$(\SetPart{B}, \cdot, \Delta)$ admits a
{\em bidendriform bialgebra structure}~\cite{Foi07}
\index{bidendriform bialgebra}%
if~$\SetPart{B}$ admits both a dendriform algebra
$(\SetPart{B}, \Gauche, \Droite)$ and a codendriform coalgebra
$(\SetPart{B}, \DeltaG, \DeltaD)$ structure with some extra compatibility
relations between~$(\Gauche, \Droite)$ and~$(\DeltaG, \DeltaD)$.
\medskip

\begin{Theoreme} \label{thm:PMBidendriforme}
    The Hopf algebra~$\MT{k}$ admits a bidendriform bialgebra structure
    for the products~$\Gauche$, $\Droite$ and the coproducts~$\DeltaG$,
    $\DeltaD$.
\end{Theoreme}
\begin{proof}
    By Propositions~\ref{prop:PMDendriforme}
    and~\ref{prop:PMDendriformCoalgebra}, $\MT{k}$ admits a dendriform
    algebra and a codendriform coalgebra structure.
    \smallskip

    The required extra compatibility relations (see~\cite{Foi07})
    between~$(\Gauche, \Droite)$ and~$(\DeltaG, \DeltaD)$ are
    established by arguments similar to the ones used in the proofs of
    Propositions~\ref{prop:PMDendriforme} and~\ref{prop:PMDendriformCoalgebra}.
\end{proof}
\medskip

Theorem~\ref{thm:PMBidendriforme} also implies that~$\MTN{k}$ and~$\MTL{k}$
admit a bidendriform bialgebra structure. Recall that an element~$x$ of
a Hopf algebra admitting a bidendriform bialgebra structure is
{\em totally primitive}
\index{totally primitive}%
if~$\DeltaG(x) = 0 = \DeltaD(x)$. Following~\cite{Foi07}, the generating
series~$\CalT_{k, n}(t)$
\index{series!$\CalT_{k, n}(t)$ }%
and~$\CalT_{k, \ell}(t)$
\index{series!$\CalT_{k, \ell}(t)$ }%
of totally primitive elements of~$\MTN{k}$ and~$\MTL{k}$ satisfy
respectively
\begin{equation}
    \CalT_{k, n}(t) = \frac{\CalH_{k, n}(t) - 1}{\CalH_{k, n}(t)^2}
    \qquad \mbox{and} \qquad
    \CalT_{k, \ell}(t) = \frac{\CalH_{k, \ell}(t) - 1}{\CalH_{k, \ell}(t)^2}.
\end{equation}
The first few dimensions of totally primitive elements of~$\MTN{1}$
and~$\MTN{2}$ are respectively
\begin{equation}
    0,\; 1,\; 5,\; 240,\; 40404,\; 24827208, \; 57266105928
\end{equation}
and
\begin{equation}
    0,\; 2,\; 48,\; 15640,\; 39023776,\; 813415850016, \; 147655768992433664.
\end{equation}
The first few dimensions of totally primitive elements of~$\MTL{1}$
and~$\MTL{2}$ are respectively
\begin{equation}
    0, \: 1, \: 0, \: 5, \: 36, \: 381, \: 4720, \: 67867, \: 1109434
\end{equation}
and
\begin{equation}
    0, \: 2, \: 0, \: 40, \: 576, \: 12192, \: 302080, \: 8686976, \:
    284015104.
\end{equation}
\medskip

\section{Related Hopf algebras} \label{sec:Liens_AHC}
In this section, we list some already known Hopf algebras and describe
their links with~$\MT{k}$. Next, we provide a method to construct
Hopf subalgebras of~$\MT{k}$.
\medskip

\subsection{Links with known algebras}

\subsubsection{Hopf algebra of colored permutations}
Recall that a {\em $k$-colored permutation} is a pair~$(\sigma, c)$
where~$\sigma$ is a permutation of size~$n$ and~$c$ is a word of
length~$n$ on the alphabet~$A_k \setminus \{0\}$.
\smallskip

In~\cite{NT10}, the authors endowed the vector spaces~$\FQSym^{(k)}$
\index{Hopf algebra!$\FQSym^{(k)}$}%
spanned by the set of all $k$-colored permutations with a Hopf algebra
structure. The elements~$\FF_{(\sigma, c)}$, where the~$(\sigma, c)$ are
$k$-colored permutations, form the fundamental basis of~$\FQSym^{(k)}$.
These Hopf algebras provide a generalization of~$\FQSym$
since~$\FQSym = \FQSym^{(1)}$.
\index{Hopf algebra!$\FQSym$}%

\begin{Proposition} \label{prop:Morphisme_MT_FQSym_k}
    The map~$\alpha_k : \FQSym^{(k)} \to \MTN{k}$
    \index{morphism!$\alpha_k$}%
    linearly defined, for any $k$-colored permutation~$(\sigma, c)$ by
    \begin{equation}
        \alpha_k\left(\FF_{(\sigma, c)}\right) := \FF_{M^{(\sigma, c)}}
    \end{equation}
    where~$M^{(\sigma, c)}$ is the $k$-packed matrix satisfying
    $M^{(\sigma, c)}_{ij} = c_j \, \delta_{i, \sigma_j}$
    is an injective Hopf morphism.
\end{Proposition}
\medskip

In particular, Proposition~\ref{prop:Morphisme_MT_FQSym_k} shows
that~$\MTN{1}$ contains~$\FQSym$. Notice that the map~$\alpha_k$ is
still well-defined on the codomain~$\MTL{k}$ instead of~$\MTN{k}$.
\medskip

\subsubsection{Hopf algebra of uniform block permutations}
Recall that a {\em uniform block permutation},
\index{uniform block permutation}%
or a {\em UBP} for short, of size~$n$ is a bijection~$\pi : \pi^d \to \pi^c$
where~$\pi^d$ and~$\pi^c$ are set partitions of~$[n]$, and, for
any~$e \in \pi^d$,~$e$ and~$\pi(e)$ have same cardinality.
\medskip

For instance, the map~$\pi$ defined by
\begin{equation} \label{equ:Exemple_UBP}
    \pi(\{1, 4, 5\}) := \{2, 5, 6\}, \quad
    \pi(\{2\}) := \{1\}, \quad
    \mbox{and} \quad
    \pi(\{3, 6\}) := \{3, 4\}
\end{equation}
is a UBP of size~$6$.
\medskip

In~\cite{AO08}, the authors endowed the vector space~$\UBP$ spanned by
the set of all UBPs with a Hopf algebra structure. The elements~$\FF_\pi$,
where the~$\pi$ are UBPs, form the fundamental basis of~$\UBP$.
\index{Hopf algebra!$\UBP$}%
The dimensions of~$\UBP$ form Sequence~\Sloane{A023998} of~\cite{Slo}
and the first few terms are
\begin{equation}
    1, \: 1, \: 3, \: 16, \: 131, \: 1496, \: 22482, \: 426833, \:
    9934563, \: 277006192, \: 9085194458.
\end{equation}
\medskip

\begin{Proposition} \label{prop:Morphisme_MT_UBP}
    The map~$\beta : \UBP^\star \to \MTN{1}$
    \index{morphism!$\beta$}%
    linearly defined, for any UBP~$\pi$ by
    \begin{equation}
        \beta\left(\FF^\star_\pi\right) := \FF_{M^\pi}
    \end{equation}
    where~$M^\pi$ is the $1$-packed matrix satisfying
    \begin{equation}
        M^\pi_{ij} :=
        \begin{cases}
            1 & \mbox{if there is }
                e \in \pi^d \mbox{ such that }
                j \in e \mbox{ and } i \in \pi(e), \\
            0 & \mbox{otherwise}.
        \end{cases}
    \end{equation}
    is an injective Hopf morphism.
\end{Proposition}
\medskip

For example, with the UBP~$\pi$ defined in~\eqref{equ:Exemple_UBP}, we
have
\begin{equation}
    \beta\left(\FF^\star_\pi\right) =
    \FF_{\begin{Matrice}
        \Zero & \Un & \Zero & \Zero & \Zero & \Zero \\
        \Un & \Zero & \Zero & \Un & \Un & \Zero \\
        \Zero & \Zero & \Un & \Zero & \Zero & \Un \\
        \Zero & \Zero & \Un & \Zero & \Zero & \Un \\
        \Un & \Zero & \Zero & \Un & \Un & \Zero \\
        \Un & \Zero & \Zero & \Un & \Un & \Zero
    \end{Matrice}}.
\end{equation}
\medskip

\begin{Corollaire}
    The Hopf algebra~$\UBP^\star$ is a free, cofree, and self-dual
    Hopf algebra which admits a bidendriform bialgebra structure.
\end{Corollaire}
\begin{proof}
    By Proposition~\ref{prop:Morphisme_MT_UBP} and the definition of
    the product on the fundamental basis of~$\UBP^\star$ (see~\cite{AO08}),
    we can see~$\UBP^\star$ as a Hopf subalgebra of~$\MTN{1}$ restricted
    on the elements~$\FF_M$ where the~$M$ are $1$-packed matrices such that
    there are UBPs~$\pi$ satisfying~$M^\pi = M$. This shows that~$\UBP^\star$
    inherits from the bidendriform bialgebra structure of~$\MTN{1}$ (see
    Theorem~\ref{thm:PMBidendriforme}). Now, since~$\UBP^\star$ admits a
    bidendriform bialgebra structure, by~\cite{Foi07}, it is free, cofree,
    and self-dual.
\end{proof}
\medskip

By using same arguments as those used in Section~\ref{subsec:Mult_Liberte},
one can build multiplicative bases of~$\UBP^\star$ by setting, for
any UBP~$\pi$,
\begin{equation}
    \EE^\star_{M^\pi} := \sum_{M^\pi \OrdMT M^{\pi'}} \FF_{M^{\pi'}}
    \qquad \mbox{and} \qquad
    \HH^\star_{M^\pi} := \sum_{M^{\pi'} \OrdMT M^\pi} \FF_{M^{\pi'}}.
\end{equation}
This gives another way to prove the freeness of~$\UBP^\star$ by using
same arguments as those of Theorem~\ref{thm:MT_Libre}. Hence, $\UBP^\star$
is freely generated by the elements $\EE_{M^\pi}$ (resp. $\HH_{M^\pi}$)
where the~$\pi$ are UBPs such that the~$M^\pi$ are connected (resp.
anti-connected) $1$-packed matrices. The first few numbers of algebraic
generators of~$\UBP^\star$ are
\begin{equation}
    0, \: 1, \: 2, \: 11, \: 98, \: 1202, \: 19052, \: 375692, \:
    8981392, \: 255253291, \: 8488918198
\end{equation}
and the first few dimensions of totally primitive elements are
\begin{equation}
    0, \: 1, \: 1, \: 7, \: 72, \: 962, \: 16135, \: 330624, \:
    8117752, \: 235133003, \: 7929041828.
\end{equation}
\medskip

Moreover, since for any UBP~$\pi$, there exists a UBP~$\pi^{-1}$ such
that the transpose of~$M^\pi$ is~$M^{\pi^{-1}}$, by
Proposition~\ref{prop:Autodualite_MT}, the map
$\phi : \UBP^\star \to \UBP$ linearly defined for any UBP~$\pi$ by
\begin{equation}
    \phi\left(\FF^\star_{M^\pi}\right) := \FF_{{M^\pi}^T}
\end{equation}
is an isomorphism.
\medskip

\subsubsection{Algebra of matrix quasi-symmetric functions}
In~\cite{DHT02} (see also~\cite{Hiv99}), the authors defined the vector
space~$\MQSym$
\index{Hopf algebra!$\MQSym$}%
spanned by the set of the (not necessarily square) matrices with entries
in~$\EnsNat$, and such that each row and each column contains at least
one nonzero entry. In this section, we simply call {\em matrices} such
sort of matrices. The elements $\MS_M$ such that~$M$ is a matrix form
the {\em quasi-multiword basis} of~$\MQSym$. The degree of a~$\MS_M$ is
given by the sum of the entries of~$M$.
\medskip

This vector space is endowed with an algebra structure where the product
of two basis elements is provided by the {\em augmented shuffle $\ashuffle$}.
\index{operator!$\ashuffle$}%
\index{augmented shuffle}%
Let $M_1$ and $M_2$ be two matrices. Any matrix $M$ of $M_1 \ashuffle M_2$
is obtained by concatenating $N_1$ and $N_2$ where $N_1$ (resp. $N_2$)
is obtained from $M_1$ (resp. $M_2$) by inserting some null rows and so
that $N_1$ and $N_2$ have both a same number of rows and each row of $M$
has at least one nonzero entry. For example,
\begin{equation}
    \MS_{\begin{Matrice} 2 & 1 \\ 0 & 1 \end{Matrice}} \cdot
    \MS_{\begin{Matrice} 1 & 3 \end{Matrice}} =
    \MS_{\begin{Matrice}
        2 & 1 & 0 & 0 \\ 0 & 1 & 0 & 0 \\ 0 & 0 & 1 & 3
    \end{Matrice}} +
    \MS_{\begin{Matrice} 2 & 1 & 0 & 0 \\ 0 & 1 & 1 & 3 \end{Matrice}} +
    \MS_{\begin{Matrice}
        2 & 1 & 0 & 0 \\ 0 & 0 & 1 & 3 \\ 0 & 1 & 0 & 0
    \end{Matrice}} +
    \MS_{\begin{Matrice} 2 & 1 & 1 & 3 \\ 0 & 1 & 0 & 0 \end{Matrice}} +
    \MS_{\begin{Matrice}
        0 & 0 & 1 & 3 \\ 2 & 1 & 0 & 0 \\ 0 & 1 & 0 & 0
    \end{Matrice}}.
\end{equation}
\medskip

Let us endow the set of matrices indexing~$\MQSym$ with a binary
relation~$\RelT$
\index{order relation!$\RelT$}%
defined in the following way. If~$M_1$ and~$M_2$ are two matrices such
that~$M_1$ has~$n$ rows and~$m$ columns, we have~$M_1 \RelT M_2$ if there
is an index~$i \in [n - 1]$ such that, denoting by~$s$ the number of~$0$
which end the $i$th row of~$M_1$, and by~$p$ the number of~$0$ which start
the~$(i + 1)$st row of~$M_1$, one has~$s + p \geq m$ and~$M_2$ is obtained
from~$M_1$ by overlaying its $i$th and $(i + 1)$st rows (see
Figure~\ref{fig:CouvertureOrdreMQSym}).
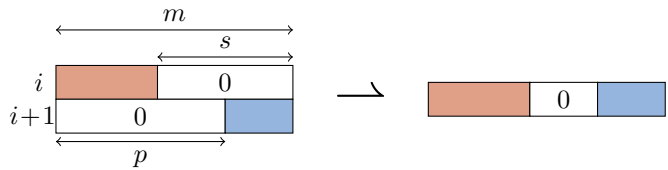
\begin{figure}[ht]
    \begin{tikzpicture}[scale=.45]
        \filldraw[draw=black,fill=BrickRed!40] (0,0) rectangle (3,-1);
        \draw[draw=black,fill=white] (3,0) rectangle
            node{\begin{math}0\end{math}}(7,-1);
        \draw[draw=black,fill=white] (0,-1) rectangle
            node{\begin{math}0\end{math}} (5,-2);
        \filldraw[draw=black,fill=RoyalBlue!40] (5,-1) rectangle (7,-2);
        \node at(-.5,-.5){\begin{math}i\end{math}};
        \node at(-.7,-1.5){\begin{math}i\!+\!1\end{math}};
        \draw[draw=black](3,.25)edge[<->]
            node[anchor=center,above]{\begin{math}s\end{math}}(7,.25);
        \draw[draw=black](0,-2.25)edge[<->]
            node[anchor=center,below]{\begin{math}p\end{math}}(5,-2.25);
        \draw[draw=black](0,1.05)edge[<->]
            node[anchor=center,above]{\begin{math}m\end{math}}(7,1.05);
        \node at(9,-1){\scalebox{2}{\begin{math}\rightharpoonup\end{math}}};
        \filldraw[draw=black,fill=BrickRed!40] (11,-.5) rectangle (14,-1.5);
        \draw[draw=black,fill=white] (14,-.5) rectangle
            node{\begin{math}0\end{math}}(16,-1.5);
        \filldraw[draw=black,fill=RoyalBlue!40] (16,-.5) rectangle (18,-1.5);
    \end{tikzpicture}
    \caption{The condition for overlaying the $i$th and $(i + 1)$st
    rows of a (not necessarily square) packed matrix according to the
    relation~$\RelT$. The darker regions contain any entries and the
    white ones, only zeros.}
    \label{fig:CouvertureOrdreMQSym}
\end{figure}
\medskip

We now endow the set of matrices that index~$\MQSym$ with the partial
order relation~$\OrdMQ$
\index{order relation!$\OrdMQ$}%
defined as the reflexive and transitive closure of~$\RelT$.
Figure~\ref{fig:OrdreMQ} shows an interval of this partial order.
\begin{figure}[ht]
     \begin{tikzpicture}[scale=.7]
         \node(1)at(0,0){$
         \begin{Matrice}
             \Un & \Un & 0   & 0 \\
             0   & 0   & \Un & 0 \\
             0   & \Un & \Un & 0 \\
             0   & 0   &0    & \Un
         \end{Matrice}$};
         \node(2)at(-3,-2){$
         \begin{Matrice}
             \Un & \Un & \Un & 0 \\
             0   & \Un & \Un & 0 \\
             0   &0    &0    & \Un
         \end{Matrice}$};
         \node(3)at(3,-2){$
         \begin{Matrice}
             \Un & \Un & 0   & 0 \\
             0   & 0   &\Un  & 0 \\
             0   &\Un  &\Un  & \Un
         \end{Matrice}$};
         \node(4)at(0,-4){$
         \begin{Matrice}
             \Un & \Un &\Un  & 0 \\
             0   &\Un  &\Un  & \Un
         \end{Matrice}$};
         \draw[line width=1pt,->] (1)--(2);
         \draw[line width=1pt,->] (1)--(3);
         \draw[line width=1pt,->] (2)--(4);
         \draw[line width=1pt,->] (3)--(4);
     \end{tikzpicture}
    \caption{The Hasse diagram of an interval for the order~$\OrdMQ$
    on (not necessarily square) packed matrices.}
    \label{fig:OrdreMQ}
\end{figure}
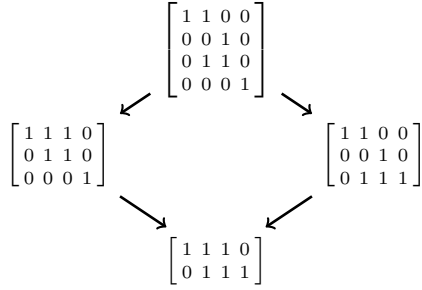
\medskip

\begin{Lemme} \label{lem:Shuffle_Augmente_Ordre_MQ}
    Let~$A$ and~$B$ be two $k$-packed matrices. Then,
    \begin{equation}
        \left\{C' : C \OrdMQ C', C \in A * B\right\}
        \; = \;
        \left\{C' \in A' \ashuffle B' :
        A \OrdMQ A', B \OrdMQ B'\right\},
    \end{equation}
    where~$*$ is the row shifted shuffle of $k$-packed matrices
    and~$\ashuffle$ is the augmented shuffle of matrices.
\end{Lemme}
\begin{proof}
    Let~$C'$ be a matrix such that $C \OrdMQ C'$ and~$C \in A * B$. By
    definition of the order~$\OrdMQ$ and the product~$*$, $C'$ can be
    obtained from~$C$ by overlaying rows coming from~$A$, rows coming
    from~$B$, or rows coming from~$A$ and~$B$. Let us denote by~$A'$
    (resp. $B'$) the matrix obtained from~$A$ (resp. $B$) by overlaying
    some of its rows. Then, we have~$A \OrdMQ A'$ and~$B \OrdMQ B'$, and,
    by definition of the augmented shuffle, $C' \in A' \ashuffle B'$.
    \smallskip

    Conversely, let~$C'$ be a matrix such that~$C' \in A' \ashuffle B'$
    where~$A'$ and~$B'$ are matrices satisfying~$A \OrdMQ A'$ and~$B \OrdMQ B'$.
    Then, by definition of the augmented shuffle of matrices, $C'$ can be
    obtained from a matrix~$C$ of $A * B$ by overlaying rows coming
    from~$A$, rows coming from~$B$, or rows coming from~$A$ and~$B$.
    Hence, $C \OrdMQ C'$.
\end{proof}
\medskip

\begin{Proposition} \label{prop:Morphisme_MT_MQSym}
    The map~$\gamma : \MTL{1}^\star \to \MQSym$
    \index{morphism!$\gamma$}%
    linearly defined, for any
    $1$-packed matrix~$M$ by
    \begin{equation}
        \gamma\left(\FF^\star_M\right) := \sum_{M \OrdMQ M'} \MS_{M'},
    \end{equation}
    is an injective algebra morphism.
\end{Proposition}
\begin{proof}
    Let~$M_1$ and~$M_2$ be two $1$-packed matrices. By definition
    of~$\gamma$, $\gamma (\FF^\star_{M_1} \cdot \FF^\star_{M_2})$ is
    multiplicity-free over the quasi-multiword basis of~$\MQSym$.
    Moreover, since the augmented shuffle is multiplicity-free,
    $\gamma (\FF^\star_{M_1}) \cdot \gamma (\FF^\star_{M_2})$ also is.
    Lemma~\ref{lem:Shuffle_Augmente_Ordre_MQ} implies that these two
    elements are equal and then, that $\gamma$ is an algebra morphism.
    The injectivity of~$\gamma$ follows by triangularity.
\end{proof}
\medskip

For instance, one has
\begin{equation}
    \gamma
    \FF^\star_{
    \begin{Matrice}
        \Un & \Un & \Zero & \Zero \\
        \Zero & \Zero & \Un & \Zero \\
        \Zero & \Un & \Un & \Zero \\
        \Zero & \Zero & \Zero & \Un
    \end{Matrice}} =
    \MS_{
    \begin{Matrice}
        \Un & \Un & 0 & 0 \\
        0 & 0 & \Un & 0 \\
        0 & \Un & \Un & 0 \\
        0 & 0 & 0 & \Un
    \end{Matrice}}
    +
    \MS_{
    \begin{Matrice}
        \Un & \Un & \Un & 0 \\
        0 & \Un & \Un & 0 \\
        0 & 0 & 0 & \Un
    \end{Matrice}}
    +
    \MS_{
    \begin{Matrice}
        \Un & \Un & 0 & 0 \\
        0 & 0 & \Un & 0 \\
        0 & \Un & \Un & \Un
    \end{Matrice}}
    +
    \MS_{
    \begin{Matrice}
        \Un & \Un & \Un & 0 \\
        0 & \Un & \Un & \Un
    \end{Matrice}}.
\end{equation}
\medskip

Notice that~$\gamma$ is not a Hopf morphism since it is not a coalgebra
morphism. Indeed, we have
\begin{equation}\
    \Delta \gamma
    \FF^\star_{
    \begin{Matrice}
        \Un & \Un \\
        \Un & \Zero
    \end{Matrice}}
    =
    1 \otimes
    \MS_{
    \begin{Matrice}
        \Un & \Un \\
        \Un & \Zero
    \end{Matrice}}
    +
    \MS_{
    \begin{Matrice}
        \Un & \Un \\
        \Un & \Zero
    \end{Matrice}}
    \otimes 1,
\end{equation}
but
\begin{equation}
    (\gamma \otimes \gamma) \Delta
    \FF^\star_{
    \begin{Matrice}
        \Un & \Un \\
        \Un & \Zero
    \end{Matrice}}
    =
        1 \otimes
    \MS_{
    \begin{Matrice}
        \Un & \Un \\
        \Un & \Zero
    \end{Matrice}}
    +
    \MS_{
    \begin{Matrice}
        \Un & \Un
    \end{Matrice}}
    \otimes
    \MS_{
    \begin{Matrice}
        \Un
    \end{Matrice}}
    +
    \MS_{
    \begin{Matrice}
        \Un & \Un \\
        \Un & \Zero
    \end{Matrice}}
    \otimes 1.
\end{equation}

\subsubsection{Diagram of embeddings}
The following diagram summarizes the relations between known Hopf algebras
related to~$\MT{k}$ and, more specifically, to its simple gradations~$\MTN{k}$
and~$\MTL{k}$. Plain arrows are Hopf algebra morphisms and the dotted
arrow is an algebra morphism. The Hopf algebra~$\ASM$ is the subject of
Section~\ref{sec:ASM}.
\begin{equation}
    \begin{split}
    \begin{tikzpicture}[scale=.25]
        \pgfmathsetmacro\x{9}
        \pgfmathsetmacro\y{7}
        \node (PMNk)        at (-\x, 0)
                {\CRouge{\begin{math}\MTN{k}\end{math}}};
        \node (PMLk)        at (\x, 0)
                {\CRouge{\begin{math}\MTL{k}\end{math}}};
        \node (PMN2)        at (-\x, -1*\y)
                {\CRouge{\begin{math}\MTN{2}\end{math}}};
        \node (PML2)        at (\x, -1*\y)
                {\CRouge{\begin{math}\MTL{2}\end{math}}};
        \node (PMN1)        at (-\x, -2*\y)
                {\CRouge{\begin{math}\MTN{1}\end{math}}};
        \node (PML1)        at (\x, -2*\y)
                {\CRouge{\begin{math}\MTL{1}\end{math}}};
        \node (FQSymk)      at (0, -1*\y)
                {\begin{math}\FQSym^{(k)}\end{math}};
        \node (FQSym2)      at (0, -2*\y)
                {\begin{math}\FQSym^{(2)}\end{math}};
        \node (FQSym)       at (0, -4*\y)
                {\begin{math}\FQSym\end{math}};
        \node (UBP)         at (-\x, -3*\y)
                {\begin{math}\UBP^\star\end{math}};
        \node (ASM)         at (0, -3*\y)
                {\CBleu{\begin{math}\ASM\end{math}}};
        \node (PML1Dual)    at (2*\x, -2*\y)
                {\CRouge{\begin{math}\MTL{1}^\star\end{math}}};
        \node (MQSym)       at (2*\x, -1*\y)
                {\begin{math}\MQSym\end{math}};
        \draw[Injection, dashed]            (PMN2)--(PMNk);
        \draw[Injection, dashed]            (PML2)--(PMLk);
        \draw[Injection, dashed]            (FQSym2)--(FQSymk);
        \draw[Injection]                    (PMN1)--(PMN2);
        \draw[Injection]                    (PML1)--(PML2);
        \draw[Injection, in=-15, out=90]    (ASM) edge (PMN1);
        \draw[Injection]                    (UBP) edge
                node[left]{\begin{math}\beta\end{math}} (PMN1);
        \draw[Injection, in=-45, out=45]    (FQSym) edge (FQSym2);
        \draw[Injection, in=-90, out=165]   (FQSym) edge (UBP);
        \draw[Injection]                    (FQSym)--(ASM);
        \draw[Injection, in=-60, out=150]   (FQSym) edge
                node[right]{\begin{math}\alpha_1\end{math}} (PMN1);
        \draw[Injection, in=240, out=30]    (FQSym) edge
                node[right]{\begin{math}\alpha_1\end{math}} (PML1);
        \draw[Injection, in=0, out=180]     (FQSym2) edge
                node[right]{\begin{math}\alpha_2\end{math}} (PMN2);
        \draw[Injection, in=180, out=0]     (FQSym2) edge
                node[left]{\begin{math}\alpha_2\end{math}} (PML2);
        \draw[Injection, in=0, out=180]     (FQSymk) edge
                node[right]{\begin{math}\alpha_k\end{math}} (PMNk);
        \draw[Injection, in=180, out=0]     (FQSymk) edge
                node[left]{\begin{math}\alpha_k\end{math}} (PMLk);
        \draw[Injection, dotted]            (PML1Dual) edge
                node[right]{\begin{math}\gamma\end{math}} (MQSym);
    \end{tikzpicture}
    \end{split}
\end{equation}
\medskip

\subsection{Equivalence relations and Hopf subalgebras}
\label{subsec:equivalences}
Several Hopf algebras can be constructed as Hopf subalgebras of the
Malvenuto-Reutenauer Hopf algebra~$\FQSym$~\cite{MR95,DHT02}. The main
examples are the Hopf algebra~$\PBT$ based on planar binary trees, first
defined by Loday and Ronco~\cite{LR98} and reconstructed by Hivert,
Novelli, and Thibon~\cite{HNT05}, and~$\FSym$ based on standard Young
tableaux, first discovered by Poirier and Reutenauer~\cite{PR95} and
reconstructed by Duchamp, Hivert, and Thibon~\cite{DHT02}.
\medskip

The starting point of these constructions is to define a
congruence~$\equiv$ on the free monoid~$A^*$ where~$A$ is a totally
ordered infinite alphabet. Then, when~$\equiv$ satisfies some
properties~\cite{HN07,Gir11}, the elements
\begin{equation}
    \PP_{[\sigma]_\equiv} :=
    \sum_{\sigma \in [\sigma]_\equiv} \FF_\sigma
\end{equation}
span a Hopf subalgebra of~$\FQSym$ indexed by the $\equiv$-equivalence
classes restricted to permutations. We shall show in this section that
an analogous construction works to construct Hopf subalgebras of~$\MT{k}$.
\medskip

\subsubsection{The sylvester and the plactic congruences}
Recall that the congruence allowing to reconstruct $\PBT$ is the
{\em sylvester congruence} (see~\cite{HNT02,HNT05}). It is denoted
by~$\Equiv{S}$ and is the reflexive and transitive closure of the
sylvester adjency relation $\Adj{S}$ defined for $u \in A^*$ and
$\La, \Lb, \Lc \in A$ by
\begin{equation}
    \La \Lc\,u\,\Lb \Adj{S} \Lc \La\,u\,\Lb
    \qquad \text{where} \quad \La \leq \Lb < \Lc.
\end{equation}
For example, the $\Equiv{S}$-equivalence class of the permutation $15423$
(see Figure~\ref{fig::sylvester:class:S}) is
\begin{equation}
    \{ 12543, 15243, 15423, 51243, 51423, 54123 \}.
\end{equation}
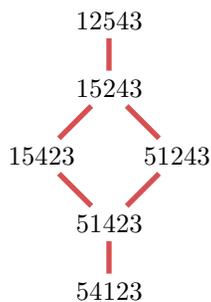
\begin{figure}[ht]
    \begin{tikzpicture}[scale=.45]
        \pgfmathsetmacro\x{2}
        \pgfmathsetmacro\y{2}
        \node (sig1)    at (0, 0)       {$12543$};
        \node (sig2)    at (0, -\y)     {$15243$};
        \node (sig3)    at (-\x, -2*\y) {$15423$};
        \node (sig4)    at (\x, -2*\y)  {$51243$};
        \node (sig5)    at (0, -3*\y)   {$51423$};
        \node (sig6)    at (0, -4*\y)   {$54123$};
        \draw[Arete]   (sig1)--(sig2);
        \draw[Arete]   (sig2)--(sig3);
        \draw[Arete]   (sig2)--(sig4);
        \draw[Arete]   (sig3)--(sig5);
        \draw[Arete]   (sig4)--(sig5);
        \draw[Arete]   (sig5)--(sig6);
    \end{tikzpicture}
    \caption[The sylvester equivalence class of the permutation $15423$.]
    {The sylvester equivalence class of the permutation $15423$. Edges
    represent sylvester adjacency relations.}
    \label{fig::sylvester:class:S}
\end{figure}
\medskip

Besides, recall that the congruence allowing to reconstruct $\FSym$ is
the {\em plactic congruence} (see~\cite{LS81,Lot02}). It is denoted
by~$\Equiv{P}$ and is the reflexive and transitive closure of the
plactic adjacency relation $\Adj{P}$ defined for $\La, \Lb, \Lc \in A$ by
\begin{subequations}
\begin{align}
    \La \Lc \Lb \Adj{P} \Lc \La \Lb
    \qquad \text{where} \quad \La \leq \Lb < \Lc, \\
    \Lb \La \Lc \Adj{P} \Lb \Lc \La
    \qquad \text{where} \quad \La < \Lb \leq \Lc.
\end{align}
\end{subequations}
\medskip

\subsubsection{The monoid of words of columns}
Let~$C_k^*$ be the free monoid generated by the set~$C_k$ of all
$n \times 1$-matrices with entries in~$A_k$, for all~$n \geq 1$. In
other words, the elements of $C_k^*$ are words whose letters are columns
and its product $\bullet$ is the concatenation of such words. When all
the letters of an element $M \in C_k^*$ have, as columns, a same number
of rows, $M$ is a matrix and we shall denote it as such in the sequel.
\medskip

The alphabet~$C_k$ is naturally equipped with the total order~$\leq$
where, for any~$c_1, c_2 \in C_k$, $c_1 \leq c_2$ if and only if the
bottom to top reading of the column~$c_1$ is lexicographically smaller
than the bottom to top reading of~$c_2$. For instance,
\begin{equation}
    \begin{Matrice}
        \Un \\
        \Zero \\
        \Zero
    \end{Matrice}
    \leq
    \begin{Matrice}
        \Zero \\
        \Zero \\
        \Un
    \end{Matrice}, \qquad
    \begin{Matrice}
        \Zero \\
        \Zero \\
        1 \\
        \Zero \\
        1
    \end{Matrice}
    \leq
    \begin{Matrice}
        \Zero \\
        \Zero \\
        1 \\
        1
    \end{Matrice}, \qquad
    \begin{Matrice}
        1 \\
        \Zero
    \end{Matrice}
    \leq
    \begin{Matrice}
        \Zero \\
        1 \\
        1 \\
        \Zero
    \end{Matrice}, \qquad
    \begin{Matrice}
        2 \\
        1 \\
        \Zero
    \end{Matrice}
    \leq
    \begin{Matrice}
        1 \\
        2 \\
        \Zero
    \end{Matrice}.
\end{equation}
\medskip

Since $C_k$ is then totally ordered and $C_k^*$ is a free monoid, one
can consider the previous two congruences on $C_k^*$ instead on $A^*$.
For instance, Figure~\ref{fig:exemples_classes_equ_matrices} represents
a $\Equiv{S}$-equivalence class and a $\Equiv{P}$-equivalence class of
packed matrices.
\begin{figure}[ht]
    \centering
    \subfigure[A sylvester equivalence class.]{
    \begin{tikzpicture}[scale=.45]
        \pgfmathsetmacro\x{4}
        \pgfmathsetmacro\y{4.5}
        \node (pm1)    at (0, 0)
        {$\begin{Matrice}
            1 & 0 & 1 & 1 & 0 \\
            1 & 0 & 1 & 0 & 1 \\
            0 & 1 & 1 & 0 & 0 \\
            0 & 0 & 0 & 1 & 1 \\
            0 & 0 & 1 & 1 & 0
        \end{Matrice}$};
        \node (pm2)    at (0, -\y)
        {$\begin{Matrice}
            1 & 1 & 0 & 1 & 0 \\
            1 & 1 & 0 & 0 & 1 \\
            0 & 1 & 1 & 0 & 0 \\
            0 & 0 & 0 & 1 & 1 \\
            0 & 1 & 0 & 1 & 0
        \end{Matrice}$};
        \node (pm3)    at (-\x, -2*\y)
        {$\begin{Matrice}
            1 & 1 & 1 & 0 & 0 \\
            1 & 1 & 0 & 0 & 1 \\
            0 & 1 & 0 & 1 & 0 \\
            0 & 0 & 1 & 0 & 1 \\
            0 & 1 & 1 & 0 & 0
        \end{Matrice}$};
        \node (pm4)    at (\x, -2*\y)
        {$\begin{Matrice}
            1 & 1 & 0 & 1 & 0 \\
            1 & 1 & 0 & 0 & 1 \\
            1 & 0 & 1 & 0 & 0 \\
            0 & 0 & 0 & 1 & 1 \\
            1 & 0 & 0 & 1 & 0
        \end{Matrice}$};
        \node (pm5)    at (0, -3*\y)
        {$\begin{Matrice}
            1 & 1 & 1 & 0 & 0 \\
            1 & 1 & 0 & 0 & 1 \\
            1 & 0 & 0 & 1 & 0 \\
            0 & 0 & 1 & 0 & 1 \\
            1 & 0 & 1 & 0 & 0
        \end{Matrice}$};
        \node (pm6)    at (0, -4*\y)
        {$\begin{Matrice}
            1 & 1 & 1 & 0 & 0 \\
            1 & 0 & 1 & 0 & 1 \\
            1 & 0 & 0 & 1 & 0 \\
            0 & 1 & 0 & 0 & 1 \\
            1 & 1 & 0 & 0 & 0
        \end{Matrice}$};
        \draw[Arete]   (pm1)--(pm2);
        \draw[Arete]   (pm2)--(pm3);
        \draw[Arete]   (pm2)--(pm4);
        \draw[Arete]   (pm3)--(pm5);
        \draw[Arete]   (pm4)--(pm5);
        \draw[Arete]   (pm5)--(pm6);
    \end{tikzpicture}}
    \qquad
    \qquad
    \subfigure[A plactic equivalence class.]{
    \begin{tikzpicture}[scale=.45]
        \pgfmathsetmacro\x{4}
        \pgfmathsetmacro\y{4.5}
        \node (pm1)    at (0, 0)
        {$\begin{Matrice}
            1 & 1 & 0 & 1 & 0 \\
            1 & 1 & 0 & 0 & 1 \\
            0 & 1 & 1 & 0 & 0 \\
            0 & 0 & 0 & 1 & 1 \\
            0 & 1 & 0 & 1 & 0
        \end{Matrice}$};
        \node (pm2)    at (-\x, -\y)
        {$\begin{Matrice}
            1 & 1 & 0 & 1 & 0 \\
            1 & 1 & 0 & 0 & 1 \\
            1 & 0 & 1 & 0 & 0 \\
            0 & 0 & 0 & 1 & 1 \\
            1 & 0 & 0 & 1 & 0
        \end{Matrice}$};
        \node (pm3)    at (\x, -\y)
        {$\begin{Matrice}
            1 & 1 & 1 & 0 & 0 \\
            1 & 1 & 0 & 0 & 1 \\
            0 & 1 & 0 & 1 & 0 \\
            0 & 0 & 1 & 0 & 1 \\
            0 & 1 & 1 & 0 & 0
        \end{Matrice}$};
        \node (pm4)    at (-\x, -2*\y)
        {$\begin{Matrice}
            1 & 1 & 1 & 0 & 0 \\
            1 & 1 & 0 & 0 & 1 \\
            1 & 0 & 0 & 1 & 0 \\
            0 & 0 & 1 & 0 & 1 \\
            1 & 0 & 1 & 0 & 0
        \end{Matrice}$};
        \node (pm5)    at (-\x, -3*\y)
        {$\begin{Matrice}
            1 & 1 & 1 & 0 & 0 \\
            1 & 0 & 1 & 0 & 1 \\
            1 & 0 & 0 & 1 & 0 \\
            0 & 1 & 0 & 0 & 1 \\
            1 & 1 & 0 & 0 & 0
        \end{Matrice}$};
        \draw[Arete]   (pm1)--(pm2);
        \draw[Arete]   (pm1)--(pm3);
        \draw[Arete]   (pm2)--(pm4);
        \draw[Arete]   (pm4)--(pm5);
    \end{tikzpicture}}
    \caption{Two equivalence classes of packed matrices.}
    \label{fig:exemples_classes_equ_matrices}
\end{figure}
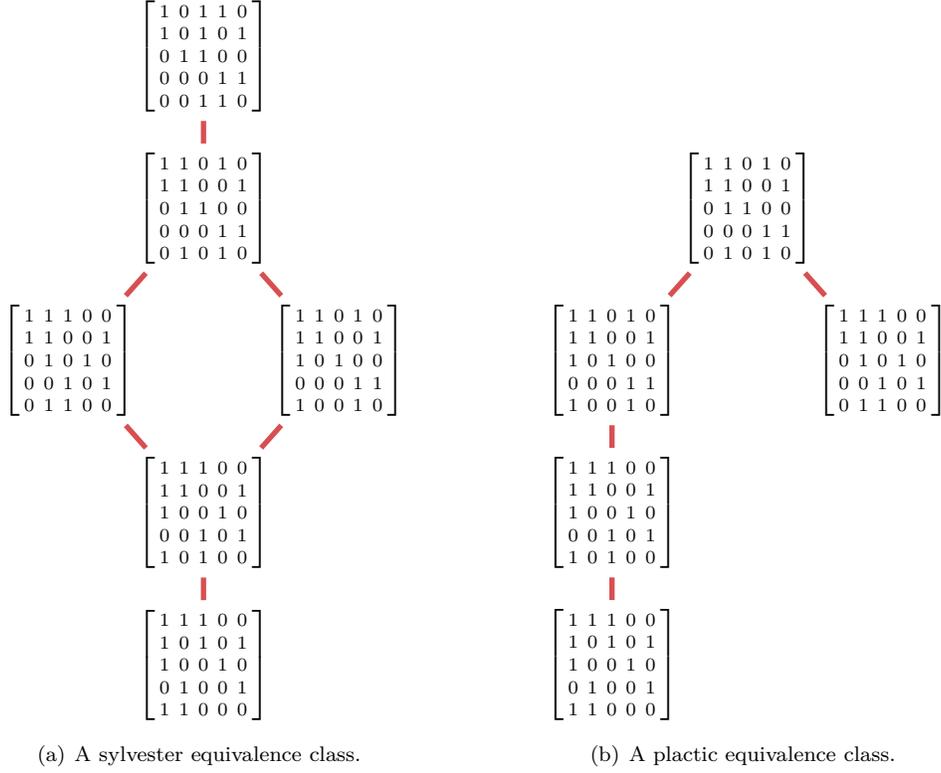
\medskip

The order relation $\leq$ on $C_k$ is compatible with the shifted
shuffle of packed matrices in the following sense. Let $M_1$ and
$M_2$ be two nonempty packed matrices and $M$ be a matrix appearing in
$M_1 \cshuffle M_2$. Then, if $c_1$ (resp. $c_2$) is a column of $M$
coming from $M_1$ (resp. $M_2$), we necessarily have $c_1 \leq c_2$ and
$c_1 \ne c_2$. The obvious analogous property holds for words of $A^*$
and the shifted shuffle of words.
\medskip

\subsubsection{Properties of equivalence relations}
An equivalence relation~$\equiv$ on~$C_k^*$ is a {\em monoid congruence}
if for all $u, v, u', v' \in C_k^*$,
\begin{equation}
    u \equiv u' \quad \text{and} \quad v \equiv v'
    \quad \text{imply} \quad
    u \bullet v \equiv u' \bullet v'.
\end{equation}
\medskip

Besides, we say that $\equiv$ is {\em compatible with the restriction to
alphabet intervals} if for any interval~$I$ of~$C_k$ and for
all~$u, v \in C_k^*$,
\begin{equation}
    u \equiv v \quad \text{implies} \quad u_{|I} \equiv v_{|I},
\end{equation}
where~$u_{|I}$ denotes the word obtained by erasing in~$u$ the letters
that are not in~$I$.
\medskip

Finally, we say that $\equiv$ is {\em compatible with the decompression
process} if for all~$u, v \in C_k^*$ such that $u$ and $v$ are matrices,
\begin{equation}
    u \equiv v \quad \mbox{if and only if} \quad
    \Compr(u) \equiv \Compr(v)
    \mbox{ and } \Eval(u) = \Eval(v),
\end{equation}
where~$\Eval(u)$ denotes the commutative image of~$u$.
\medskip

\subsubsection{Construction of Hopf subalgebras}
Given an equivalence relation~$\equiv$ on the words of~$C_k^*$ and
a $\equiv$-equivalence class $[M]_\equiv$ of packed matrices of~$C_k^*$,
we consider the elements
\begin{equation} \label{equ:Somme_Classe}
    \PP_{[M]_\equiv} := \sum_{M' \in [M]_\equiv} \FF_{M'}
\end{equation}
of $\MT{k}$.
\medskip

One has for instance
\begin{footnotesize}
\begin{equation}
    \PP_{\left[
    \begin{Matrice}
        1 & 1 & 0 & 1 & 0 \\
        1 & 1 & 0 & 0 & 1 \\
        0 & 1 & 1 & 0 & 0 \\
        0 & 0 & 0 & 1 & 1 \\
        0 & 1 & 0 & 1 & 0
    \end{Matrice}
    \right]_\Equiv{P}} =
    \FF_{\begin{Matrice}
        1 & 1 & 0 & 1 & 0 \\
        1 & 1 & 0 & 0 & 1 \\
        0 & 1 & 1 & 0 & 0 \\
        0 & 0 & 0 & 1 & 1 \\
        0 & 1 & 0 & 1 & 0
    \end{Matrice}} +
    \FF_{\begin{Matrice}
        1 & 1 & 0 & 1 & 0 \\
        1 & 1 & 0 & 0 & 1 \\
        1 & 0 & 1 & 0 & 0 \\
        0 & 0 & 0 & 1 & 1 \\
        1 & 0 & 0 & 1 & 0
    \end{Matrice}} +
    \FF_{\begin{Matrice}
        1 & 1 & 1 & 0 & 0 \\
        1 & 1 & 0 & 0 & 1 \\
        0 & 1 & 0 & 1 & 0 \\
        0 & 0 & 1 & 0 & 1 \\
        0 & 1 & 1 & 0 & 0
    \end{Matrice}} +
    \FF_{\begin{Matrice}
        1 & 1 & 1 & 0 & 0 \\
        1 & 1 & 0 & 0 & 1 \\
        1 & 0 & 0 & 1 & 0 \\
        0 & 0 & 1 & 0 & 1 \\
        1 & 0 & 1 & 0 & 0
    \end{Matrice}} +
    \FF_{\begin{Matrice}
        1 & 1 & 1 & 0 & 0 \\
        1 & 0 & 1 & 0 & 1 \\
        1 & 0 & 0 & 1 & 0 \\
        0 & 1 & 0 & 0 & 1 \\
        1 & 1 & 0 & 0 & 0
    \end{Matrice}}.
\end{equation}
\end{footnotesize}
\medskip

In particular, if~$\equiv$ is compatible with the decompression process,
any $\equiv$-equivalence class of a packed matrix only contains packed
matrices. The family $\PP_{[M]_\equiv}$, where the $[M]_\equiv$ are
$\equiv$-equivalence classes of packed matrices, forms then a basis of
a vector subspace of~$\MT{k}$ denoted by $\MT{k}^\equiv$.

\begin{Theoreme} \label{thm:Sous_AHC_Relation}
    Let~$\equiv$ be an equivalence relation on the words of~$C_k^*$ such
    that~$\equiv$
    \begin{enumerate}
        \item is a monoid congruence on~$C_k^*$; \label{item:Congr}
        \item is compatible with the restriction to alphabet intervals;
        \label{item:Inter}
        \item is compatible with the decompression process.
        \label{item:Decompr}
    \end{enumerate}
    Then, $\MT{k}^\equiv$ is a Hopf subalgebra of~$\MT{k}$.
\end{Theoreme}
\begin{proof}
    Let us show that the product is well-defined on $\MT{k}^\equiv$.
    Let~$[M_1]_\equiv$ and~$[M_2]_\equiv$ be two $\equiv$-equivalence
    classes of $k$-packed matrices. We have
    \begin{equation} \label{equ:Sous_AHC_Relation_Pr}
        \PP_{[M_1]_\equiv} \cdot \PP_{[M_2]_\equiv} =
        \sum_{\substack{M_1 \in [M_1]_\equiv \\ M_2 \in [M_2]_\equiv}} \;
        \sum_{M \in M_1 \cshuffle M_2} \FF_M.
    \end{equation}
    Let~$M$ be a $k$-packed matrix such that~$\FF_M$ appears
    in~\eqref{equ:Sous_AHC_Relation_Pr} and~$M'$ be a $k$-packed matrix
    such that~$M' \equiv M$. Then, there is a pair of $k$-packed
    matrices~$(M_1, M_2)$ such that~$M_1 \in [M_1]_\equiv$, $M_2 \in [M_2]_\equiv$,
    and~$M \in M_1 \cshuffle M_2$. By definition of the shifted shuffle,
    this pair is unique. Let~$m_1$ (resp. $m_2$) be the size of~$M_1$
    (resp. $M_2$). Let~$c_1$ (resp. $d_1)$ be the smallest (resp. greatest)
    column of~$M_1 \circ m_2$ and~$c_2$ (resp. $d_2)$ be the smallest
    (resp. greatest) column of~$m_1 \circ M_2$. Then, since all columns
    of~$M_1 \circ m_2$ are strictly smaller than the ones of~$m_1 \circ M_2$,
    the intervals~$[c_1, d_1]$ and~$[c_2, d_2]$ are disjoint.
    By~\eqref{item:Inter}, $M \equiv M'$
    implies~$M_{|[c_1, d_1]} \equiv M'_{|[c_1, d_1]}$
    and~$M_{|[c_2, d_2]} \equiv M'_{|[c_2, d_2]}$. Moreover,
    by~\eqref{item:Decompr} and by definition of~$\circ$, we have
    \begin{equation}
        M_1 = \Compr\left(M_{|[c_1, d_1]}\right) \equiv
        \Compr\left(M'_{|[c_1, d_1]}\right) =: M'_1
    \end{equation}
    and
    \begin{equation}
        M_2 = \Compr\left(M_{|[c_2, d_2]}\right) \equiv
        \Compr\left(M'_{|[c_2, d_2]}\right) =: M'_2.
    \end{equation}
    Thus, we have~$M' \in M'_1 \cshuffle M'_2$, showing that~$\FF_{M'}$
    also appears in~\eqref{equ:Sous_AHC_Relation_Pr} and that the product
    is well-defined on $\MT{k}^\equiv$.
    \smallskip

    Let us now show that the coproduct is well-defined on $\MT{k}^\equiv$.
    Let~$[M]_\equiv$ be a $\equiv$-equivalence class of $k$-packed matrices.
    We have
    \begin{equation} \label{equ:Sous_AHC_Relation_Copr}
        \Delta \left( \PP_{[M]_\equiv} \right) =
        \sum_{M \in [M]_\equiv} \;
        \sum_{M = L \bullet R}
        \FF_{\Compr(L)} \otimes \FF_{\Compr(R)}.
    \end{equation}
    Let~$M_1$ and~$M_2$ be two $k$-packed matrices such
    that~$\FF_{M_1} \otimes \FF_{M_2}$ appears
    in~\eqref{equ:Sous_AHC_Relation_Copr} and~$M'_1$ and~$M'_2$ two
    $k$-packed matrices such that~$M'_1 \equiv M_1$ and~$M'_2 \equiv M_2$.
    Then, there is a $k$-packed matrix~$M \in [M]_\equiv$ such
    that~$M = L \bullet R$, $\Compr(L) = M_1$, and~$\Compr(R) = M_2$.
    By~\eqref{item:Decompr}, $M'_1$ (resp. $M'_2$) is a permutation
    of~$M_1$ (resp. $M_2$). Thus, there exist two elements $L'$ and $R'$
    of~$C_k^*$ that are respectively permutations of~$L$ and~$R$ which
    satisfy~$\Compr(L') = M'_1$ and~$\Compr(R') = M'_2$. Again
    by~\eqref{item:Decompr}, we have~$L' \equiv L$ and~$R' \equiv R$. Now,
    by~\eqref{item:Congr},
    \begin{equation}
        M = L \bullet R \equiv L' \bullet R' =: M'.
    \end{equation}
    Hence, $M' \equiv M$ and~$\FF_{M'_1} \otimes \FF_{M'_2}$ also appears
    in~\eqref{equ:Sous_AHC_Relation_Copr}.
    \smallskip

    We have shown that the product and the coproduct of~$\MT{k}$ are still
    well-defined on $\MT{k}^\equiv$. Hence, $\MT{k}^\equiv$ is a Hopf
    subalgebra of~$\MT{k}$.
\end{proof}
\medskip

We say that an equivalence relation~$\equiv$ on~$C_k^*$ is a
{\em good congruence} if it satisfies~\eqref{item:Congr},
\eqref{item:Inter} and~\eqref{item:Decompr} of
Theorem~\ref{thm:Sous_AHC_Relation}. Let~$\equiv$ be a good congruence.
Note that since~$\equiv$ is compatible with the decompression process,
any matrix contained in a $\equiv$-equivalence class~$[M]_\equiv$ is
obtained by switching columns of $M$. Then, any $\equiv$-equivalence
class~$[M]_\equiv$ of~$k$-packed matrices only contains matrices whose
size and number of nonzero entries are the same as in $M$. Hence,
Theorem~\ref{thm:Sous_AHC_Relation} also implies that the
family~\eqref{equ:Somme_Classe} forms a basis of Hopf subalgebras of
both~$\MTN{k}$ and~$\MTL{k}$. We respectively denote these
by~$\MTN{k}^\equiv$ and~$\MTL{k}^\equiv$.
\medskip

\subsubsection{Computer experiments}
Let us recall here the definitions of some well-known good congruences.
\medskip

The {\em Baxter congruence} (see~\cite{Gir12}), denoted by~$\Equiv{Bx}$,
is the reflexive and transitive closure of the Baxter adjacency
relation $\Adj{Bx}$ defined for $u, v \in A^*$ and
$\La, \Lb, \Lc, \Ld \in A$ by
\begin{subequations}
\begin{align}
    \Lc\,u\,\La\Ld\,v\,\Lb \Adj{Bx} \Lc\,u\,\Ld\La\,v\,\Lb
    \qquad \text{where} \quad \La \leq \Lb < \Lc \leq \Ld, \\
    \Lb\,u\,\Ld\La\,v\,\Lc \Adj{Bx} \Lb\,u\,\La\Ld\,v\,\Lc
    \qquad \text{where} \quad \La < \Lb \leq \Lc < \Ld.
\end{align}
\end{subequations}
\medskip

The {\em Bell congruence} (see~\cite{Rey07}), denoted by~$\Equiv{Bl}$,
is the reflexive and transitive closure of the Bell adjacency relation
$\Adj{Bl}$ defined for $u \in A^*$ and $\La, \Lb, \Lc \in A$ by
\begin{equation}
    \La \Lc\,u\,\Lb \Adj{Bl} \Lc \La\,u\,\Lb
    \qquad \text{where} \quad \La \leq \Lb < \Lc
    \text{ and for all } \Ld \in u, \Ld \geq \Lc.
\end{equation}
\medskip

The {\em hypoplactic congruence} (see~\cite{KT97,KT99}), denoted
by~$\Equiv{H}$, is the reflexive and transitive closure of the
hypoplactic adjacency relation $\Adj{H}$ defined for $u \in A^*$ and
$\La, \Lb, \Lc \in A$ by
\begin{subequations}
\begin{align}
    \La\Lc\,u\,\Lb \Adj{H} \Lc\La\,u\,\Lb
    \qquad \text{where} \quad \La \leq \Lb < \Lc, \\
    \Lb\,u\,\Lc\La \Adj{H} \Lb\,u\,\La\Lc
    \qquad \text{where} \quad \La < \Lb \leq \Lc.
\end{align}
\end{subequations}
\medskip

The {\em total congruence} equivalence relation, denoted by~$\Equiv{T}$,
is the reflexive and transitive closure of the total adjacency relation
$\Adj{T}$ defined by~$u \,\Equiv{T}\, v$ for any~$u, v \in A^*$ such
that~$\Eval(u) = \Eval(v)$.
\medskip

By Theorem \ref{thm:Sous_AHC_Relation}, all these congruences lead
to bigraded Hopf subalgebras of $\MT{k}$. Table~\ref{tab:Dim_Sous_AHC_Rel}
shows first few dimensions of the Hopf subalgebras of~$\MTN{1}$ and~$\MTL{1}$
obtained from these congruences, computed by computer exploration.
\begin{table}[ht]
    \centering
    \begin{tabular}{l|llllllll}
        Hopf algebra & \multicolumn{8}{c}{First dimensions} \\ \hline \hline
        $\MTN{1}^{\Equiv{Bx}}$ & 1 & 1 & 7 & 265 & 38051 \\
        $\MTN{1}^{\Equiv{Bl}}$ & 1 & 1 & 7 & 221 & 25789 \\
        $\MTN{1}^{\Equiv{S}}$ & 1 & 1 & 7 & 221 & 24243 \\
        $\MTN{1}^{\Equiv{P}}$ & 1 & 1 & 7 & 177 & 17339 \\
        $\MTN{1}^{\Equiv{H}}$ & 1 & 1 & 7 & 177 & 13887 \\
        $\MTN{1}^{\Equiv{T}}$ & 1 & 1 & 4 & 57 & 2306 \\ \hline
        $\MTL{1}^{\Equiv{Bx}}$ & 1 & 1 & 2 & 10 & 68 & 578 & 5782 & 65745 \\
        $\MTL{1}^{\Equiv{Bl}}$ & 1 & 1 & 2 & 9 & 53 & 390 & 3389 & 33881 \\
        $\MTL{1}^{\Equiv{S}}$ & 1 & 1 & 2 & 9 & 52 & 364 & 2918 & 26138 \\
        $\MTL{1}^{\Equiv{P}}$ & 1 & 1 & 2 & 8 & 41 & 266 & 1976 & 16569 \\
        $\MTL{1}^{\Equiv{H}}$ & 1 & 1 & 2 & 8 & 39 & 220 & 1396 & 9716 \\
        $\MTL{1}^{\Equiv{T}}$ & 1 & 1 & 1 & 3 & 11 & 43 & 191 & 939 \\
    \end{tabular}
    \bigskip

    \caption{First few dimensions of the Hopf subalgebras~$\MTN{1}^\equiv$
    and~$\MTL{1}^\equiv$, where~$\equiv$ is successively the Baxter, Bell,
    sylvester, plactic, hypoplactic, and total congruence.}
    \label{tab:Dim_Sous_AHC_Rel}
\end{table}
\medskip

\section{Alternating sign matrices} \label{sec:ASM}
Recall that an {\em alternating sign matrix}~\cite{MRR83},
\index{alternating sign matrix}%
or an {\em ASM} for short, of size~$n$ is a square matrix of order~$n$
with entries in the alphabet~$\{\Zero, \Plus, \Moins\}$ such that every
row and column starts and ends by~$\Plus$ and in every row and column,
the~$\Plus$ and the~$\Moins$ alternate. For instance,
\begin{equation}
    \delta :=
    \begin{Matrice}
        \Zero & \Plus & \Zero & \Zero & \Zero \\
        \Zero & \Zero & \Plus & \Zero & \Zero \\
        \Plus & \Moins & \Zero & \Zero & \Plus \\
        \Zero & \Plus & \Moins & \Plus & \Zero \\
        \Zero & \Zero & \Plus & \Zero & \Zero
    \end{Matrice}
\end{equation}
is an ASM of size~$5$.
\medskip

\subsection{Hopf algebra structure}
Let~$\delta$ be an ASM. We denote by~$M^\delta$ the matrix satisfying
\begin{equation}
    M^\delta_{ij} :=
    \begin{cases}
        \Un & \mbox{if } \delta_{ij} \in \{\Plus, \Moins\}, \\
        \Zero & \mbox{otherwise}.
    \end{cases}
\end{equation}
\medskip

For instance, with the ASM~$\delta$ defined above, we obtain
\begin{equation} \label{equ:Exemple_ASM}
    M^\delta =
    \begin{Matrice}
        \Zero & \Un & \Zero & \Zero & \Zero \\
        \Zero & \Zero & \Un & \Zero & \Zero \\
        \Un & \Un & \Zero & \Zero & \Un \\
        \Zero & \Un & \Un & \Un & \Zero \\
        \Zero & \Zero & \Un & \Zero & \Zero
    \end{Matrice}.
\end{equation}
\medskip

It is immediate that~$M^\delta$ is a $1$-packed matrix of the same size
than~$\delta$. Besides, observe that since the~$\Plus$ and the~$\Moins$
alternate in an ASM, by starting from  a $1$-packed matrix~$M$, there is
at most one ASM~$\delta$ such that~$M^\delta = M$.
\medskip

Let~$\ASM$
\index{Hopf algebra!$\ASM$}%
be the vector space spanned by the set of all ASMs. For any ASM $\delta$,
let us denote by $\FF_\delta$ the element $\FF_{M^\delta}$. Due to the
above observation, the family $\FF_\delta$, where $\delta$ are ASMs,
spans $\ASM$. Moreover, since the map $\FF_{\delta} \mapsto \FF_{M^\delta}$
is an injective morphism from $\ASM$ to $\MT{1}$, this family forms a
basis.
\medskip

The product and the coproduct of $\MT{1}$ induce the product and the
coproduct of $\ASM$. For example, we have
\begin{equation}
\label{exemple::prod::asm}
    \FF_{
    \begin{Matrice}
        \ZeroB & \PlusB & \ZeroB \\
        \PlusB & \MoinsB & \PlusB \\
        \ZeroB & \PlusB & \ZeroB
    \end{Matrice}}
    \cdot
    \FF_{
    \begin{Matrice}
        \PlusR
    \end{Matrice}}
    =
    \FF_{
    \begin{Matrice}
        \ZeroB & \PlusB & \ZeroB & \Zero \\
        \PlusB & \MoinsB & \PlusB & \Zero \\
        \ZeroB & \PlusB & \ZeroB & \Zero \\
        \Zero & \Zero & \Zero & \PlusR
    \end{Matrice}}
    +
    \FF_{
    \begin{Matrice}
        \ZeroB & \PlusB & \Zero & \ZeroB \\
        \PlusB & \MoinsB & \Zero & \PlusB \\
        \ZeroB & \PlusB & \Zero & \ZeroB \\
        \Zero & \Zero & \PlusR & \Zero
    \end{Matrice}}
    +
    \FF_{
    \begin{Matrice}
        \ZeroB & \Zero & \PlusB & \ZeroB \\
        \PlusB & \Zero & \MoinsB & \PlusB \\
        \ZeroB & \Zero & \PlusB & \ZeroB \\
        \Zero & \PlusR & \Zero & \Zero
    \end{Matrice}}
    +
    \FF_{\begin{Matrice}
        \Zero & \ZeroB & \PlusB & \ZeroB \\
        \Zero & \PlusB & \MoinsB & \PlusB \\
        \Zero & \ZeroB & \PlusB & \ZeroB \\
        \PlusR & \Zero & \Zero & \Zero
    \end{Matrice}},
\end{equation}
and
\begin{equation}
    \Delta \FF_{
    \begin{Matrice}
        \Zero & \Plus & \Zero & \Zero \\
        \Zero & \Zero & \Zero & \Plus \\
        \Plus & \Moins & \Plus & \Zero \\
        \Zero & \Plus & \Zero & \Zero
    \end{Matrice}}
    =
    \FF_{\emptyset}
    \otimes
    \FF_{
    \begin{Matrice}
        \Zero & \Plus & \Zero & \Zero \\
        \Zero & \Zero & \Zero & \Plus \\
        \Plus & \Moins & \Plus & \Zero \\
        \Zero & \Plus & \Zero & \Zero
    \end{Matrice}}
    +
    \FF_{
    \begin{Matrice}
        \Zero & \Plus & \Zero \\
        \Plus & \Moins & \Plus \\
        \Zero & \Plus & \Zero
    \end{Matrice}}
    \otimes
    \FF_{
    \begin{Matrice}
        \Plus
    \end{Matrice}}
    +
    \FF_{
    \begin{Matrice}
        \Zero & \Plus & \Zero & \Zero \\
        \Zero & \Zero & \Zero & \Plus \\
        \Plus & \Moins & \Plus & \Zero \\
        \Zero & \Plus & \Zero & \Zero
    \end{Matrice}}
    \otimes
    \FF_{\emptyset}.
\end{equation}
\medskip

\begin{Theoreme} \label{thm:AHC_ASM}
    The vector space~$\ASM$, endowed with the product and coproduct
    of~$\MT{1}$, forms a free, cofree, and self-dual bigraded Hopf
    algebra which admits a bidendriform bialgebra structure.
\end{Theoreme}
\begin{proof}
    Let $\delta_1$ and $\delta_2$ be two ASMs of respective sizes $n_1$
    and $n_2$ and let $M \in M^{\delta_1} \cshuffle M^{\delta_2}$. Let us
    denote by $M_1$ (resp. $M_2$) the matrix consisting in the first $n_1$
    (resp. last $n_2$) rows of $M$. By construction, $M_1$ contains columns
    coming from $\delta_1$ and some null columns. The relative order of
    columns of $M^{\delta_1}$ is the same as in $M_1$, \ie, the $i$th column
    of $M^{\delta_1}$ is the $i$th nonzero column of $M_1$. Hence, the
    rows of $M_1$ start and end with $\Plus$ and then $\Plus$ and $\Moins$
    alternate. Similarly, the same property is satisfied in $M_2$.
    Furthermore, the nonzero column of $M_1$ are followed by null columns
    of $M_2$ and the nonzero column of $M_2$ are preceded by null columns
    of $M_1$. Hence, the columns of $M$ start and end with $\Plus$ and
    $\Plus$ and $\Moins$ alternate. Thus $M$ is an ASM so that $\ASM$ is
    stable for the product of~$\MT{1}$.
    \smallskip

    Let $\delta$ be an ASM and $L \bullet R$ be a column decomposition of
    $M^\delta$. By Lemma~\ref{lem:Decomposition}, a column decomposition
    never splits a matrix by separating two nonzero entries on a same row.
    Then, the nonzero rows of $L$ and $R$ start and end with $\Plus$ and
    $\Plus$ and $\Moins$ alternate. Thus, $\Compr(L)$ and $\Compr(R)$ are
    ASMs and $\ASM$ is stable for the coproduct of~$\MT{1}$.
    \smallskip

    This shows that $\ASM$ is a Hopf subalgebra of~$\MT{1}$ and
    also that~$\ASM$ inherits from the bidendriform bialgebra
    structure of~$\MT{1}$ (see Theorem~\ref{thm:PMBidendriforme}).
    Finally, since~$\ASM$ admits a bidendriform bialgebra structure,
    by~\cite{Foi07}, it is free, cofree, and self-dual.
\end{proof}
\medskip

From now, we shall see~$\ASM$ as a simply graded Hopf algebra so that
the degree of any~$\FF_{\delta}$, where~$\delta$ is an ASM, is the
size of~$\delta$. The dimensions of~$\ASM$ form Sequence~\Sloane{A005130}
of~\cite{Slo} and the first few terms are
\begin{equation}
    1, \: 1, \: 2, \: 7, \: 42, \: 429, \: 7436, \: 218348, \: 10850216, \:
    911835460, \: 129534272700.
\end{equation}
\medskip

By using same arguments as those used in Section~\ref{subsec:Mult_Liberte},
one can build multiplicative bases of~$\ASM$ by setting, for any
ASM~$\delta$,
\begin{equation}
    \EE_{\delta} := \sum_{M^\delta \OrdMT M^{\delta'}} \FF_{\delta'}
    \qquad \mbox{and} \qquad
    \HH_{\delta} := \sum_{M^{\delta'} \OrdMT M^\delta} \FF_{\delta'}.
\end{equation}
This gives another way to prove the freeness of~$\ASM$ by using same
arguments as those of Theorem~\ref{thm:MT_Libre}. Hence, $\ASM$ is freely
generated by the elements $\EE_{\delta}$ (resp. $\HH_{\delta}$) where
the $\delta$ are ASMs such that the~$M^\delta$ are connected (resp.
anti-connected) $1$-packed matrices. The first few numbers of algebraic
generators of~$\ASM$ are
\begin{equation}
    0, \: 1, \: 1, \: 4, \: 29, \: 343, \: 6536, \: 202890, \: 10403135, \:
    889855638, \: 127697994191
\end{equation}
and the first few dimensions of totally primitive elements are
\begin{equation}
    0, \: 1, \: 0, \: 2, \: 20, \: 277, \: 5776, \: 188900, \: 9980698, \:
    868571406, \: 125895356788.
\end{equation}
\medskip

Moreover, since the transpose of an ASM is also an ASM, by
Proposition~\ref{prop:Autodualite_MT}, the map $\phi : \ASM \to \ASM^\star$
linearly defined for any ASM~$\delta$ by
\begin{equation}
    \phi\left(\FF_{\delta}\right) := \FF^\star_{{\delta}^T}
\end{equation}
is an isomorphism.
\medskip

\subsection{Alternating sign matrices statistics} \label{subsec:ASM_stats}
We recall here the definitions of some statistics on ASMs. Their
description passes through six-vertex configurations and osculating
paths, combinatorial objects in bijection with ASMs.
\medskip

The statistics discussed in this article have been already exploited in
the literature. For instance, in \cite{EKLP92}, the authors focused on these
statistics to understand the relationship between domino tilings of
Aztec diamonds and ASMs.
\medskip

\subsubsection{Six-vertex configurations}
A {\em six-vertex configuration} (see for example~\cite{bressoud99,BAX08}
for further information and references) of size~$n$ is a $n \times n$
square grid with oriented edges so that each vertex has two incoming and
two outcoming edges. There are six possible configurations for each vertex.
We consider here the six-vertex model with {\em domain wall boundary
conditions}~\cite{KO82}
\index{six-vertex configuration}%
{\em i.e.}, all horizontal (resp. vertical) edges on the boundary of
this model are oriented inwardly (resp. outwardly)
(see Figure~\ref{subfig:six_sommets}).
\medskip

The bijection~\cite{Kup96} between ASMs of size~$n$ and six-vertex
configurations of the same size consists in replacing each vertex
configuration by~$0$, $\Plus$, or~$\Moins$ according to the rules
described in Figure~\ref{fig::asm::vmb}.
\begin{figure}[ht]
    \centering
    \begin{tikzpicture}[scale=.6]
        \begin{scope}
            \pgfkeys{/confNE}
            \node at(0,-1.5){$0$};
            \node at(0,1.5){$\NE$};
        \end{scope}
        \begin{scope}[xshift=3cm]
            \pgfkeys{/confSW}
            \node at(0,-1.5){$0$};
            \node at(0,1.5){$\SW$};
        \end{scope}
        \begin{scope}[xshift=6cm]
            \pgfkeys{/confSE}
            \node at(0,-1.5){$0$};
            \node at(0,1.5){$\SE$};
        \end{scope}
        \begin{scope}[xshift=9cm]
            \pgfkeys{/confNW}
            \node at(0,-1.5){$0$};
            \node at(0,1.5){$\NW$};
        \end{scope}
        \begin{scope}[xshift=12cm]
            \pgfkeys{/confOI}
            \node at(0,-1.5){$\Plus$};
            \node at(0,1.5){$\OI$};
        \end{scope}
        \begin{scope}[xshift=15cm]
            \pgfkeys{/confIO}
            \node at(0,-1.5){$\Moins$};
            \node at(0,1.5){$\IO$};
        \end{scope}
    \end{tikzpicture}
    \caption{Correspondence between vertices of six-vertex configurations
    and entries of ASMs.}
    \label{fig::asm::vmb}
\end{figure}
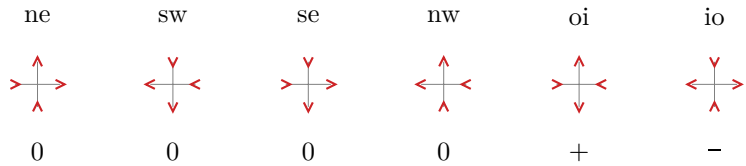
Reciprocally, to recover a six-vertex model from an ASM $\delta$, we
first replace each nonzero entry of $\delta$ by the corresponding vertex
configuration (see the last two configurations of Figure~\ref{fig::asm::vmb}).
Then, for each zero entry of $\delta$, we look at the sum $\ell$ (resp. $a$)
of the entries to the left (resp. above) of it and in the same row (resp.
column). By the alternating property of the ASMs, $\ell$ and $a$ belong
to $\{0, 1\}$. Now, set in $\delta$ the configuration
$\begin{tikzpicture}[scale=0.4]
    \pgfkeys{/orientations}
    \draw[gril] (-.5,0) grid (0.5,0);
    \draw[seta,shift={(-.5,0)}] \ei;
    \draw[seta,shift={(0.5,0)}] \ei;
\end{tikzpicture}$
(resp.
$\begin{tikzpicture}[scale=0.4]
    \pgfkeys{/orientations}
    \draw[gril] (-.5,0) grid (0.5,0);
    \draw[seta,shift={(-.5,0)}] \di;
    \draw[seta,shift={(0.5,0)}] \di;
\end{tikzpicture}$)
if $\ell = 1$ (resp. $\ell = 0$) together with the configuration
$\raisebox{-.5em}{\begin{tikzpicture}[scale=0.4]
    \pgfkeys{/orientations}
    \draw[gril] (0,-.5) grid (0,0.5);
    \draw[seta,shift={(0,-.5)}] \bi;
    \draw[seta,shift={(0,0.5)}] \bi;
\end{tikzpicture}}$ (resp.
$\raisebox{-.5em}{\begin{tikzpicture}[scale=0.4]
    \pgfkeys{/orientations}
    \draw[gril] (0,-.5) grid (0,0.5);
    \draw[seta,shift={(0,-.5)}] \ci;
    \draw[seta,shift={(0,0.5)}] \ci;
\end{tikzpicture}}$)
if $a = 1$ (resp. $a = 0$). Figures~\ref{subfig:asm}
and~\ref{subfig:six_sommets} form an example.
\begin{figure}[ht]
    \centering
    \subfigure[An ASM $\delta$.]{
    \qquad
    $\begin{Matrice}
        0 & 0 & \Plus & 0 & 0 & 0 \\
        \Plus & 0 & 0 & 0 & 0 & 0 \\
        0 & 0 & 0 & 0 & \Plus & 0 \\
        0 & \Plus & 0 & 0 & \Moins & \Plus \\
        0 & 0 & 0 & \Plus & 0 & 0 \\
        0 & 0 & 0 & 0 & \Plus & 0
    \end{Matrice}$
    \qquad
    \label{subfig:asm}}
    \qquad\qquad\qquad
    \subfigure[The corner sum matrix in bijection with~$\delta$.]{
    \qquad
    $\begin{Matrice}
        6 \, & 5 \, & 4 \, & 3 \, & 2 \, & 1 \, \\
        5 & 4 & 3 & 3 & 2 & 1 \\
        4 & 4 & 3 & 3 & 2 & 1 \\
        3 & 3 & 2 & 2 & 1 & 1 \\
        2 & 2 & 2 & 2 & 1 & 0 \\
        1 & 1 & 1 & 1 & 1 & 0
    \end{Matrice}$
    \qquad
    \label{subfig:somme_en_coins}}
    \\
    \subfigure[The six-vertex configuration in bijection with~$\delta$.]{
    \qquad
    \begin{tikzpicture}[scale=.5]
        \pgfkeys{/orientations}
        \draw[gril] (-.5,-.5) grid (5.5,5.5);
        \draw[seta,shift={(0,-.5)}] \bi;
        \draw[seta,shift={(1,-.5)}] \bi;
        \draw[seta,shift={(2,-.5)}] \bi;
        \draw[seta,shift={(3,-.5)}] \bi;
        \draw[seta,shift={(4,-.5)}] \bi;
        \draw[seta,shift={(5,-.5)}] \bi;
        \draw[seta,shift={(0,0.5)}] \bi;
        \draw[seta,shift={(1,0.5)}] \bi;
        \draw[seta,shift={(2,0.5)}] \bi;
        \draw[seta,shift={(3,0.5)}] \bi;
        \draw[seta,shift={(4,0.5)}] \ci;
        \draw[seta,shift={(5,0.5)}] \bi;
        \draw[seta,shift={(0,1.5)}] \bi;
        \draw[seta,shift={(1,1.5)}] \bi;
        \draw[seta,shift={(2,1.5)}] \bi;
        \draw[seta,shift={(3,1.5)}] \ci;
        \draw[seta,shift={(4,1.5)}] \ci;
        \draw[seta,shift={(5,1.5)}] \bi;
        \draw[seta,shift={(0,2.5)}] \bi;
        \draw[seta,shift={(1,2.5)}] \ci;
        \draw[seta,shift={(2,2.5)}] \bi;
        \draw[seta,shift={(3,2.5)}] \ci;
        \draw[seta,shift={(4,2.5)}] \bi;
        \draw[seta,shift={(5,2.5)}] \ci;
        \draw[seta,shift={(0,3.5)}] \bi;
        \draw[seta,shift={(1,3.5)}] \ci;
        \draw[seta,shift={(2,3.5)}] \bi;
        \draw[seta,shift={(3,3.5)}] \ci;
        \draw[seta,shift={(4,3.5)}] \ci;
        \draw[seta,shift={(5,3.5)}] \ci;
        \draw[seta,shift={(0,4.5)}] \ci;
        \draw[seta,shift={(1,4.5)}] \ci;
        \draw[seta,shift={(2,4.5)}] \bi;
        \draw[seta,shift={(3,4.5)}] \ci;
        \draw[seta,shift={(4,4.5)}] \ci;
        \draw[seta,shift={(5,4.5)}] \ci;
        \draw[seta,shift={(0,5.5)}] \ci;
        \draw[seta,shift={(1,5.5)}] \ci;
        \draw[seta,shift={(2,5.5)}] \ci;
        \draw[seta,shift={(3,5.5)}] \ci;
        \draw[seta,shift={(4,5.5)}] \ci;
        \draw[seta,shift={(5,5.5)}] \ci;
        \draw[seta,shift={(-.5,0)}] \di;
        \draw[seta,shift={(-.5,1)}] \di;
        \draw[seta,shift={(-.5,2)}] \di;
        \draw[seta,shift={(-.5,3)}] \di;
        \draw[seta,shift={(-.5,4)}] \di;
        \draw[seta,shift={(-.5,5)}] \di;
        \draw[seta,shift={(0.5,0)}] \di;
        \draw[seta,shift={(0.5,1)}] \di;
        \draw[seta,shift={(0.5,2)}] \di;
        \draw[seta,shift={(0.5,3)}] \di;
        \draw[seta,shift={(0.5,4)}] \ei;
        \draw[seta,shift={(0.5,5)}] \di;
        \draw[seta,shift={(1.5,0)}] \di;
        \draw[seta,shift={(1.5,1)}] \di;
        \draw[seta,shift={(1.5,2)}] \ei;
        \draw[seta,shift={(1.5,3)}] \di;
        \draw[seta,shift={(1.5,4)}] \ei;
        \draw[seta,shift={(1.5,5)}] \di;
        \draw[seta,shift={(2.5,0)}] \di;
        \draw[seta,shift={(2.5,1)}] \di;
        \draw[seta,shift={(2.5,2)}] \ei;
        \draw[seta,shift={(2.5,3)}] \di;
        \draw[seta,shift={(2.5,4)}] \ei;
        \draw[seta,shift={(2.5,5)}] \ei;
        \draw[seta,shift={(3.5,0)}] \di;
        \draw[seta,shift={(3.5,1)}] \ei;
        \draw[seta,shift={(3.5,2)}] \ei;
        \draw[seta,shift={(3.5,3)}] \di;
        \draw[seta,shift={(3.5,4)}] \ei;
        \draw[seta,shift={(3.5,5)}] \ei;
        \draw[seta,shift={(4.5,0)}] \ei;
        \draw[seta,shift={(4.5,1)}] \ei;
        \draw[seta,shift={(4.5,2)}] \di;
        \draw[seta,shift={(4.5,3)}] \ei;
        \draw[seta,shift={(4.5,4)}] \ei;
        \draw[seta,shift={(4.5,5)}] \ei;
        \draw[seta,shift={(5.5,0)}] \ei;
        \draw[seta,shift={(5.5,1)}] \ei;
        \draw[seta,shift={(5.5,2)}] \ei;
        \draw[seta,shift={(5.5,3)}] \ei;
        \draw[seta,shift={(5.5,4)}] \ei;
        \draw[seta,shift={(5.5,5)}] \ei;
    \end{tikzpicture}
    \qquad
    \label{subfig:six_sommets}}
    \qquad\qquad\qquad
    \subfigure[The set of osculating paths in bijection with~$\delta$.]{
    \qquad
    \begin{tikzpicture}[scale=.5]
        \draw[gril] (-.5,-.5) grid (5.5,5.5);
        \draw[path] (-.5,5) -- (0,5) -- (0,5.5);
        \draw[path] (-.5,4)--(0,4)--(0,5)--(1,5)--(1,5.5);
        \draw[path] (-.5,3)--(1,3)--(1,4)--(1,5)--(2,5) --(2,5.5);
        \draw[path] (-.5,2)--(1,2)--(1,3)--(3,3)--(3,5.5);
        \draw[path] (-.5,1)--(3,1)--(3,3)--(4,3)--(4,5.5);
        \draw[path] (-.5,0)--(4,0)--(4,2)--(5,2)--(5,5.5);
    \end{tikzpicture}
    \qquad
    \label{subfig:chemins}}
    \caption{Four objects in correspondence: ASMs, six-vertex
    configurations, corner sum matrices, and sets of osculating paths.}
    \label{fig:objets_en_bijection_avec_asm}
\end{figure}
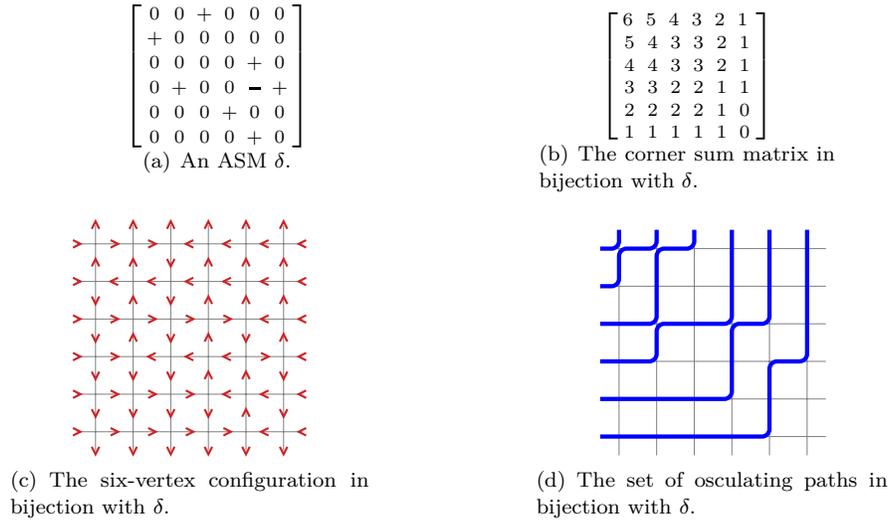
\medskip

\subsubsection{Statistics on six-vertex configurations and ASMs}
Let us denote by~$\NE(\delta)$
\index{statistic!$\NE$}%
(resp. $\SW(\delta)$,
\index{statistic!$\SW$}%
$\SE(\delta)$,
\index{statistic!$\SE$}%
$\NW(\delta)$,
\index{statistic!$\NW$}%
$\OI(\delta)$,
\index{statistic!$\OI$}%
$\IO(\delta)$)
\index{statistic!$\IO$}%
the number of vertices~$\NE$ (resp. $\SW$, $\SE$, $\NW$, $\OI$, $\IO$)
in the six-vertex configuration in bijection with the ASM~$\delta$. Let
$\ZZ := \{\SE,\NW,\SW,\NE\}$ be the set of
\index{set!$\ZZ$}%
the statistics counting the four configurations of $0$ and
$\NN := \{\IO,\OI\}$
\index{set!$\NN$}%
be the set of the statistics counting the two nonzero configurations.
\medskip

\subsubsection{Sets of osculating paths}
These statistics share some symmetries that are naturally interpreted on
sets of osculating paths (see~\cite{BMH95,BEHREND08}). Let $\Pi$ be a
$n \times n$ square of lattice points with rows (resp. columns) labelled
from $1$ to $n$ from top to bottom (resp. from left to right). A
{\em lattice path} on $\Pi$
\index{lattice path}%
is a sequence $(v_0, v_1, \dots, v_r)$ of vertices of $\Pi$ such that
$v_i - v_{i-1} \in \{(1, 0), (0, -1)\}$ for all $i \in [r]$. A
{\em set of osculating paths} on $\Pi$
\index{set of osculating paths}%
is a set of lattice paths on $\Pi$ in which different paths do not cross
each other but can share some vertices.
\medskip

We can associate a set of osculating paths with any six-vertex configuration
according to the rules described in Figure~\ref{fig::vmb::op}. Domain
boundary conditions ensure that each path starts at $(i, 1)$ and ends at
$(1, j)$ for some $i$ and $j$.
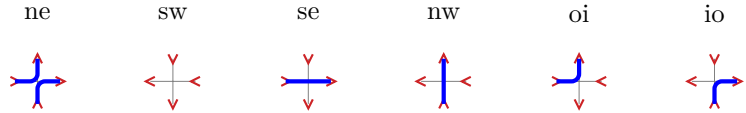
\begin{figure}[ht]
    \centering
    \begin{tikzpicture}[scale=.6]
        \begin{scope}
            \pgfkeys{/confNE}
            \node at(0,1.5){$\NE$};
            \draw[path] (-.5,0) -- (0,0) -- (0,.5);
            \draw[path] (0,-.5) -- (0,0) -- (.5,0);
        \end{scope}
        \begin{scope}[xshift=3cm]
            \pgfkeys{/confSW}
            \node at(0,1.5){$\SW$};
        \end{scope}
        \begin{scope}[xshift=6cm]
            \pgfkeys{/confSE}
            \draw[path] (-.5,0) -- (.5,0);
            \node at(0,1.5){$\SE$};
        \end{scope}
        \begin{scope}[xshift=9cm]
            \pgfkeys{/confNW}
            \draw[path] (0,-.5) -- (0,.5);
            \node at(0,1.5){$\NW$};
        \end{scope}
        \begin{scope}[xshift=12cm]
            \pgfkeys{/confOI}
            \node at(0,1.5){$\OI$};
            \draw[path] (-.5,0) -- (0,0) -- (0,.5);
        \end{scope}
        \begin{scope}[xshift=15cm]
            \pgfkeys{/confIO}
            \draw[path] (0,-.5) -- (0,0) -- (.5,0);
            \node at(0,1.5){$\IO$};
        \end{scope}
    \end{tikzpicture}
    \caption{Correspondence between vertices of six-vertex configurations
    and osculating paths.}
    \label{fig::vmb::op}
\end{figure}
Figures~\ref{subfig:six_sommets} and~\ref{subfig:chemins} form an example.
\medskip

The direct interpretation of ASMs in terms of sets of osculating paths
is directly based upon the corner sum matrix introduced in~\cite{RR86}:
given an $n \times n$ matrix $M$, the {\em corner sum matrix}
\index{corner sum matrix}%
$\bar{M}$ of $M$ is defined by
\begin{equation}
    \bar{M}_{i, j} :=
    \sum_{\substack{i \leq k \leq n \\ j \leq \ell \leq n}}
    M_{k, \ell}.
\end{equation}
Figures~\ref{subfig:asm} and~\ref{subfig:somme_en_coins} form an
example. We associate with any ASM $\delta$ of size $n$ the set of
osculating paths described as follows. By regarding $\bar{\delta}$ as a
$(n+1) \times (n+1)$ square of lattice points, we draw on it the south
and the east boundaries of the areas consisting in a same value greater than
zero. This produces a set of~$n$ osculating paths.
Figures~\ref{subfig:somme_en_coins} and~\ref{subfig:chemins} form an
example.
Since the steps of the paths in the first row (resp. column) are, by
construction, always vertical (resp. horizontal), this set of osculating
paths can be seen without loss of information on the $n \times n$ square
of lattice points.
\medskip

The different $2\times 2$ submatrices configurations in a corner sum
matrix of an ASM are exactly
\begin{equation}
    \begin{split}
        \begin{tikzpicture}[scale=.5]
            \draw[gril,step=1.5cm] (0,0) grid (3,3);
            \node at(.75, 2.25)  {\footnotesize \begin{math}2 + a\end{math}};
            \node at(.75, .75)   {\footnotesize \begin{math}1 + a\end{math}};
            \node at(2.25, 2.25) {\footnotesize \begin{math}1 + a\end{math}};
            \node at(2.25, .75)  {\footnotesize \begin{math}a\end{math}};
            \draw[path] (0, 1.5) -- ++(1.5, 0) -- ++(0, 1.5);
            \draw[path] (1.5, 0) -- ++(0, 1.5) -- ++(1.5, 0);
        \end{tikzpicture}
    \end{split}\,,
    \quad
    \begin{split}
        \begin{tikzpicture}[scale=.5]
            \draw[gril,step=1.5cm] (0,0) grid (3,3);
            \node at(.75,2.25){\footnotesize \begin{math}a\end{math}};
            \node at(.75,.75){\footnotesize \begin{math}a\end{math}};
            \node at(2.25,2.25){\footnotesize \begin{math}a\end{math}};
            \node at(2.25,.75){\footnotesize \begin{math}a\end{math}};
        \end{tikzpicture}
    \end{split}\,,
    \quad
    \begin{split}
        \begin{tikzpicture}[scale=.5]
            \draw[gril,step=1.5cm] (0,0) grid (3,3);
            \node at(.75,2.25){\footnotesize \begin{math}1 + a\end{math}};
            \node at(.75,.75){\footnotesize \begin{math}a\end{math}};
            \node at(2.25,2.25){\footnotesize \begin{math}1 + a\end{math}};
            \node at(2.25,.75){\footnotesize \begin{math}a\end{math}};
            \draw[path] (0, 1.5) -- ++(3, 0);
        \end{tikzpicture}
    \end{split}\,,
    \quad
    \begin{split}
        \begin{tikzpicture}[scale=.5]
            \draw[gril,step=1.5cm] (0,0) grid (3,3);
            \node at(.75,2.25){\footnotesize \begin{math}1 + a\end{math}};
            \node at(.75,.75){\footnotesize \begin{math}1 + a\end{math}};
            \node at(2.25,2.25){\footnotesize \begin{math}a\end{math}};
            \node at(2.25,.75){\footnotesize \begin{math}a\end{math}};
            \draw[path] (1.5, 0) -- ++(0, 3);
        \end{tikzpicture}
    \end{split}\,,
    \quad
    \begin{split}
        \begin{tikzpicture}[scale=.5]
            \draw[gril,step=1.5cm] (0,0) grid (3,3);
            \node at(.75,2.25){\footnotesize \begin{math}1 + a\end{math}};
            \node at(.75,.75){\footnotesize \begin{math}a\end{math}};
            \node at(2.25,2.25){\footnotesize \begin{math}a\end{math}};
            \node at(2.25,.75){\footnotesize \begin{math}a\end{math}};
            \draw[path] (0, 1.5) -- ++(1.5, 0) -- ++(0, 1.5);
        \end{tikzpicture}
    \end{split}\,,
    \quad
    \begin{split}
        \begin{tikzpicture}[scale=.5]
            \draw[gril,step=1.5cm] (0,0) grid (3,3);
            \node at(.75,2.25){\footnotesize \begin{math}1 + a\end{math}};
            \node at(.75,.75){\footnotesize \begin{math}1 + a\end{math}};
            \node at(2.25,2.25){\footnotesize \begin{math}1 + a\end{math}};
            \node at(2.25,.75){\footnotesize \begin{math}a\end{math}};
            \draw[path] (1.5, 0) -- ++(0, 1.5) -- ++(1.5, 0);
        \end{tikzpicture}
    \end{split}\,.
\end{equation}
They obviously describe the correspondence given in Figure~\ref{fig::vmb::op}.
\medskip

\subsubsection{Symmetries between ASMs statistics}

\begin{Proposition} \label{prop::symmetrie::6V}
    Let~$\delta$ be an ASM of size~$n$. Then,
    \begin{equation} \label{eq::symmetrie::6V}
        \SE(\delta) = \NW(\delta), \quad
        \NE(\delta) = \SW(\delta), \quad
        \OI(\delta) = \IO(\delta) + n.
    \end{equation}
\end{Proposition}
\begin{proof}
    Consider the set of osculating paths $P$ associated with~$\delta$
    and the correspondence between the vertices of six-vertex
    configurations and osculating paths (see Figure~\ref{fig::vmb::op}).
    The first identity of~\eqref{eq::symmetrie::6V} is equivalent to say
    that there are in $P$ as many horizontal steps as vertical steps.
    Since in $P$, any osculating path connects the $i$th vertex of the
    first column with the $i$th vertex of the first row of the grid, for
    any~$i \in [n]$, this property holds. Whence the first identity.
    \smallskip

    Consider now the ASM $\delta'$ where, for any $i \in [n]$, the
    $i$th row of $\delta'$ is the $(n - i + 1)$st row of $\delta$. Then, in
    the six-vertex configuration in bijection with $\delta'$, all $\NE$
    (resp. $\SW$) configurations come from $\SE$ (resp. $\NW$) configurations
    of the six-vertex configuration in bijection with $\delta$. Then, the
    second identity follows from the first one.
    \smallskip

    The last identity follows immediately from the fact that any row and
    column of $\delta$ starts and ends by $\Plus$, and the $\Plus$ and
    the $\Moins$ alternate.
\end{proof}
\medskip

\subsection{Algebraic interpretation of some statistics}
\label{subsec:ASM_stats_alg}
We provide algebraic interpretations of the statistics on ASMs recalled
in the previous section by using the Hopf algebra~$\ASM$. To be more
precise, we study the algebraic quotients of~$\ASM$ by equivalence
relations defined via ASMs statistics.
\medskip

\subsubsection{Maps from $\ASM$ to $q$-rational functions}
Let us recall the following notations in the algebra~$\K(q)$ of
$q$-rational functions:
\begin{equation}
    [n]_q := 1 + q + \dots + q^{n - 1}, \quad n \geq 1,
\end{equation}
\begin{equation}
    [0]_q! := 1, \qquad
    [n]_q! := [1]_q [2]_q \dots [n]_q, \quad n \geq 1,
\end{equation}
\begin{equation}
\label{eq::gaussian::pol}
    \genfrac{[}{]}{0pt}{1}{n_1 + n_2}{n_1, n_2}_q
    :=
    \frac{[n_1 + n_2]_q!} {[n_1]_q! \: [n_2]_q!},
    \quad n_1, n_2 \geq 0.
\end{equation}
\medskip

\begin{Lemme} \label{lem:io}
    Let $\delta$, $\delta_1$, and $\delta_2$ be three ASMs  such that
    $M^\delta \in M^{\delta_1} \cshuffle M^{\delta_2}$. Then, for any
    $s \in \NN$,
    \begin{equation}
        s(\delta) = s(\delta_1) + s(\delta_2).
    \end{equation}
\end{Lemme}
\begin{proof}
    The two statistics $\OI$ and $\IO$ of $\NN$, respectively count the
    number of entries $\Plus$ and $\Moins$ in ASMs. This result follows
    from the fact that the shifted shuffle of packed matrices does not
    add nor remove nonzero entries and the fact that any nonzero entry
    encoding a $\Plus$ (resp. $\Moins$) in the operands $M^{\delta_1}$
    and $M^{\delta_2}$ gives rise to a $\Plus$ (resp. $\Moins$) in
    $M^{\delta}$.
\end{proof}
\medskip

Here is the product~\eqref{exemple::prod::asm} in~$\ASM$, seen on
six-vertex configurations, where boldfaced vertices are of kind~$\IO$.
\begin{equation} \label{exemple::prod::6V::IO}
    \FF_{
    \begin{tikzpicture}[scale=.4]
        \pgfkeys{/orientations}
        \draw[gril] (.5,.5) grid (3.5,3.5);
        \draw[seta,shift={(1,0.5)}] \bi;
        \draw[seta,shift={(2,0.5)}] \bi;
        \draw[seta,shift={(3,0.5)}] \bi;
        \draw[seta,shift={(1,3.5)}] \ci;
        \draw[seta,shift={(2,3.5)}] \ci;
        \draw[seta,shift={(3,3.5)}] \ci;
        \draw[seta,shift={(3.5,1)}] \ei;
        \draw[seta,shift={(3.5,2)}] \ei;
        \draw[seta,shift={(3.5,3)}] \ei;
        \draw[seta,shift={(.5,1)}] \di;
        \draw[seta,shift={(.5,2)}] \di;
        \draw[seta,shift={(.5,3)}] \di;
        \draw[seta,shift={(1.5,1)}] \di;
        \draw[seta,shift={(2.5,1)}] \ei;
        \draw[seta,shift={(1.5,2)}] \ei;
        \draw[seta,shift={(2.5,2)}] \di;
        \draw[seta,shift={(1.5,3)}] \di;
        \draw[seta,shift={(2.5,3)}] \ei;
        \node at (2,2) {\tiny\io};
        \draw[seta,shift={(1,1.5)}] \bi;
        \draw[seta,shift={(1,2.5)}] \bi;
        \draw[seta,shift={(2,1.5)}] \ci;
        \draw[seta,shift={(2,2.5)}] \bi;
        \draw[seta,shift={(3,1.5)}] \bi;
        \draw[seta,shift={(3,2.5)}] \ci;
    \end{tikzpicture}}
    \cdot
    \FF_{
    \begin{tikzpicture}[scale=.4]
        \pgfkeys{/orientations}
        \draw[gril] (.5,.5) grid (1.5,1.5);
        \draw[seta,shift={(1,0.5)}] \bi;
        \draw[seta,shift={(1,1.5)}] \ci;
        \draw[seta,shift={(1.5,1)}] \ei;
        \draw[seta,shift={(.5,1)}] \di;
    \end{tikzpicture}}
    =
    \FF_{
    \begin{tikzpicture}[scale=.4]
        \pgfkeys{/orientations}
        \draw[gril] (.5,.5) grid (4.5,4.5);
        \draw[seta,shift={(1,0.5)}] \bi;
        \draw[seta,shift={(2,0.5)}] \bi;
        \draw[seta,shift={(3,0.5)}] \bi;
        \draw[seta,shift={(4,0.5)}] \bi;
        \draw[seta,shift={(1,4.5)}] \ci;
        \draw[seta,shift={(2,4.5)}] \ci;
        \draw[seta,shift={(3,4.5)}] \ci;
        \draw[seta,shift={(4,4.5)}] \ci;
        \draw[seta,shift={(4.5,1)}] \ei;
        \draw[seta,shift={(4.5,2)}] \ei;
        \draw[seta,shift={(4.5,3)}] \ei;
        \draw[seta,shift={(4.5,4)}] \ei;
        \draw[seta,shift={(.5,1)}] \di;
        \draw[seta,shift={(.5,2)}] \di;
        \draw[seta,shift={(.5,3)}] \di;
        \draw[seta,shift={(.5,4)}] \di;
        \draw[seta,shift={(1.5,1)}] \di;
        \draw[seta,shift={(2.5,1)}] \di;
        \draw[seta,shift={(3.5,1)}] \di;
        \draw[seta,shift={(1.5,2)}] \di;
        \draw[seta,shift={(2.5,2)}] \ei;
        \draw[seta,shift={(3.5,2)}] \ei;
        \node at (2,3) {\io};
        \draw[seta,shift={(1.5,3)}] \ei;
        \draw[seta,shift={(2.5,3)}] \di;
        \draw[seta,shift={(3.5,3)}] \ei;
        \draw[seta,shift={(1.5,4)}] \di;
        \draw[seta,shift={(2.5,4)}] \ei;
        \draw[seta,shift={(3.5,4)}] \ei;
        \draw[seta,shift={(1,1.5)}] \bi;
        \draw[seta,shift={(1,2.5)}] \bi;
        \draw[seta,shift={(1,3.5)}] \ci;
        \draw[seta,shift={(2,1.5)}] \bi;
        \draw[seta,shift={(2,2.5)}] \ci;
        \draw[seta,shift={(2,3.5)}] \bi;
        \draw[seta,shift={(3,1.5)}] \bi;
        \draw[seta,shift={(3,2.5)}] \bi;
        \draw[seta,shift={(3,3.5)}] \ci;
        \draw[seta,shift={(4,1.5)}] \ci;
        \draw[seta,shift={(4,2.5)}] \ci;
        \draw[seta,shift={(4,3.5)}] \ci;
    \end{tikzpicture}}
    +
    \FF_{
    \begin{tikzpicture}[scale=.4]
        \pgfkeys{/orientations}
        \draw[gril] (.5,.5) grid (4.5,4.5);
        \draw[seta,shift={(1,0.5)}] \bi;
        \draw[seta,shift={(2,0.5)}] \bi;
        \draw[seta,shift={(3,0.5)}] \bi;
        \draw[seta,shift={(4,0.5)}] \bi;
        \draw[seta,shift={(1,4.5)}] \ci;
        \draw[seta,shift={(2,4.5)}] \ci;
        \draw[seta,shift={(3,4.5)}] \ci;
        \draw[seta,shift={(4,4.5)}] \ci;
        \draw[seta,shift={(4.5,1)}] \ei;
        \draw[seta,shift={(4.5,2)}] \ei;
        \draw[seta,shift={(4.5,3)}] \ei;
        \draw[seta,shift={(4.5,4)}] \ei;
        \draw[seta,shift={(.5,1)}] \di;
        \draw[seta,shift={(.5,2)}] \di;
        \draw[seta,shift={(.5,3)}] \di;
        \draw[seta,shift={(.5,4)}] \di;
        \draw[seta,shift={(1.5,1)}] \di;
        \draw[seta,shift={(2.5,1)}] \di;
        \draw[seta,shift={(3.5,1)}] \ei;
        \draw[seta,shift={(1.5,2)}] \di;
        \draw[seta,shift={(2.5,2)}] \ei;
        \draw[seta,shift={(3.5,2)}] \ei;
        \node at (2,3) {\io};
        \draw[seta,shift={(1.5,3)}] \ei;
        \draw[seta,shift={(2.5,3)}] \di;
        \draw[seta,shift={(3.5,3)}] \di;
        \draw[seta,shift={(1.5,4)}] \di;
        \draw[seta,shift={(2.5,4)}] \ei;
        \draw[seta,shift={(3.5,4)}] \ei;
        \draw[seta,shift={(1,1.5)}] \bi;
        \draw[seta,shift={(1,2.5)}] \bi;
        \draw[seta,shift={(1,3.5)}] \ci;
        \draw[seta,shift={(2,1.5)}] \bi;
        \draw[seta,shift={(2,2.5)}] \ci;
        \draw[seta,shift={(2,3.5)}] \bi;
        \draw[seta,shift={(3,1.5)}] \ci;
        \draw[seta,shift={(3,2.5)}] \ci;
        \draw[seta,shift={(3,3.5)}] \ci;
        \draw[seta,shift={(4,1.5)}] \bi;
        \draw[seta,shift={(4,2.5)}] \bi;
        \draw[seta,shift={(4,3.5)}] \ci;
    \end{tikzpicture}}
    +
    \FF_{
    \begin{tikzpicture}[scale=.4]
        \pgfkeys{/orientations}
        \draw[gril] (.5,.5) grid (4.5,4.5);
        \draw[seta,shift={(1,0.5)}] \bi;
        \draw[seta,shift={(2,0.5)}] \bi;
        \draw[seta,shift={(3,0.5)}] \bi;
        \draw[seta,shift={(4,0.5)}] \bi;
        \draw[seta,shift={(1,4.5)}] \ci;
        \draw[seta,shift={(2,4.5)}] \ci;
        \draw[seta,shift={(3,4.5)}] \ci;
        \draw[seta,shift={(4,4.5)}] \ci;
        \draw[seta,shift={(4.5,1)}] \ei;
        \draw[seta,shift={(4.5,2)}] \ei;
        \draw[seta,shift={(4.5,3)}] \ei;
        \draw[seta,shift={(4.5,4)}] \ei;
        \draw[seta,shift={(.5,1)}] \di;
        \draw[seta,shift={(.5,2)}] \di;
        \draw[seta,shift={(.5,3)}] \di;
        \draw[seta,shift={(.5,4)}] \di;
        \draw[seta,shift={(1.5,1)}] \di;
        \draw[seta,shift={(2.5,1)}] \ei;
        \draw[seta,shift={(3.5,1)}] \ei;
        \draw[seta,shift={(1.5,2)}] \di;
        \draw[seta,shift={(2.5,2)}] \di;
        \draw[seta,shift={(3.5,2)}] \ei;
        \draw[seta,shift={(1.5,3)}] \ei;
        \draw[seta,shift={(2.5,3)}] \ei;
        \draw[seta,shift={(3.5,3)}] \di;
        \node at (3,3) {\io};
        \draw[seta,shift={(1.5,4)}] \di;
        \draw[seta,shift={(2.5,4)}] \di;
        \draw[seta,shift={(3.5,4)}] \ei;
        \draw[seta,shift={(1,1.5)}] \bi;
        \draw[seta,shift={(1,2.5)}] \bi;
        \draw[seta,shift={(1,3.5)}] \ci;
        \draw[seta,shift={(2,1.5)}] \ci;
        \draw[seta,shift={(2,2.5)}] \ci;
        \draw[seta,shift={(2,3.5)}] \ci;
        \draw[seta,shift={(3,1.5)}] \bi;
        \draw[seta,shift={(3,2.5)}] \ci;
        \draw[seta,shift={(3,3.5)}] \bi;
        \draw[seta,shift={(4,1.5)}] \bi;
        \draw[seta,shift={(4,2.5)}] \bi;
        \draw[seta,shift={(4,3.5)}] \ci;
    \end{tikzpicture}}
    +
    \FF_{
    \begin{tikzpicture}[scale=.4]
        \pgfkeys{/orientations}
        \draw[gril] (.5,.5) grid (4.5,4.5);
        \draw[seta,shift={(1,0.5)}] \bi;
        \draw[seta,shift={(2,0.5)}] \bi;
        \draw[seta,shift={(3,0.5)}] \bi;
        \draw[seta,shift={(4,0.5)}] \bi;
        \draw[seta,shift={(1,4.5)}] \ci;
        \draw[seta,shift={(2,4.5)}] \ci;
        \draw[seta,shift={(3,4.5)}] \ci;
        \draw[seta,shift={(4,4.5)}] \ci;
        \draw[seta,shift={(4.5,1)}] \ei;
        \draw[seta,shift={(4.5,2)}] \ei;
        \draw[seta,shift={(4.5,3)}] \ei;
        \draw[seta,shift={(4.5,4)}] \ei;
        \draw[seta,shift={(.5,1)}] \di;
        \draw[seta,shift={(.5,2)}] \di;
        \draw[seta,shift={(.5,3)}] \di;
        \draw[seta,shift={(.5,4)}] \di;
        \draw[seta,shift={(1.5,1)}] \ei;
        \draw[seta,shift={(2.5,1)}] \ei;
        \draw[seta,shift={(3.5,1)}] \ei;
        \draw[seta,shift={(1.5,2)}] \di;
        \draw[seta,shift={(2.5,2)}] \di;
        \draw[seta,shift={(3.5,2)}] \ei;
        \draw[seta,shift={(1.5,3)}] \di;
        \draw[seta,shift={(2.5,3)}] \ei;
        \draw[seta,shift={(3.5,3)}] \di;
        \draw[seta,shift={(1.5,4)}] \di;
        \draw[seta,shift={(2.5,4)}] \di;
        \draw[seta,shift={(3.5,4)}] \ei;
        \node at (3,3) {\io};
        \draw[seta,shift={(1,1.5)}] \ci;
        \draw[seta,shift={(1,2.5)}] \ci;
        \draw[seta,shift={(1,3.5)}] \ci;
        \draw[seta,shift={(2,1.5)}] \bi;
        \draw[seta,shift={(2,2.5)}] \bi;
        \draw[seta,shift={(2,3.5)}] \ci;
        \draw[seta,shift={(3,1.5)}] \bi;
        \draw[seta,shift={(3,2.5)}] \ci;
        \draw[seta,shift={(3,3.5)}] \bi;
        \draw[seta,shift={(4,1.5)}] \bi;
        \draw[seta,shift={(4,2.5)}] \bi;
        \draw[seta,shift={(4,3.5)}] \bi;
    \end{tikzpicture}}
\end{equation}
\medskip

\begin{Proposition} \label{prop:OI_IO_Morphismes}
    The map~$\phi_{s} : \ASM \to \K(q)$ linearly defined, for any
    $s \in \NN$ and any ASM~$\delta$ of size~$n$ by
    \begin{equation}
    \label{eq::morphisme::io::oi}
        \phi_{s}\left(\FF_{\delta}\right) := \frac{q^{s(\delta)}}{n!}
    \end{equation}
    is an algebra morphism.
\end{Proposition}
\begin{proof}
    This result follows immediately from Lemma~\ref{lem:io} and the fact
    that the product of two matrices of sizes~$n_1$ and~$n_2$ in~$\ASM$
    over the fundamental basis contains~$\binom{n_1 + n_2}{n_1}$ terms.
\end{proof}
\medskip

\begin{Lemme}  \label{lem:nw}
    Let $\delta$, $\delta_1$, and $\delta_2$ be three ASMs  such that
    $M^\delta \in M^{\delta_1} \cshuffle M^{\delta_2}$. Let $m$ be the
    size of $\delta_2$ (resp. $\delta_1$) and $\{k_1 < k_2 < \dots < k_m\}$
    be the set of the indices of the columns of $M^\delta$ coming from
    $M^{\delta_2}$ (resp. $M^{\delta_1}$). Then, for any $s \in \{\NW, \SE\}$
    (resp. $s \in \{\SW, \NE\}$),
    \begin{equation} \label{eq::0NW::shuffle}
        s(\delta) =
        s(\delta_1) + s(\delta_2) +
        \sum_{1 \leq j \leq m} (k_j-j).
    \end{equation}
\end{Lemme}
\begin{proof}
    Let us prove the statement for the $\NW$ statistic. Let us denote by
    $n_1$ the size of $\delta_1$ and by $M_1$ (resp. $M_2$) the first
    $n_1$ (resp. the last $m$) rows of $\delta$.
    \smallskip

    Notice that the zero columns of $M_2$ have no $\NW$ configuration
    and that the $\NW$ configurations lying in the nonzero columns of
    $M_1$ (resp. $M_2$) are those of $\delta_1$ (resp. $\delta_2$). It
    remains to count, for all $j \in [m]$, the number of $\NW$
    configurations in the $k_j$th column of $M_1$. Observe that the sums
    of the entries above any zero of the $k_j$th column are $0$. Besides,
    there are exactly $k_j - j$ zeros in the $k_j$th column such that the
    sums of the entries to their left are $1$. These zeros are, by
    definition, $\NW$ configurations, whence \eqref{eq::0NW::shuffle}.
    \smallskip

    This is also valid for the statistic $\SW$ since the symmetry
    consisting in swapping the $i$th and $(n - i + 1)$st row of ASMs
    of size $n$ exchanges the $\NW$ configurations into $\SW$
    configurations. By Proposition~\ref{prop::symmetrie::6V}, this also
    proves the statement of the $\SE$ and $\NE$ statistics.
\end{proof}
\medskip

Here is the product~\eqref{exemple::prod::asm} in~$\ASM$, seen on
six-vertex configurations, where boldfaced vertices are of kind~$\NW$.
\begin{equation} \label{exemple::prod::6V}
    \FF_{
    \begin{tikzpicture}[scale=.4]
        \pgfkeys{/orientations}
        \draw[gril] (.5,.5) grid (3.5,3.5);
        \draw[seta,shift={(1,0.5)}] \bi;
        \draw[seta,shift={(2,0.5)}] \bi;
        \draw[seta,shift={(3,0.5)}] \bi;
        \draw[seta,shift={(1,3.5)}] \ci;
        \draw[seta,shift={(2,3.5)}] \ci;
        \draw[seta,shift={(3,3.5)}] \ci;
        \draw[seta,shift={(3.5,1)}] \ei;
        \draw[seta,shift={(3.5,2)}] \ei;
        \draw[seta,shift={(3.5,3)}] \ei;
        \draw[seta,shift={(.5,1)}] \di;
        \draw[seta,shift={(.5,2)}] \di;
        \draw[seta,shift={(.5,3)}] \di;
        \draw[seta,shift={(1.5,1)}] \di;
        \draw[seta,shift={(2.5,1)}] \ei;
        \draw[seta,shift={(1.5,2)}] \ei;
        \draw[seta,shift={(2.5,2)}] \di;
        \draw[seta,shift={(1.5,3)}] \di;
        \draw[seta,shift={(2.5,3)}] \ei;
        \node at (3,3) {\nw};
        \draw[seta,shift={(1,1.5)}] \bi;
        \draw[seta,shift={(1,2.5)}] \bi;
        \draw[seta,shift={(2,1.5)}] \ci;
        \draw[seta,shift={(2,2.5)}] \bi;
        \draw[seta,shift={(3,1.5)}] \bi;
        \draw[seta,shift={(3,2.5)}] \ci;
    \end{tikzpicture}}
    \cdot
    \FF_{
    \begin{tikzpicture}[scale=.4]
        \pgfkeys{/orientations}
        \draw[gril] (.5,.5) grid (1.5,1.5);
        \draw[seta,shift={(1,0.5)}] \bi;
        \draw[seta,shift={(1,1.5)}] \ci;
        \draw[seta,shift={(1.5,1)}] \ei;
        \draw[seta,shift={(.5,1)}] \di;
    \end{tikzpicture}}
    =
    \FF_{
    \begin{tikzpicture}[scale=.4]
        \pgfkeys{/orientations}
        \draw[gril] (.5,.5) grid (4.5,4.5);
        \draw[seta,shift={(1,0.5)}] \bi;
        \draw[seta,shift={(2,0.5)}] \bi;
        \draw[seta,shift={(3,0.5)}] \bi;
        \draw[seta,shift={(4,0.5)}] \bi;
        \draw[seta,shift={(1,4.5)}] \ci;
        \draw[seta,shift={(2,4.5)}] \ci;
        \draw[seta,shift={(3,4.5)}] \ci;
        \draw[seta,shift={(4,4.5)}] \ci;
        \draw[seta,shift={(4.5,1)}] \ei;
        \draw[seta,shift={(4.5,2)}] \ei;
        \draw[seta,shift={(4.5,3)}] \ei;
        \draw[seta,shift={(4.5,4)}] \ei;
        \draw[seta,shift={(.5,1)}] \di;
        \draw[seta,shift={(.5,2)}] \di;
        \draw[seta,shift={(.5,3)}] \di;
        \draw[seta,shift={(.5,4)}] \di;
        \draw[seta,shift={(1.5,1)}] \di;
        \draw[seta,shift={(2.5,1)}] \di;
        \draw[seta,shift={(3.5,1)}] \di;
        \draw[seta,shift={(1.5,2)}] \di;
        \draw[seta,shift={(2.5,2)}] \ei;
        \draw[seta,shift={(3.5,2)}] \ei;
        \node at (4,2) {\nw};
        \draw[seta,shift={(1.5,3)}] \ei;
        \draw[seta,shift={(2.5,3)}] \di;
        \draw[seta,shift={(3.5,3)}] \ei;
        \node at (4,3) {\nw};
        \draw[seta,shift={(1.5,4)}] \di;
        \draw[seta,shift={(2.5,4)}] \ei;
        \draw[seta,shift={(3.5,4)}] \ei;
        \node at (4,4) {\nw};
        \node at (3,4) {\nw};
        \draw[seta,shift={(1,1.5)}] \bi;
        \draw[seta,shift={(1,2.5)}] \bi;
        \draw[seta,shift={(1,3.5)}] \ci;
        \draw[seta,shift={(2,1.5)}] \bi;
        \draw[seta,shift={(2,2.5)}] \ci;
        \draw[seta,shift={(2,3.5)}] \bi;
        \draw[seta,shift={(3,1.5)}] \bi;
        \draw[seta,shift={(3,2.5)}] \bi;
        \draw[seta,shift={(3,3.5)}] \ci;
        \draw[seta,shift={(4,1.5)}] \ci;
        \draw[seta,shift={(4,2.5)}] \ci;
        \draw[seta,shift={(4,3.5)}] \ci;
    \end{tikzpicture}}
    +
    \FF_{
    \begin{tikzpicture}[scale=.4]
        \pgfkeys{/orientations}
        \draw[gril] (.5,.5) grid (4.5,4.5);
        \draw[seta,shift={(1,0.5)}] \bi;
        \draw[seta,shift={(2,0.5)}] \bi;
        \draw[seta,shift={(3,0.5)}] \bi;
        \draw[seta,shift={(4,0.5)}] \bi;
        \draw[seta,shift={(1,4.5)}] \ci;
        \draw[seta,shift={(2,4.5)}] \ci;
        \draw[seta,shift={(3,4.5)}] \ci;
        \draw[seta,shift={(4,4.5)}] \ci;
        \draw[seta,shift={(4.5,1)}] \ei;
        \draw[seta,shift={(4.5,2)}] \ei;
        \draw[seta,shift={(4.5,3)}] \ei;
        \draw[seta,shift={(4.5,4)}] \ei;
        \draw[seta,shift={(.5,1)}] \di;
        \draw[seta,shift={(.5,2)}] \di;
        \draw[seta,shift={(.5,3)}] \di;
        \draw[seta,shift={(.5,4)}] \di;
        \draw[seta,shift={(1.5,1)}] \di;
        \draw[seta,shift={(2.5,1)}] \di;
        \draw[seta,shift={(3.5,1)}] \ei;
        \draw[seta,shift={(1.5,2)}] \di;
        \draw[seta,shift={(2.5,2)}] \ei;
        \draw[seta,shift={(3.5,2)}] \ei;
        \node at (3,2) {\nw};
        \draw[seta,shift={(1.5,3)}] \ei;
        \draw[seta,shift={(2.5,3)}] \di;
        \draw[seta,shift={(3.5,3)}] \di;
        \draw[seta,shift={(1.5,4)}] \di;
        \draw[seta,shift={(2.5,4)}] \ei;
        \draw[seta,shift={(3.5,4)}] \ei;
        \node at (4,4) {\nw};
        \node at (3,4) {\nw};
        \draw[seta,shift={(1,1.5)}] \bi;
        \draw[seta,shift={(1,2.5)}] \bi;
        \draw[seta,shift={(1,3.5)}] \ci;
        \draw[seta,shift={(2,1.5)}] \bi;
        \draw[seta,shift={(2,2.5)}] \ci;
        \draw[seta,shift={(2,3.5)}] \bi;
        \draw[seta,shift={(3,1.5)}] \ci;
        \draw[seta,shift={(3,2.5)}] \ci;
        \draw[seta,shift={(3,3.5)}] \ci;
        \draw[seta,shift={(4,1.5)}] \bi;
        \draw[seta,shift={(4,2.5)}] \bi;
        \draw[seta,shift={(4,3.5)}] \ci;
    \end{tikzpicture}}
    +
    \FF_{
    \begin{tikzpicture}[scale=.4]
        \pgfkeys{/orientations}
        \draw[gril] (.5,.5) grid (4.5,4.5);
        \draw[seta,shift={(1,0.5)}] \bi;
        \draw[seta,shift={(2,0.5)}] \bi;
        \draw[seta,shift={(3,0.5)}] \bi;
        \draw[seta,shift={(4,0.5)}] \bi;
        \draw[seta,shift={(1,4.5)}] \ci;
        \draw[seta,shift={(2,4.5)}] \ci;
        \draw[seta,shift={(3,4.5)}] \ci;
        \draw[seta,shift={(4,4.5)}] \ci;
        \draw[seta,shift={(4.5,1)}] \ei;
        \draw[seta,shift={(4.5,2)}] \ei;
        \draw[seta,shift={(4.5,3)}] \ei;
        \draw[seta,shift={(4.5,4)}] \ei;
        \draw[seta,shift={(.5,1)}] \di;
        \draw[seta,shift={(.5,2)}] \di;
        \draw[seta,shift={(.5,3)}] \di;
        \draw[seta,shift={(.5,4)}] \di;
        \draw[seta,shift={(1.5,1)}] \di;
        \draw[seta,shift={(2.5,1)}] \ei;
        \draw[seta,shift={(3.5,1)}] \ei;
        \draw[seta,shift={(1.5,2)}] \di;
        \draw[seta,shift={(2.5,2)}] \di;
        \draw[seta,shift={(3.5,2)}] \ei;
        \draw[seta,shift={(1.5,3)}] \ei;
        \draw[seta,shift={(2.5,3)}] \ei;
        \draw[seta,shift={(3.5,3)}] \di;
        \node at (2,3) {\nw};
        \draw[seta,shift={(1.5,4)}] \di;
        \draw[seta,shift={(2.5,4)}] \di;
        \draw[seta,shift={(3.5,4)}] \ei;
        \node at (4,4) {\nw};
        \draw[seta,shift={(1,1.5)}] \bi;
        \draw[seta,shift={(1,2.5)}] \bi;
        \draw[seta,shift={(1,3.5)}] \ci;
        \draw[seta,shift={(2,1.5)}] \ci;
        \draw[seta,shift={(2,2.5)}] \ci;
        \draw[seta,shift={(2,3.5)}] \ci;
        \draw[seta,shift={(3,1.5)}] \bi;
        \draw[seta,shift={(3,2.5)}] \ci;
        \draw[seta,shift={(3,3.5)}] \bi;
        \draw[seta,shift={(4,1.5)}] \bi;
        \draw[seta,shift={(4,2.5)}] \bi;
        \draw[seta,shift={(4,3.5)}] \ci;
    \end{tikzpicture}}
    +
    \FF_{
    \begin{tikzpicture}[scale=.4]
        \pgfkeys{/orientations}
        \draw[gril] (.5,.5) grid (4.5,4.5);
        \draw[seta,shift={(1,0.5)}] \bi;
        \draw[seta,shift={(2,0.5)}] \bi;
        \draw[seta,shift={(3,0.5)}] \bi;
        \draw[seta,shift={(4,0.5)}] \bi;
        \draw[seta,shift={(1,4.5)}] \ci;
        \draw[seta,shift={(2,4.5)}] \ci;
        \draw[seta,shift={(3,4.5)}] \ci;
        \draw[seta,shift={(4,4.5)}] \ci;
        \draw[seta,shift={(4.5,1)}] \ei;
        \draw[seta,shift={(4.5,2)}] \ei;
        \draw[seta,shift={(4.5,3)}] \ei;
        \draw[seta,shift={(4.5,4)}] \ei;
        \draw[seta,shift={(.5,1)}] \di;
        \draw[seta,shift={(.5,2)}] \di;
        \draw[seta,shift={(.5,3)}] \di;
        \draw[seta,shift={(.5,4)}] \di;
        \draw[seta,shift={(1.5,1)}] \ei;
        \draw[seta,shift={(2.5,1)}] \ei;
        \draw[seta,shift={(3.5,1)}] \ei;
        \draw[seta,shift={(1.5,2)}] \di;
        \draw[seta,shift={(2.5,2)}] \di;
        \draw[seta,shift={(3.5,2)}] \ei;
        \draw[seta,shift={(1.5,3)}] \di;
        \draw[seta,shift={(2.5,3)}] \ei;
        \draw[seta,shift={(3.5,3)}] \di;
        \draw[seta,shift={(1.5,4)}] \di;
        \draw[seta,shift={(2.5,4)}] \di;
        \draw[seta,shift={(3.5,4)}] \ei;
        \node at (4,4) {\nw};
        \draw[seta,shift={(1,1.5)}] \ci;
        \draw[seta,shift={(1,2.5)}] \ci;
        \draw[seta,shift={(1,3.5)}] \ci;
        \draw[seta,shift={(2,1.5)}] \bi;
        \draw[seta,shift={(2,2.5)}] \bi;
        \draw[seta,shift={(2,3.5)}] \ci;
        \draw[seta,shift={(3,1.5)}] \bi;
        \draw[seta,shift={(3,2.5)}] \ci;
        \draw[seta,shift={(3,3.5)}] \bi;
        \draw[seta,shift={(4,1.5)}] \bi;
        \draw[seta,shift={(4,2.5)}] \bi;
        \draw[seta,shift={(4,3.5)}] \bi;
    \end{tikzpicture}}
\end{equation}
\medskip

\begin{Proposition} \label{prop::morphisme::0NW}
    The map~$\phi'_s : \ASM \to \K(q)$ linearly defined, for any $s\in \ZZ$
    and any ASM~$\delta$ of size~$n$ by
    \begin{equation}\label{eq::morphisme::SE::NW}
        \phi'_s\left(\FF_{\delta}\right) := \frac{q^{s(\delta)}}{[n]_q!}
    \end{equation}
    is an algebra morphism.
\end{Proposition}
\begin{proof}
    Let us prove the statement of $\NW$ statistic; the three other cases
    are analogous. Let $\delta_1$ and $\delta_2$ be two ASMs of
    respective sizes $n_1$ and $n_2$. Lemma~\ref{lem:nw} implies
    \begin{equation}\begin{split}
        \phi'_{\NW}(\FF_{\delta_1} \cdot \FF_{\delta_2})
        & =
            \frac{q^{\NW(\delta_1) + \NW(\delta_2)}}{[n_1+n_2]_q!}
            \sum_{\{k_1, \ldots, k_{n_2}\} \subset \{1, \ldots, n_1+n_2\}}
            q^{(k_1-1) + \cdots + (k_{n_2}-n_2)} \\[.5em]
        & = \displaystyle
            \frac{q^{\NW(\delta_1) + \NW(\delta_2)}}{[n_1+n_2]_q!}
                \genfrac{[}{]}{0pt}{1}{n_1+n_2}{n_1, n_2}_q \\[.5em]
        & = \phi'_{\NW}(\FF_{\delta_1}) \cdot \phi'_{\NW}(\FF_{\delta_2}).
        \qedhere
    \end{split}\end{equation}
\end{proof}
\medskip

By similar arguments, all previous results remain valid in the dual
$\ASM^\star$ of $\ASM$. Hence,
\begin{Proposition} \label{prop::morphisme::0NW:2}
    The maps $\psi_s : \ASM^\star \to \K(q)$ and
    $\psi'_t : \ASM^\star \to \K(q)$ linearly defined, for any $s\in\NN$,
    $t\in\ZZ$, and any ASM~$\delta$ of size~$n$ by
    \begin{equation}
    \label{eq::morphisme::SE::NW::star}
    \psi_{s}\left(\FF^\star_{\delta}\right) := \frac{q^{s(\delta)}}{n!}
    \qquad\qquad\mbox{and}\qquad\qquad
        \psi'_t\left(\FF^\star_{\delta}\right)
        := \frac{q^{t(\delta)}}{[n]_q!}
    \end{equation}
    are algebra morphisms.
\end{Proposition}
\medskip

Here is the product~\eqref{exemple::prod::asm} in~$\ASM^\star$, seen on
six-vertex configurations, where the vertices represented by squares are
of kind~$\IO$ while those represented by circles are of kind~$\NW$.
\begin{equation}
    \FF^\star_{
    \begin{tikzpicture}[scale=.4]
        \pgfkeys{/orientations}
        \draw[gril] (.5,.5) grid (3.5,3.5);
        \draw[seta,shift={(1,0.5)}] \bi;
        \draw[seta,shift={(2,0.5)}] \bi;
        \draw[seta,shift={(3,0.5)}] \bi;
        \draw[seta,shift={(1,3.5)}] \ci;
        \draw[seta,shift={(2,3.5)}] \ci;
        \draw[seta,shift={(3,3.5)}] \ci;
        \draw[seta,shift={(3.5,1)}] \ei;
        \draw[seta,shift={(3.5,2)}] \ei;
        \draw[seta,shift={(3.5,3)}] \ei;
        \draw[seta,shift={(.5,1)}] \di;
        \draw[seta,shift={(.5,2)}] \di;
        \draw[seta,shift={(.5,3)}] \di;
        \draw[seta,shift={(1.5,1)}] \di;
        \draw[seta,shift={(2.5,1)}] \ei;
        \draw[seta,shift={(1.5,2)}] \ei;
        \draw[seta,shift={(2.5,2)}] \di;
        \draw[seta,shift={(1.5,3)}] \di;
        \draw[seta,shift={(2.5,3)}] \ei;
        \node at (3,3) {\nw};\node at (2,2) {\io};
        \draw[seta,shift={(1,1.5)}] \bi;
        \draw[seta,shift={(1,2.5)}] \bi;
        \draw[seta,shift={(2,1.5)}] \ci;
        \draw[seta,shift={(2,2.5)}] \bi;
        \draw[seta,shift={(3,1.5)}] \bi;
        \draw[seta,shift={(3,2.5)}] \ci;
    \end{tikzpicture}}
    \cdot
    \FF^\star_{
    \begin{tikzpicture}[scale=.4]
        \pgfkeys{/orientations}
        \draw[gril] (.5,.5) grid (1.5,1.5);
        \draw[seta,shift={(1,0.5)}] \bi;
        \draw[seta,shift={(1,1.5)}] \ci;
        \draw[seta,shift={(1.5,1)}] \ei;
        \draw[seta,shift={(.5,1)}] \di;
    \end{tikzpicture}}
    =
    \FF^\star_{
    \begin{tikzpicture}[scale=.4]
        \pgfkeys{/orientations}
        \draw[gril] (.5,.5) grid (4.5,4.5);
        \draw[seta,shift={(1,0.5)}] \bi;
        \draw[seta,shift={(2,0.5)}] \bi;
        \draw[seta,shift={(3,0.5)}] \bi;
        \draw[seta,shift={(4,0.5)}] \bi;
        \draw[seta,shift={(1,4.5)}] \ci;
        \draw[seta,shift={(2,4.5)}] \ci;
        \draw[seta,shift={(3,4.5)}] \ci;
        \draw[seta,shift={(4,4.5)}] \ci;
        \draw[seta,shift={(4.5,1)}] \ei;
        \draw[seta,shift={(4.5,2)}] \ei;
        \draw[seta,shift={(4.5,3)}] \ei;
        \draw[seta,shift={(4.5,4)}] \ei;
        \draw[seta,shift={(.5,1)}] \di;
        \draw[seta,shift={(.5,2)}] \di;
        \draw[seta,shift={(.5,3)}] \di;
        \draw[seta,shift={(.5,4)}] \di;
        \draw[seta,shift={(1.5,1)}] \di;
        \draw[seta,shift={(2.5,1)}] \di;
        \draw[seta,shift={(3.5,1)}] \di;
        \draw[seta,shift={(1.5,2)}] \di;
        \draw[seta,shift={(2.5,2)}] \ei;
        \draw[seta,shift={(3.5,2)}] \ei;
        \node at (4,2) {\nw};
        \draw[seta,shift={(1.5,3)}] \ei;
        \draw[seta,shift={(2.5,3)}] \di;
        \draw[seta,shift={(3.5,3)}] \ei;
        \node at (4,3) {\nw};
        \draw[seta,shift={(1.5,4)}] \di;
        \draw[seta,shift={(2.5,4)}] \ei;
        \draw[seta,shift={(3.5,4)}] \ei;
        \node at (4,4) {\nw};
        \node at (3,4) {\nw};\node at (2,3) {\io};
        \draw[seta,shift={(1,1.5)}] \bi;
        \draw[seta,shift={(1,2.5)}] \bi;
        \draw[seta,shift={(1,3.5)}] \ci;
        \draw[seta,shift={(2,1.5)}] \bi;
        \draw[seta,shift={(2,2.5)}] \ci;
        \draw[seta,shift={(2,3.5)}] \bi;
        \draw[seta,shift={(3,1.5)}] \bi;
        \draw[seta,shift={(3,2.5)}] \bi;
        \draw[seta,shift={(3,3.5)}] \ci;
        \draw[seta,shift={(4,1.5)}] \ci;
        \draw[seta,shift={(4,2.5)}] \ci;
        \draw[seta,shift={(4,3.5)}] \ci;
    \end{tikzpicture}}
    +
    \FF^\star_{
    \begin{tikzpicture}[scale=.4]
        \pgfkeys{/orientations}
        \draw[gril] (.5,.5) grid (4.5,4.5);
        \draw[seta,shift={(1,0.5)}] \bi;
        \draw[seta,shift={(2,0.5)}] \bi;
        \draw[seta,shift={(3,0.5)}] \bi;
        \draw[seta,shift={(4,0.5)}] \bi;
        \draw[seta,shift={(1,4.5)}] \ci;
        \draw[seta,shift={(2,4.5)}] \ci;
        \draw[seta,shift={(3,4.5)}] \ci;
        \draw[seta,shift={(4,4.5)}] \ci;
        \draw[seta,shift={(4.5,1)}] \ei;
        \draw[seta,shift={(4.5,2)}] \ei;
        \draw[seta,shift={(4.5,3)}] \ei;
        \draw[seta,shift={(4.5,4)}] \ei;
        \draw[seta,shift={(.5,1)}] \di;
        \draw[seta,shift={(.5,2)}] \di;
        \draw[seta,shift={(.5,3)}] \di;
        \draw[seta,shift={(.5,4)}] \di;
        \draw[seta,shift={(1.5,1)}] \di;
        \draw[seta,shift={(2.5,1)}] \ei;
        \draw[seta,shift={(3.5,1)}] \ei;
        \draw[seta,shift={(1.5,2)}] \di;
        \draw[seta,shift={(2.5,2)}] \di;
        \draw[seta,shift={(3.5,2)}] \di;
        \draw[seta,shift={(1.5,3)}] \ei;
        \draw[seta,shift={(2.5,3)}] \di;
        \draw[seta,shift={(3.5,3)}] \ei;
        \node at (4,3) {\nw};\node at (2,3) {\io};
        \draw[seta,shift={(1.5,4)}] \di;
        \draw[seta,shift={(2.5,4)}] \ei;
        \draw[seta,shift={(3.5,4)}] \ei;
        \node at (4,4) {\nw};
        \node at (3,4) {\nw};
        \draw[seta,shift={(1,1.5)}] \bi;
        \draw[seta,shift={(1,2.5)}] \bi;
        \draw[seta,shift={(1,3.5)}] \ci;
        \draw[seta,shift={(2,1.5)}] \ci;
        \draw[seta,shift={(2,2.5)}] \ci;
        \draw[seta,shift={(2,3.5)}] \bi;
        \draw[seta,shift={(3,1.5)}] \bi;
        \draw[seta,shift={(3,2.5)}] \bi;
        \draw[seta,shift={(3,3.5)}] \ci;
        \draw[seta,shift={(4,1.5)}] \bi;
        \draw[seta,shift={(4,2.5)}] \ci;
        \draw[seta,shift={(4,3.5)}] \ci;
    \end{tikzpicture}}
    +
    \FF^\star_{
    \begin{tikzpicture}[scale=.4]
        \pgfkeys{/orientations}
        \draw[gril] (.5,.5) grid (4.5,4.5);
        \draw[seta,shift={(1,0.5)}] \bi;
        \draw[seta,shift={(2,0.5)}] \bi;
        \draw[seta,shift={(3,0.5)}] \bi;
        \draw[seta,shift={(4,0.5)}] \bi;
        \draw[seta,shift={(1,4.5)}] \ci;
        \draw[seta,shift={(2,4.5)}] \ci;
        \draw[seta,shift={(3,4.5)}] \ci;
        \draw[seta,shift={(4,4.5)}] \ci;
        \draw[seta,shift={(4.5,1)}] \ei;
        \draw[seta,shift={(4.5,2)}] \ei;
        \draw[seta,shift={(4.5,3)}] \ei;
        \draw[seta,shift={(4.5,4)}] \ei;
        \draw[seta,shift={(.5,1)}] \di;
        \draw[seta,shift={(.5,2)}] \di;
        \draw[seta,shift={(.5,3)}] \di;
        \draw[seta,shift={(.5,4)}] \di;
        \draw[seta,shift={(1.5,1)}] \di;
        \draw[seta,shift={(2.5,1)}] \ei;
        \draw[seta,shift={(3.5,1)}] \ei;
        \draw[seta,shift={(1.5,2)}] \ei;
        \draw[seta,shift={(2.5,2)}] \di;
        \draw[seta,shift={(3.5,2)}] \ei;
        \draw[seta,shift={(1.5,3)}] \di;
        \draw[seta,shift={(2.5,3)}] \di;
        \draw[seta,shift={(3.5,3)}] \di;
        \draw[seta,shift={(1.5,4)}] \di;
        \draw[seta,shift={(2.5,4)}] \di;
        \draw[seta,shift={(3.5,4)}] \ei;
        \node at (4,4) {\nw};\node at (2,2) {\io};
        \node at (3,4) {\nw};
        \draw[seta,shift={(1,1.5)}] \bi;
        \draw[seta,shift={(1,2.5)}] \ci;
        \draw[seta,shift={(1,3.5)}] \ci;
        \draw[seta,shift={(2,1.5)}] \ci;
        \draw[seta,shift={(2,2.5)}] \bi;
        \draw[seta,shift={(2,3.5)}] \bi;
        \draw[seta,shift={(3,1.5)}] \bi;
        \draw[seta,shift={(3,2.5)}] \ci;
        \draw[seta,shift={(3,3.5)}] \ci;
        \draw[seta,shift={(4,1.5)}] \bi;
        \draw[seta,shift={(4,2.5)}] \bi;
        \draw[seta,shift={(4,3.5)}] \ci;
    \end{tikzpicture}}
    +
    \FF^\star_{
    \begin{tikzpicture}[scale=.4]
        \pgfkeys{/orientations}
        \draw[gril] (.5,.5) grid (4.5,4.5);
        \draw[seta,shift={(1,0.5)}] \bi;
        \draw[seta,shift={(2,0.5)}] \bi;
        \draw[seta,shift={(3,0.5)}] \bi;
        \draw[seta,shift={(4,0.5)}] \bi;
        \draw[seta,shift={(1,4.5)}] \ci;
        \draw[seta,shift={(2,4.5)}] \ci;
        \draw[seta,shift={(3,4.5)}] \ci;
        \draw[seta,shift={(4,4.5)}] \ci;
        \draw[seta,shift={(4.5,1)}] \ei;
        \draw[seta,shift={(4.5,2)}] \ei;
        \draw[seta,shift={(4.5,3)}] \ei;
        \draw[seta,shift={(4.5,4)}] \ei;
        \draw[seta,shift={(.5,1)}] \di;
        \draw[seta,shift={(.5,2)}] \di;
        \draw[seta,shift={(.5,3)}] \di;
        \draw[seta,shift={(.5,4)}] \di;
        \draw[seta,shift={(1.5,1)}] \di;
        \draw[seta,shift={(2.5,1)}] \ei;
        \draw[seta,shift={(3.5,1)}] \ei;
        \draw[seta,shift={(1.5,2)}] \ei;
        \draw[seta,shift={(2.5,2)}] \di;
        \draw[seta,shift={(3.5,2)}] \ei;
        \draw[seta,shift={(1.5,3)}] \di;
        \draw[seta,shift={(2.5,3)}] \ei;
        \draw[seta,shift={(3.5,3)}] \ei;
        \node at (3,3) {\nw};\node at (2,2) {\io};
        \draw[seta,shift={(1.5,4)}] \di;
        \draw[seta,shift={(2.5,4)}] \di;
        \draw[seta,shift={(3.5,4)}] \di;
        \draw[seta,shift={(1,1.5)}] \bi;
        \draw[seta,shift={(1,2.5)}] \ci;
        \draw[seta,shift={(1,3.5)}] \ci;
        \draw[seta,shift={(2,1.5)}] \ci;
        \draw[seta,shift={(2,2.5)}] \bi;
        \draw[seta,shift={(2,3.5)}] \ci;
        \draw[seta,shift={(3,1.5)}] \bi;
        \draw[seta,shift={(3,2.5)}] \ci;
        \draw[seta,shift={(3,3.5)}] \ci;
        \draw[seta,shift={(4,1.5)}] \bi;
        \draw[seta,shift={(4,2.5)}] \bi;
        \draw[seta,shift={(4,3.5)}] \bi;
    \end{tikzpicture}}\,.
\end{equation}
\medskip

\subsubsection{Equivalence relations on ASMs and associated subspaces
of $\ASM$}
Let~$S\subseteq \ZZ \cup \NN$ be a set of statistics and $\sim_S$ be the
equivalence relation on the set of ASMs defined, for any ASMs $\delta_1$
and $\delta_2$ of the same size, by
\begin{equation}
    \delta_1 \sim_S \delta_2
    \quad \mbox{if and only if} \quad
    s\left(\delta_1\right) = s\left(\delta_2\right)
    \mbox{ for all } s \in S.
\end{equation}
We denote by~$\ideal{S}$ the associated vector space spanned by
\begin{equation}
    \left\{\FF_{\delta_1}-\FF_{\delta_2}, \; \delta_1 \sim_S \delta_2\right\}.
\end{equation}
\medskip

\subsubsection{The algebra $\ASM/_{I_\IO}$}
Let us first study the statistic $\IO \in \NN$.
\medskip

\begin{Proposition} \label{prop:quotient:io}
    The quotient~$\QASM{\ideal{\IO}}$ is a commutative algebra.
\end{Proposition}
\begin{proof}
    The subspace $\ideal{\IO}$ of $\ASM$ is a two-sided ideal of $\ASM$.
    Indeed, let $\delta$, $\delta_1$, and $\delta_2$ be three ASMs such
    that $\delta_1 \sim_{\IO} \delta_2$. Since the products
    $\FF_{\delta}\cdot\FF_{\delta_i}$ and $\FF_{\delta_i}\cdot\FF_{\delta}$
    for $i \in \{1, 2\}$ have the same number of terms, Lemma~\ref{lem:io}
    implies that the products
    $\FF_{\delta} \cdot (\FF_{\delta_1} - \FF_{\delta_2})$ and
    $(\FF_{\delta_1} - \FF_{\delta_2}) \cdot \FF_{\delta}$ are
    in $\ideal{\IO}$. Hence, $\QASM{\ideal{\IO}}$ is an algebra.
    \smallskip

    Besides, the ideal $\ideal{\IO}$ contains the commutators. Indeed,
    let $\delta_1$ and $\delta_2$ be two ASMs. Since the products
    $\FF_{\delta_1} \cdot \FF_{\delta_2}$ and
    $\FF_{\delta_2} \cdot \FF_{\delta_1}$ have the same number of terms,
    Lemma~\ref{lem:io} implies that
    $\FF_{\delta_1}\cdot \FF_{\delta_2} - \FF_{\delta_2} \cdot \FF_{\delta_1}$
    is in $\ideal{\IO}$. Thus, $\QASM{\ideal{\IO}}$ is commutative as
    an algebra.
\end{proof}
\medskip

Note however that $\QASM{\ideal{\IO}}$ does not inherit the structure of
a coalgebra of $\ASM$ because even if
\begin{equation}
    x := \FF_{\begin{Matrice}
      0 & \Plus & 0 & 0 \\
      \Plus & \Moins & \Plus & 0 \\
      0 & \Plus & 0 & 0 \\
      0 & 0 & 0 & \Plus
     \end{Matrice}} -
     \FF_{\begin{Matrice}
      0 & \Plus & 0 & 0 \\
      \Plus & \Moins & 0 & \Plus \\
      0 & \Plus & 0 & 0 \\
      0 & 0 & \Plus & 0
     \end{Matrice}}
\end{equation}
is an element of $\ideal{\IO}$, the element
\begin{equation}
    \Delta(x) =
        1 \otimes x +
        \FF_{\begin{Matrice}
            0 & \Plus & 0 \\
            \Plus & \Moins & \Plus \\
            0 & \Plus & 0
        \end{Matrice}} \otimes
        \FF_{\begin{Matrice}
            \Plus
        \end{Matrice}} +
        x \otimes 1
\end{equation}
is not in $\ASM \otimes \ideal{\IO} + \ideal{\IO} \otimes \ASM$. Hence,
$\ideal{\IO}$ is not a coideal.
\medskip

\begin{Proposition}
    The dimension $A_n^{\IO}$ of the $n$th graded component
    of~$\QASM{\ideal{\IO}}$ is~$\left\lfloor\frac{n^2}{4}\right\rfloor+1$.
\end{Proposition}
\begin{proof}
    Let $\delta$ be an ASM of size $n$ with a maximal number of $\IO$
    configurations ({\em i.e.}, a maximal number of $\Moins$). Then, it
    is easy to see that
    \begin{equation}
        \IO(\delta) =
        \sum\limits_{i = 0}^{\left\lfloor\frac{n}{2}\right\rfloor - 1} i +
        \sum\limits_{i = 1}^{\left\lfloor\frac{n - 1}{2}\right\rfloor} i.
    \end{equation}
    Indeed, the first and last row of an ASM can contain only one $\Plus$
    and no $\Moins$. Let $i \geq 2$ and let $A_{i-1}$ be the matrix consisting
    in the first $i-1$ rows
    of $A$. The $\Moins$\,s in row $i$ can only be in those columns for
    which the corresponding column sum of the submatrix $A_{i-1}$ is
    $1$. Since the row sums of $A_{i-1}$ are $1$ and the column sums
    of $A_{i-1}$ are $0$ or $1$, exactly $i-1$ of the column sums of
    $A_{i-1}$ are $1$. We conclude that there are at most $(i-1)$
    $\Moins$\,s in row $i$. The same argument applies to the column sums
    taken from bottom to top. Hence, the rows $i$ and $n-i+1$ are at most
    $(i-1)$ $\Moins$\,s. If $n$ is odd, then the row $(n+1)/2$ has only
    nonzero entries, alternating between $\Plus$ and $\Moins$, and the
    row $(n+1)/2$ has $\left\lfloor\frac{n}{2}\right\rfloor$ $\Moins$\,s.
    \smallskip

    Now, since for any $0 \leq k \leq \IO(\delta)$, there exists an
    ASM $\delta'$ such that $\IO(\delta') = k$, we obtain, by a simple
    computation, the statement of the proposition.
\end{proof}
\medskip

The dimensions of~$\QASM{\ideal{\IO}}$ form Sequence~\Sloane{A033638}
of~\cite{Slo} and the first few terms are
\begin{equation}
    1, \: 1, \: 1, \: 2, \: 3, \: 5, \: 7, \: 10, \: 13, \: 17, \: 21.
\end{equation}
\medskip

A basic argument on generating series implies that these dimensions cannot
be the ones of a free commutative algebra and hence, $\QASM{\ideal{\IO}}$
is not free as a commutative algebra.
\medskip

Using the symmetry between the statistics $\IO$ and $\OI$ provided by
Proposition~\ref{prop::symmetrie::6V}, we immediately have
$\sim_{\OI} = \sim_{\IO}$ and then, $\QASM{\ideal{\OI}} = \QASM{\ideal{\IO}}$.
\bigskip

\subsubsection{The algebra $\ASM/_{I_\NW}$}
Let us now study the statistic $\NW \in \ZZ$.
\medskip

\begin{Proposition} \label{prop:quotient:nw}
    The quotient~$\QASM{\ideal{\NW}}$ is a commutative algebra.
\end{Proposition}
\begin{proof}
    The subspace $\ideal{\NW}$ of $\ASM$ is a two-sided ideal of $\ASM$.
    Indeed, let $\delta$, $\delta_1$ and $\delta_2$ be three ASMs of
    respective sizes $n$, $n_1$ and $n_2$ such that
    $\delta_1 \sim_{\NW} \delta_2$. Lemma~\ref{lem:nw} implies that the
    number of~$\NW$ configurations of an ASM~$\delta'$ such that
    $\FF_{\delta'}$ appears in $\FF_{\delta} \cdot \FF_{\delta_1}$
    (resp. $\FF_{\delta} \cdot \FF_{\delta_2}$) depends only on the number
    of $\NW$ configurations in $\delta$ and $\delta_1$ (resp. $\delta_2$)
    and a subset of $[n+n_1]$ (resp. $[n+n_2]$)
    of size $n_1$ (resp. $n_2$) corresponding to the positions in $\delta'$
    of the columns coming from $\delta_1$ (resp. $\delta_2$).
    Since $\delta_1 \sim_{\NW} \delta_2$, the product
    $\FF_{\delta} \cdot (\FF_{\delta_1} - \FF_{\delta_2})$ is
    in $\ideal{\NW}$. Similarly
    $(\FF_{\delta_1} - \FF_{\delta_2}) \cdot \FF_{\delta}$ also is in
    $\ideal{\NW}$. Hence, $\QASM{\ideal{\NW}}$ is an algebra.
    \smallskip

    The ideal $\ideal{\NW}$ contains the commutators. Indeed, let
    $\delta_1$ and $\delta_2$ be two ASMs of respective sizes $n_1$ and
    $n_2$. The symmetry of $q$-binomial coefficients implies that
    there are as many subsets $S_{1,2}$ of $[n_1+n_2]$ of size $n_2$
    as subset $S_{2,1}$ of $[n_1+n_2]$ of size $n_1$ such that the
    sum of elements of $S_{1,2}$ is equal to the sum of elements of $S_{2,1}$.
    Lemma~\ref{lem:io} implies that
    $\FF_{\delta_1}\cdot \FF_{\delta_2} - \FF_{\delta_2} \cdot \FF_{\delta_1}$
    is in $\ideal{\NW}$. Thus, $\QASM{\ideal{\NW}}$ is commutative
    as an algebra.
\end{proof}
\medskip

Note however that $\QASM{\ideal{\NW}}$ does not inherit the structure of
a coalgebra of $\ASM$ because even if
\begin{equation}
     x := \FF_{\begin{Matrice}
          0 & 0 & 0 & \Plus\\
          \Plus & 0 & 0 & 0\\
          0 & 0 & \Plus & 0\\
          0 & \Plus & 0 & 0
     \end{Matrice}} -
     \FF_{\begin{Matrice}
          0 & 0 & \Plus & 0\\
          0 & \Plus & 0 & 0\\
          \Plus & 0 & \Moins & \Plus\\
          0 & 0 & \Plus & 0
     \end{Matrice}}
\end{equation}
is an element of $\ideal{\NW}$, the element
\begin{equation}
    \Delta(x) =
        1 \otimes x +
        \FF_{\begin{Matrice}
          \Plus
        \end{Matrice}} \otimes
        \FF_{\begin{Matrice}
            0 & 0 & \Plus\\
            0 & \Plus & 0\\
            \Plus & 0 & 0
        \end{Matrice}} +
        \FF_{\begin{Matrice}
            \Plus & 0\\
            0 & \Plus
        \end{Matrice}} \otimes
        \FF_{\begin{Matrice}
            0 & \Plus\\
            \Plus & 0
        \end{Matrice}} +
        \FF_{\begin{Matrice}
            \Plus & 0 & 0\\
            0 & 0 & \Plus\\
            0 & \Plus & 0
        \end{Matrice}} \otimes
        \FF_{\begin{Matrice}
            \Plus
        \end{Matrice}} +
        x \otimes 1
\end{equation}
is not in $\ASM \otimes \ideal{\NW} + \ideal{\NW} \otimes \ASM$. Hence,
$\ideal{\NW}$ is not a coideal.
\medskip

\begin{Proposition}
    The dimension $A_n^{\NW}$ of the $n$th graded component of
    $\QASM{\ideal{\NW}}$ is $\binom{n}{2}+1$.
\end{Proposition}
\begin{proof}
    Let us first show that there are at least $\binom{n}{2}+1$
    $\sim_{\NW}$-equivalence classes of ASMs of size $n$ by considering
    a process that associates with a permutation matrix $M_1$ of size $n$
    a permutation matrix $M_2$ such that $\NW(M_2)=\NW(M_1)+1$. If $M_1$
    is not the permutation matrix $\identity_n$ of the identity, there is
    a greatest integer $k \geq 0$ such that $M_1 = \identity_k \Over M'_1$
    and $M'_1$ is not empty. Consider now the matrix
    $M_2 := \identity_k \Over M'_2$ where $M'_2$ is the matrix obtained
    by swapping the $(i-1)$st and the $i$th columns of $M'_1$ so that $i$
    is the index of the column of $M'_1$ containing its uppermost $1$.
    Starting with the matrix $M_1$ of size $n$ of the form
    $1 \Under \cdots \Under 1$, we can iteratively apply the previous
    process $\binom{n}{2}$ times. Since each iteration
    obviously increases by one the number of $\NW$ configurations, all
    the $\binom{n}{2}+1$ permutation matrices are in
    different $\sim_{\NW}$-equivalence classes.
    \smallskip

    Let us then show that there are no more than
    $\binom{n}{2}+1$ $\sim_{\NW}$-equivalence classes of
    ASMs of size~$n$. Each entry of an ASM $\delta$ of size~$n$ gives
    rise to a configuration among the six possible. Then,
    \begin{equation} \label{eq::configurations}
        n^2 = \NW(\delta) + \NE(\delta) + \SW(\delta) + \NW(\delta) +
        \IO(\delta) + \OI(\delta).
    \end{equation}
    By using the symmetries provided by Proposition~\ref{prop::symmetrie::6V},
    \eqref{eq::configurations} becomes
    \begin{equation}
        n^2 = 2\,\SW(\delta) + 2\,\NW(\delta) + 2\,\IO(\delta) + n
    \end{equation}
    and we deduce that $\NW(\delta) \leq \frac{n^2-n}{2} = \binom{n}{2}$.
\end{proof}
\medskip

The dimensions of~$\QASM{\ideal{\NW}}$ form Sequence~\Sloane{A152947}
of~\cite{Slo} and the first few terms are
\begin{equation}
    1, \: 1, \: 2, \: 4, \: 7, \: 11, \: 16, \: 22, \: 29, \: 37, \: 46,
    \: 56.
\end{equation}
\medskip

A basic argument on generating series implies that these dimensions cannot
be the ones of a free commutative algebra and hence, $\QASM{\ideal{\NW}}$
is not free as a commutative algebra.
\medskip

Using the symmetry between the statistics $\NW$ and $\SE$ provided by
Proposition~\ref{prop::symmetrie::6V}, we immediately have
$\sim_{\SE} = \sim_{\NW}$ and then, $\QASM{\ideal{\SE}} = \QASM{\ideal{\NW}}$.
Moreover, by using the same arguments as before, $\QASM{\ideal{\SW}}$
and $\QASM{\ideal{\NE}}$ are the same commutative algebras.
\medskip

Note that the map $\theta : \QASM{\ideal{\NW}} \to \QASM{\ideal{\SW}}$
linearly defined for any ASM $\delta$ by
\begin{equation}
    \theta(\pi_{\NW}(\FF_{\delta})) :=
    \pi_{\SW}\left(\FF_{\overleftarrow{\delta}}\right),
\end{equation}
where $\pi_{\NW}$ (resp. $\pi_{\SW}$) is the canonical projection from $\ASM$
to $\QASM{\ideal{\NW}}$ (resp. $\QASM{\ideal{\SW}}$) and $\overleftarrow{\delta}$
is the ASM where, for any $i \in [n]$, the $i$th column of
$\overleftarrow{\delta}$ is the $(n - i + 1)$st column of $\delta$,
is an isomorphism between $\QASM{\ideal{\NW}}$ and $\QASM{\ideal{\SW}}$.
\bigskip

\subsubsection{The algebra $\ASM/_{I_{\IO, \NW}}$}
Let us finally study the set of statistics $\{\IO, \NW\}$.
\medskip

\begin{Proposition} \label{prop:quotient:nw_io}
    The quotient~$\QASM{\ideal{\IO,\NW}}$ is a commutative algebra.
\end{Proposition}
\begin{proof}
    This follows directly from Propositions~\ref{prop:quotient:io}
    and~\ref{prop:quotient:nw}.
\end{proof}
\medskip

Note however that $\QASM{\ideal{\IO,\NW}}$ does not
inherit the structure of a coalgebra of $\ASM$ because even if
\begin{equation}
     x :=
     \FF_{\begin{Matrice}
          0 & \Plus & 0 & 0\\
          \Plus & \Moins & \Plus & 0\\
          0 & 0 & 0 & \Plus\\
          0 & \Plus & 0 & 0
     \end{Matrice}} -
     \FF_{\begin{Matrice}
          0 & \Plus & 0 & 0\\
          0 & 0 & \Plus & 0\\
          \Plus & \Moins & 0 & \Plus\\
          0 & \Plus & 0 & 0
     \end{Matrice}}
\end{equation}
is an element of $\ideal{\IO,\NW}$, the element
\begin{equation}
    \Delta(x) =
        1 \otimes x +
        \FF_{\begin{Matrice}
            0 & \Plus & 0 \\
            \Plus & \Moins & \Plus \\
            0 & \Plus & 0
        \end{Matrice}} \otimes
        \FF_{\begin{Matrice}
            \Plus
        \end{Matrice}} +
        x \otimes 1
\end{equation}
is not in $\ASM \otimes \ideal{\IO,\NW} + \ideal{\IO,\NW} \otimes \ASM$.
Hence, $\ideal{\IO,\NW}$ is not a coideal.
\medskip

By computer exploration, the first few dimensions of~$\QASM{\ideal{\IO,\NW}}$
are
\begin{equation}
    1, \: 1, \: 2, \: 5, \: 13, \: 31, \: 66, \: 127, \: 225,
\end{equation}
and seems to be Sequence~\Sloane{A116701} of~\cite{Slo}.
\medskip

A basic argument on generating series implies that these dimensions cannot
be the ones of a free commutative algebra and hence, $\QASM{\ideal{\IO,\NW}}$
is not free as a commutative algebra.
\bigskip

\subsubsection{Others quotients of $\ASM$}
Using the symmetries provided by Proposition~\ref{prop::symmetrie::6V},
all the algebras $\QASM{\ideal{S}}$, where $S$ contains two nonsymmetric
statistics, are equal to~$\QASM{\ideal{\IO,\NW}}$. Moreover, note that
by using the same arguments as before, one can prove that for any
$S \in \ZZ \cup \NN$, $\QASM{\ideal{S}}$ is a commutative algebra
isomorphic to $\QASM{\ideal{\IO}}$, $\QASM{\ideal{\NW}}$, or
$\QASM{\ideal{\IO,\NW}}$.


\bibliographystyle{alpha}
\bibliography{Bibliographie}

\end{document}